\documentclass[11pt,a4paper,leqno]{amsart}
\usepackage{version}
\usepackage[utf8]{inputenc}
\usepackage{amsmath,amssymb,amsfonts, amsthm, yfonts}
\usepackage{calrsfs}
\usepackage{esint}
\usepackage{hyperref}
\usepackage{mathtools}
\usepackage{bbm}
\usepackage{enumitem}

\newcommand*{\dd}{\mathop{}\!\mathrm{d}}
\newcommand*{\QQ}{\pazocal{Q}}
\newcommand*{\PP}{\pazocal{P}}
\newcommand*{\FF}{\pazocal{F}}
\newcommand*{\DD}{\pazocal{D}}
\newcommand*{\pSS}{\pazocal{S}}
\newcommand*{\ox}{\overline{x}}
\newcommand*{\oy}{\overline{y}}
\newcommand*{\oz}{\overline{z}}
\newcommand*{\ow}{\overline{w}}
\newcommand*{\Aone}{\text{\normalfont{\textbf{A}}}_{\boldsymbol{1}}}
\newcommand*{\Atwo}{\text{\normalfont{\textbf{A}}}_{\boldsymbol{2}}}
\newcommand*{\Athree}{\text{\normalfont{\textbf{A}}}_{\boldsymbol{3}}}
\newcommand*{\Pone}{\text{\normalfont{\textbf{P}}}_{\boldsymbol{1}}}
\newcommand*{\Ptwo}{\text{\normalfont{\textbf{P}}}_{\boldsymbol{2}}}
\newcommand*{\Pthree}{\text{\normalfont{\textbf{P}}}_{\boldsymbol{3}}}
\newcommand*{\Pfour}{\text{\normalfont{\textbf{P}}}_{\boldsymbol{4}}}
\newcommand*{\Pfive}{\text{\normalfont{\textbf{P}}}_{\boldsymbol{5}}}

\newcommand*{\Psy}{P_{\text{\normalfont{sy}}}}
\makeatletter
\newcommand*{\rom}[1]{\expandafter\@slowromancap\romannumeral #1@}
\makeatother
\usepackage{amsmath,amsthm,amssymb,latexsym,epsfig,graphicx,subfigure}
\newcommand*\mathinhead[2]{\texorpdfstring{$#1$}{#2}}

\hypersetup{
	colorlinks=true,
	linkcolor=magenta,
	filecolor=black,      
	urlcolor=blue,
	citecolor=cyan,
}

\numberwithin{equation}{section}

\theoremstyle{plain}
\newtheorem*{thm*}{Theorem}
\newtheorem{thm}{Theorem}[section]
\newtheorem{lem}[thm]{Lemma}

\newtheorem{cor}[thm]{Corollary}
\theoremstyle{definition}
\newtheorem{defn}{Definition}[section]

\theoremstyle{remark}
\newtheorem{rem}{\textbf{Remark}}[section]

\DeclareMathAlphabet{\pazocal}{OMS}{zplm}{m}{n}

\begin{document}
	\title[On the semi-additivity of the $1/2$-symmetric caloric capacity]{On the semi-additivity of the $1/2$-symmetric caloric capacity}
	\author{Joan Hernández, Joan Mateu, and Laura Prat}\thanks{J. H. has been supported by PID2020-114167GB-I00 (Mineco, Spain). J. M. has been partially supported by 2021SGR-00071 (Catalonia) and PID2020-112881GB-I00 (Mineco, Spain). L.P. has been partially supported by 2021SGR-00071 (Catalonia) and PID2020-114167GB-I00 (Mineco, Spain)}
	\maketitle
	
	\begin{abstract}
		In this paper we study properties of a variant of the $1/2$-caloric capacity, called $1/2$-symmetric caloric capacity. The latter is associated simultaneously with the $1/2$-fractional heat equation and its conjugate. We establish its semi-additivity in $\mathbb{R}^{n+1}$ and, moreover, we compute explicitly the $1/2$-symmetric caloric capacity of rectangles, which illustrates its anisotropic behavior.
		
		
	\end{abstract}
	
	\section{Introduction}
	\label{sec1}
	\bigskip
	
	The study of removable subsets for caloric equations done by Mateu, Prat \& Tolsa in \cite{MPrTo} and by Mateu \& Prat in \cite{MPr} introduces, naturally, caloric capacities associated to such PDE's. In the first reference $(1,1/2)$-Lipschitz caloric capacity was introduced, as well as the notion of equivalence between the nullity of this quantity and the removability of compact subsets for the heat equation, i.e. the one associated with the differential operator $\Theta:=(-\Delta_x)+\partial_t$, where $(x,t)\in\mathbb{R}^n\times \mathbb{R}$. In \cite{MPr}, Mateu \& Prat studied the corresponding caloric capacities associated with the fractional heat equation. That is, the equation associated with the pseudo-differential operator $\Theta^s:=(-\Delta_x)^s+\partial_t$, for $0<s<1$. The authors distinguished the cases $s=1/2$, $1/2<s<1$ and $0<s<1/2$, focusing mainly on the first. The study for $s=1/2$ was quite successful, obtaining a removability result for the $\Theta^{1/2}$-equation, as it was done for the heat equation. For instance, if $f$ is a solution of the $1/2$-heat equation in $\mathbb{R}^{n+1}\setminus{E}$ satisfying $\|f\|_{L^\infty(\mathbb{R}^{n+1})}<\infty,$ then $E$ will be removable if and only if the $1/2$-caloric capacity of $E$ is null. One of the main features that characterized the case $s=1/2$ is that the authors could work with an explicit fundamental solution for the operator $\Theta^{1/2}$, namely
	\begin{equation*}
		P(x,t):=\frac{t}{\big( |x|^2+t^2 \big)^{(n+1)/2}}\chi_{\{t>0\}}(x,t),
	\end{equation*}
	a continuous function in $\mathbb{R}^{n+1}\setminus{\{0\}}$, harmonic in $\mathbb{R}^{n+1}\setminus{\{t=0\}}$. This, together with the fact that the Fourier symbols of $(-\Delta_x)^{1/2}$ and $\partial_t$ share the same homogeneity, allowed Mateu \& Prat to characterize the removability of compact sets in a simpler way than the cases $s\neq 1/2$. More precisely, they established that the removability condition for $E$ can be studied by controlling the supremum of expressions of the form $|\langle T, 1 \rangle|$, where $T$ is a distribution supported on $E$ satisfying the normalization condition $\|P\ast T\|_{L^\infty(\mathbb{R}^{n+1})} \leq 1$. The latter supremum is known as the $1/2$-caloric capacity of $E$, written $\gamma_{\Theta^{1/2}}(E)$. If the supremum is only taken among positive Borel measures supported on $E$ (with the same normalization condition), we obtain the smaller capacity $\gamma_{\Theta^{1/2},+}(E)$. On the other hand, for $s\neq 1/2$, although there is a fundamental solution $P_s(x,t)$, its expression is no longer explicit. However, there are precise estimates computed by the probabilists Blumenthal \& Getoor \cite{BlGe} who described the power-like tail behaviour of $P_s$ and allowed the authors to carry out a similar study. Nevertheless, there was still a technical obstacle: the fact that the Fourier symbols of $(-\Delta_x)^s$ and $\partial_t$ do not share homogeneity. This feature, essentially, ended up by limiting the amount of results obtained.
    
	In the present paper we aim at obtaining a more precise description of a smaller variant of $\gamma_{\Theta^{1/2}}$, already introduced in \cite[\textsection 4]{MPr}. Denoted by $\widetilde{\gamma}_{\Theta^{1/2}}$, we call it $1/2$-\textit{symmetric caloric capacity} and it is the supremum of expressions of the form $|\langle T, 1 \rangle|$, where $T$ is a distribution supported on a compact set now satisfying both conditions
	\begin{equation*}
		\|P\ast T\|_{L^\infty(\mathbb{R}^{n+1})} \leq 1 \quad \text{ and } \quad \|P^\ast\ast T\|_{L^\infty(\mathbb{R}^{n+1})} \leq 1,    
	\end{equation*}
	where $P^\ast(x,t):=P(-x,-t)$ is the conjugate kernel of $P$. We will write $\widetilde{\gamma}_{\Theta^{1/2},+}$ its smaller version, which takes into account only positive Borel measures. The main result of this text is found in Section \ref{sec6} and it states that, in fact, the latter capacity is not that small:
	\begin{thm*}
		For $E\subset \mathbb{R}^{n+1}$ compact,
		\begin{equation*}
			\widetilde{\gamma}_{\Theta^{1/2}}(E) \approx \gamma_{\Theta^{1/2},+}(E).
		\end{equation*}
	\end{thm*}
	Such result implies, bearing in mind the results of \cite{He}, the semi-additivity of $\widetilde{\gamma}_{\Theta^{1/2}}$ (Theorem \ref{thm6.7}). Sections \ref{sec2}, \ref{sec3}, \ref{sec4} and \ref{sec5} are devoted to present the notation and preliminary results necessary to prove the above estimates. In Section \ref{sec2} we introduce basic terminology and properties, in Section \ref{sec3} we prove Theorem \ref{thm3.5}, which encapsulates all the different ways to define, equivalently, $\gamma_{\Theta^{1/2},+}$. In Section \ref{sec4} we use a particular characterization of the latter via a variational approach. With it, we construct a Whitney decomposition of a certain family of compact sets, so that we gain control of their $\widetilde{\gamma}_{\Theta^{1/2}}$ capacity. Finally, in Section \ref{sec5}, we prove Theorem \ref{thm5.19}, a general comparability result between $\widetilde{\gamma}_{\Theta^{1/2}}$ and $\gamma_{\Theta^{1/2}}$ if an additional assumption $\Athree$ is satisfied, and in Section \ref{sec6} we remove such assumption and obtain the general result. Our methods are influenced by those of Tolsa \cite[Ch.5]{To3} and Volberg \cite{Vo}, where the same type of comparability results are studied for analytic and Lipschitz harmonic capacity respectively.
    
	In our setting, since $P$ is not an anti-symmetric kernel, such arguments can be applied with some additional assumptions based on those of general $Tb$-theorems (see \cite{NTVo2} and \cite{HyMar}). The fact that $P$ is harmonic outside the hyperplane $\{t=0\}$, an $\pazocal{L}^{n+1}$-null set, and that satisfies being an $n$-dimensional Calderón-Zygmund kernel \cite[Lemma 2.1]{MPr} are the essential features used to carry out similar arguments. Finally, Section \ref{sec7} is devoted to the computation of the $\widetilde{\gamma}_{\Theta^{1/2}}$ capacity of a rectangle $R\subset \mathbb{R}^{2}$, obtaining:
	\begin{equation*}
		\widetilde{\gamma}_{\Theta^{1/2}}(R)\approx \ell_t\, \Bigg[ \frac{1}{2}\ln\bigg( 1+\frac{\ell_t^2}{\ell_x^2}\bigg) +\frac{\ell_t}{\ell_x}\arctan\bigg( \frac{\ell_x}{\ell_t} \bigg) \Bigg]^{-1},
	\end{equation*}
	where $\ell_x$ and $\ell_t$ are the horizontal and vertical side lengths of $R$. The above behavior differs substantially, for example, with that obtained for $\gamma(R)$, the analytic capacity of a rectangle. For the latter one has $\gamma(R)\approx \text{diam}(R)$ (see \cite[Proposition 1.5]{To3}).
    
	Let us also stress one last feature of the above capacities associated to the $1/2$-fractional heat equations: the kernels $P$ and $P^\ast$ are both \textit{nonnegative}. This suggests that, maybe, using classical arguments of potential theory (see \cite[Ch.I \& II]{La}, \cite[Ch.3]{Ra} or \cite{Ki}), the comparability between capacities defined through positive measures and distributions should follow by a simpler argument similar to that of \cite[p.10]{Ve}, provided that one is able to prove the existence of an \textit{equilibrium measure}. The authors have not been able to deduce its existence for $\gamma_{\Theta^{1/2},+}$, the main obstacle being the incapability of obtaining an equivalent formulation of the latter capacity in terms of an (inverted) infimum over energies. The following simple example already illustrates some problems that arise by proceeding this way: pick $E\subset \mathbb{R}^2$ the horizontal line segment $[0,1]\times\{0\}$ and $\mu:=\pazocal{H}^1|_{E}$. In \cite[\textsection 6]{MPr} Mateu \& Prat bound explicitly $P\ast\mu$ by $\pi$, obtaining that $\gamma_{\Theta^{1/2},+}(E)\geq \pi^{-1}$. However, by the definition of $P$ and the choice of $\mu$, it is also clear that $P\ast\mu(x,0)=0$, for any $x\in [0,1]$. Therefore, if we were to compute the energy of $\mu$ we would obtain
	\begin{equation*}
		I[\mu]:=\int P\ast\mu \dd\mu = \int_0^1 P\ast \mu(x,0)\dd\pazocal{H}^1(x)=0,
	\end{equation*}
	which would imply $\gamma_{\Theta^{1/2},+}(E)=+\infty$. The latter example also shows that the potentials $P\ast\mu$ do not obey the so called \textit{maximum principle}, i.e. they do not attain their maximum values at the $\text{supp}(\mu)$. To avoid this problem, the authors have tried to work with the auxiliary potentials $U_\mu(x,t):=\limsup_{(y,s)\to(x,t)} P\ast \mu(y,s)$. For these the maximum principle holds (this follows, essentially, from the fact that $P$ is subharmonic in $\mathbb{R}^{n+1}\setminus{\{0\}}$). However, the potentials $U_\mu$ still lack the \textit{continuity principle}, i.e. if $U_\mu$ restricted to $\text{supp}(\mu)$ is continuous, then $U_\mu$ should be continuous everywhere else (it already fails in the same example of above); which is an essential tool in order to carry out the construction of possible equilibrium measures. In any case, the above classical methods were finally discarded and we chose to follow arguments similar to those of \cite[Ch.5]{To3} and \cite{Vo}. \smallskip
	
	\textit{About the notation used in the sequel}: We use an upper bar to denote generic points in $\mathbb{R}^{n+1}$, as for example $\ox,\oy,\oz\ldots$ For a given $\ox=(x,t)$, $x\in\mathbb{R}^n$ will be usually referred to as the \textit{spatial} variable, and $t\in \mathbb{R}$ the \textit{time} variable.
    
    Constants appearing in this text can depend on the dimension of the ambient space and possibly on the Calderón-Zygmund (CZ) constants of $P$, and their value may change at different occurrences. The notation $A\lesssim B$ means that there exists such a constant $C$, so that $A\leq CB$. Moreover, $A\approx B$ is equivalent to $A\lesssim B \lesssim A$. Also, $A \simeq B$ will mean $A= CB$. 
    
    We will write $B_r(\ox)$ to denote the usual Euclidean ball in $\mathbb{R}^{n+1}$ of radius $r>0$ and center $\ox$. A cube in $\mathbb{R}^{n+1}$ centered at $\ox$ with side length $\ell(Q)$ will be denoted $Q(\ox,\ell(Q))$ and usually named $Q, R, S\ldots$ Cubes will always have sides parallel to the coordinate axes. If we write $\lambda Q$ we mean a dilation of $Q$ of factor $\lambda$. That is, $\lambda Q$ will be the cube concentric with $Q$ of side length $\lambda \ell(Q)$ (we follow analogous conventions with balls).
    
    We also emphasize that the gradient symbol $\nabla$ will refer to $(\nabla_x,\partial_t)$, with $x\in \mathbb{R}^n$ and $t\in \mathbb{R}$. 

    We will also write $\|\cdot \|_{\infty}:=\|\cdot\|_{L^\infty(\mathbb{R}^{n+1})}$.

    The symbol $\chi_A$, where $A\subset \mathbb{R}^{n+1}$ is Borel measurable, means the indicator function that equals $1$ on $A$, and is 0 otherwise.

    For any function $f:\mathbb{R}^{n+1}\to \mathbb{R}$, the symbol $f^\ast$ will denote $f^\ast(\ox):=f(-\ox)$. Hence $f$ being an even function is equivalent to $f=f^\ast$. 
    
    Fourier transforms of smooth functions $f:\mathbb{R}^{n+1}\to\mathbb{R}$ will be denoted by $\pazocal{F}[f]$ or simply as $\widehat{f}$. They will always be taken with respect to spatial variables. For a fixed $t\in\mathbb{R}$, we put $g_t(x):=f(x,t)$ and write $\widehat{f}(\cdot,t):=\widehat{g_t}$.

    Since the Laplacian operator (fractional or not) will frequently appear in our discussion and will always be taken with respect to spatial variables, we will adopt the notation $\Delta:=\Delta_x$.
    
	\section{Notation and basic definitions}
	\label{sec2}
	Our ambient space will be $\mathbb{R}^{n+1}$ and $\Theta^{1/2}$ will denote the \textit{$1/2$-heat operator},
	\begin{equation*}
		\Theta^{1/2} := (-\Delta)^{1/2} + \partial_t,
	\end{equation*}
	where $(-\Delta)^{1/2}$ is a pseudo-differential operator known as the \textit{$1/2$-Laplacian} with respect to the spatial variable. It may be defined through its Fourier transform for every fixed $t$,
	\begin{equation*}
		\pazocal{F}[{(-\Delta)^{1/2}}f](\xi,t)=|\xi|\pazocal{F}[f](\xi,t),
	\end{equation*}
	or by its integral representation
	\begin{align*}
		(-\Delta)^{1/2} f(x,t)&\simeq \text{p.v.}\int_{\mathbb{R}^n}\frac{f(x,t)-f(y,t)}{|x-y|^{n+1}}\text{d}y  \\
		&\simeq \int_{\mathbb{R}^n}\frac{f(x+y,t)-2f(x,t)+f(x-y,t)}{|y|^{n+1}}\text{d}y. 
	\end{align*}
	The reader may find more details about the properties of such operator in \cite[\textsection{3}]{DPV} or \cite{St}. Borrowing the notation of \cite{MPr}, let $P$ be the fundamental solution of the $1/2$-heat equation in $\mathbb{R}^{n+1}$, which is given by \cite[Eq. 2.2]{Va}
	\begin{equation*}
		P(\ox)= \frac{t}{|\ox|^{n+1}}\, \chi_{\{t>0\}}(\ox).
	\end{equation*}
	Notice that the previous kernel is locally integrable in $\mathbb{R}^{n+1}$, continuous in $\mathbb{R}^{n+1}\setminus{\{0\}}$ and not differentiable at any point of the hyperplane $\{t=0\}$. Another function that will frequently appear is $P^{\ast}$,
	\begin{equation*}
		P^{\ast}(\ox):=P(-\ox)= \frac{-t}{|\ox|^{n+1}}\, \chi_{\{t<0\}}(\ox),
	\end{equation*}
	that is the conjugate kernel associated to $P$. For a given distribution $T$, we also define $T^\ast$ to be the distribution acting as
	\begin{equation*}
		\langle T^\ast, \varphi \rangle := \langle T, \varphi^\ast \rangle, \quad \forall \varphi \in \pazocal{C}^\infty_c(\mathbb{R}^{n+1}). 
	\end{equation*}
	It is not hard to check (approximating via test functions and using the associativity of \cite[Theorem 8.15]{Ca}) that for any $T$ distribution and $S$ distribution with compact support,
	\begin{equation}
		\label{eq2.1}
		\langle T\ast S, \varphi \rangle = \langle S, T^\ast \ast \varphi \rangle, \quad \forall \varphi \in \pazocal{C}^\infty_c(\mathbb{R}^{n+1}).
	\end{equation}
	
	With this notions we observe that on the one hand, for any $\ox=(x,t)\in \mathbb{R}^{n+1}$ we have
	\begin{align*}
		(-\Delta)^{1/2}&P^{\ast}(\ox)=C'\int_{\mathbb{R}^n}\frac{P^{\ast}(x+y,t)-2P^{\ast}(x,t)+P^{\ast}(x-y,t)}{|y|^{n+1}}\text{d}y\\
		&\hspace{-0.2cm}=C'\int_{\mathbb{R}^n}\frac{P(-x-y,-t)-2P(-x,-t)+P(-x+y,-t)}{|y|^{n+1}}\text{d}y =(-\Delta)^{1/2}P(-\ox).
	\end{align*}
	That is, we have  
	\begin{equation*}
		(-\Delta)^{1/2}P^{\ast} = \big[ (-\Delta)^{1/2}P \big]^\ast.
	\end{equation*}
	On the other hand, we also have for any $\varphi$ test function in $\mathbb{R}^{n+1}$,
	\begin{equation*}
		\langle (\partial_t P)^\ast, \varphi \rangle := \langle \partial_t P, \varphi^\ast \rangle := - \langle P, \partial_t \varphi^\ast \rangle = \langle P, (\partial_t \varphi)^\ast \rangle = \langle P^\ast, \partial_t \varphi \rangle =: -\langle \partial_tP^\ast, \varphi \rangle.
	\end{equation*}
	Therefore, the following distributional identity holds
	\begin{equation*}
		(\partial_t P)^\ast = -\partial_t P^\ast.
	\end{equation*}
	Hence, if we define the operator
	\begin{equation*}
		\overline{\Theta}^{1/2} := (-\Delta)^{1/2} - \partial_t,
	\end{equation*}
	we have that
	\begin{equation*}
		\overline{\Theta}^{1/2}P^{\ast} = \big[ \Theta^{1/2} P \big]^\ast = \delta_0^\ast = \delta_0,
	\end{equation*}
	implying that $P^{\ast}$ is the fundamental solution of $\overline{\Theta}^{1/2}$. We will refer to $P^\ast$ and $\overline{\Theta}^{1/2}$ as the \textit{conjugates} of $P$ and $\Theta^{1/2}$ respectively.
    
	We recall the definition of the $1/2$\textit{-caloric capacity} and some of its fundamental variants introduced in \cite{MPr}. For $E\subset \mathbb{R}^{n+1}$ compact, define its $1/2$-caloric capacity as
	\begin{equation*}
		\gamma_{\Theta^{1/2}}(E)=\sup |\langle T, 1 \rangle| ,
	\end{equation*}
	where the supremum is taken over all distributions $T$ with $\text{supp}(T)\subset E$ satisfying
	\begin{equation*}
		\|P\ast T\|_{\infty}:=\|P\ast T\|_{L^\infty(\mathbb{R}^{n+1})} \leq 1.
	\end{equation*}
	Such distributions will be called \textit{admissible for $\gamma_{\Theta^{1/2}}(E)$}. In \cite[Lemma 3.4]{MPr} it is proved that for a distribution $T$ in $\mathbb{R}^{n+1}$ with $\|P\ast T\|_\infty\leq 1$, if $\varphi$ is a test function supported on $Q\subset \mathbb{R}^{n+1}$ a cube, with $\|\nabla \varphi\|_\infty \leq \ell(Q)^{-1}$, then
	\begin{equation*}
		|\langle T, \varphi \rangle |\leq C \ell(Q)^n,
	\end{equation*}
	for some constant $C>0$, that we may assume $C\geq 1$ without loss of generality. 
    \begin{defn}
        If the previous property holds for a distribution $T$, we say that $T$ has \textit{growth} of degree $n$ or simply $n$-growth (\textit{with constant $C$}). We write $\pazocal{T}_n(E)$ those distributions with $n$-growth and constant 1 supported on a compact set $E\subset \mathbb{R}^{n+1}$.
    \end{defn}
    It is easy to see that the previous definition of growth for a distribution generalizes the usual notion of growth if $T$ coincides with a positive measure.

    \begin{defn}
        Recall that a positive measure $\mu$ is said to have $n$-growth with constant $C$ if $\mu\big( B(\ox,r) \big) \leq Cr^n$, for all $\ox\in \mathbb{R}^{n+1},\, r>0$. We also write $\Sigma_n(E)$ the collection measures with $n$-growth and constant 1 supported on a compact set $E\subset \mathbb{R}^{n+1}$.    
    \end{defn}

	It is clear that the above properties are invariant if formulated using cubes instead of balls.
    
	The fact that $C$ depends only on the dimension of the ambient space and the CZ constants of $P$, motivates the following redefinition of $\gamma_{\Theta^{1/2}}$, that will be the one used throughout the whole text:
	\begin{defn}
		For a compact subset $E\subset \mathbb{R}^{n+1}$, define its \textit{$1/2$-caloric capacity} as
		\begin{equation*}
			\gamma(E)=\gamma_{\Theta^{1/2}}(E):=\sup |\langle T, 1 \rangle| ,
		\end{equation*}
		the supremum taken over distributions in $\pazocal{T}_n(E)$ and satisfying $\|P\ast T\|_{\infty}\leq 1$.
	\end{defn}
	
	We also define the ($1/2,+$)-\textit{caloric capacity}, denoted by $\gamma_{\Theta^{1/2},+}$, in the same way as $\gamma_{\Theta^{1/2}}$, but with the supremum only taken with respect to positive Borel measures. More precisely,
	\begin{equation*}
		\gamma_{+}(E)=\gamma_{\Theta^{1/2},+}(E):=\sup \big\{ \mu(E)\,:\, \mu \in \Sigma_n(E), \, \|P\ast \mu\|_\infty \leq 1 \big\},
	\end{equation*}
	Analogous definitions are associated with the operator $\overline{\Theta}^{1/2}$, giving rise to the objects
	\begin{equation*}
		\overline{\gamma}=\gamma_{\overline{\Theta}^{1/2}} \quad \text{and} \quad \overline{\gamma}_+=\gamma_{\overline{\Theta}^{1/2},+}.
	\end{equation*}
    Moreover, it is usual to extend all of the above definitions to a greater variety of sets. Namely, if $E\subseteq \mathbb{R}^{n+1}$ is any Borel set,
	\begin{equation*}
		\gamma(E):=\sup_{\substack{K\subseteq E\\ K \text{compact}}} \gamma(K),
	\end{equation*}
	and similarly for the rest of capacities.
    
	In some of the results presented in \cite[\textsection 3]{MPr}, one encounters expressions of the form $\langle T, \varphi\rangle$, where $\varphi$ is a compactly supported $\pazocal{C}^1$ function, and $T$ is a compactly supported distribution satisfying $\|P\ast T\|_\infty\leq 1$. This last estimate and the fact that $(-\Delta)^{1/2}\varphi\in L^1(\mathbb{R}^{n+1})$, proved in \cite[\textsection 3]{MPr}, allows to give meaning to $\langle T, \varphi\rangle$ (that a priori may not have sense since $\varphi$ is not a test function) as follows:
	\begin{equation*}
		\langle T, \varphi\rangle := \big\langle P\ast T, \overline{\Theta}^{1/2}\varphi \big\rangle = \big\langle P\ast T, (-\Delta)^{1/2}\varphi+\partial_t\varphi \big\rangle.
	\end{equation*}

	We need to extend such definition to a slightly wider class of functions. We do not claim that such class will be the largest where $\langle T, \varphi \rangle$ can be defined if $T$ is a compactly supported distribution with $\|P\ast T\|_\infty\leq 1$, but it will suffice for our purposes.
	\begin{defn}
		Let $Q\subset \mathbb{R}^{n+1}$ be a cube, $\pazocal{N}\subset \mathbb{R}^{n+1}$ a $\pazocal{L}^{n+1}$-null set, and $\varphi:\mathbb{R}^{n+1}\to \mathbb{R}$ a function. For $m\in\mathbb{N}\cup\{\infty\}$ we will write
		\begin{equation*}
			\varphi\in \pazocal{C}_c(Q)\cap \Big( \pazocal{C}_{x}^{m}(Q), \pazocal{C}_{t}^{m}(Q\setminus{\pazocal{N}}) \Big) \quad \text{or simply} \quad \varphi\in \pazocal{C}^m_{c,\pazocal{N}}(Q),
		\end{equation*}
		if $\varphi$ is continuous and compactly supported on $Q$, $m$ times continuously differentiable with respect to the spatial variables, and $m$ times continuously differentiable with respect to the temporal variable $t$ except for a null set $\pazocal{N}$ (where it may not even be differentiable). We also write
		\begin{equation*}
			\varphi\in \pazocal{C}(Q)\cap \Big( \pazocal{C}_{x}^{m}(Q), \pazocal{C}_{t}^{m}(Q\setminus{\pazocal{N}}) \Big) = \pazocal{C}^m_{\pazocal{N}}(Q),
		\end{equation*}
		if $\varphi$ satisfies the above regularity properties on $Q$, but we do not require it to be compactly supported there (in fact, it may even have a larger domain of definition). In the sequel we will only be interested in the case $m=\infty$.
	\end{defn}
	For $\varphi\in \pazocal{C}^{m}_{c,\pazocal{N}}(Q)$ one can define $\langle T, \varphi \rangle$ as above, assuming $\|P\ast T \|_\infty \leq 1$ and $T$ compactly supported. Indeed, observe that we only have to give meaning to $\langle P\ast T, (-\Delta)^{1/2}\varphi\rangle$ and $\langle P\ast T, \partial_t\varphi\rangle$. The former term can be defined as in \cite[\textsection 3]{MPr}, where the integrability of $(-\Delta)^{1/2}\varphi$ is proved using only regularity assumptions over the spatial variables. Regarding the term $\langle P\ast T, \partial_t\varphi\rangle$, since $P\ast T \in L^\infty(\mathbb{R}^{n+1})$, it can be simply defined as
	\begin{equation*}
		\langle P\ast T, \partial_t\varphi\rangle := \int_{\text{supp}(\varphi)\setminus{\pazocal{N}}} P\ast T(\ox)\partial_t\varphi(\ox)\dd\ox.
	\end{equation*}
	
	\begin{rem}
		\label{rem2.1}
		We notice that if for $\varphi \in \pazocal{C}_{c,\pazocal{N}}^1(Q)$ it also happens that $\|\nabla \varphi \|_\infty \leq \ell(Q)^{-1}$, then \cite[Corollary 3.3]{MPr} also holds with exactly the same proof for such $\varphi$; and therefore also does \cite[Theorem 3.1]{MPr} (localization of potentials) and \cite[Lemma 3.4]{MPr} ($n$-growth of admissible distributions). Here, condition $\|\nabla \varphi \|_\infty \leq \ell(Q)^{-1}$ must be understood as: for every test function $\psi$,
		\begin{equation*}
			|\langle \nabla \varphi, \psi \rangle | = \bigg\rvert \int_Q \varphi (\nabla \psi) \, \dd\pazocal{L}^{n+1}\bigg\rvert = \bigg\rvert \int_{Q\setminus{\pazocal{{N}}}} (\nabla \varphi) \psi \, \dd\pazocal{L}^{n+1}\bigg\rvert \leq \ell(Q)^{-1}\|\psi\|_{L^1}. 
		\end{equation*}
	\end{rem}
	To end this introductory section, we define more capacities that are already presented in \cite{MPr} and some other auxiliary ones. First, the so called 1/2-\textit{symmetric caloric capacities}, where the normalization conditions for the potentials against $P$ and $P^\ast$ are both required:
	\begin{align*}
		\widetilde{\gamma}(E)=\widetilde{\gamma}_{\Theta^{1/2}}(E)&:= \sup \big\{ |\langle T, 1 \rangle|\,:\, T \in \pazocal{T}_n(E), \, \|P\ast T\|_\infty \leq 1, \|P^\ast \ast T\|_\infty \leq 1 \big\},\\
		\widetilde{\gamma}_+(E)=\widetilde{\gamma}_{\Theta^{1/2},+}(E)&:= \sup \big\{ \mu(E)\,:\, \mu \in \Sigma_n(E), \, \|P\ast \mu\|_\infty \leq 1, \|P^\ast \ast \mu\|_\infty \leq 1 \big\}.
	\end{align*}
	We also have the following auxiliary capacities, that we only consider in the context of positive measures:
	\begin{align*}
		\gamma_{\text{\normalfont{sy}},+}(E)&:= \sup \big\{ \mu(E)\,:\, \mu \in \Sigma_n(E), \, \|P_{\text{\normalfont{sy}}}\ast \mu\|_\infty \leq 1\big\},
	\end{align*}
	where $\Psy$ is the symmetric part of $P$, that is,
	\begin{equation*}
		P_{\text{\normalfont{sy}}}(\ox):=\frac{P(\ox)+P(-\ox)}{2} =  \frac{|t|}{2|\ox|^{n+1}}.
	\end{equation*}
	Finally, we also consider the following capacities defined via a normalization condition that involves an $L^2(\mu)$ operator bound:
	\begin{align*}
		\gamma_{\text{op}}(E) = \gamma_{\Theta^{1/2},\text{op}}(E) &:= \sup\big\{\mu(E)\;:\; \mu\in \Sigma_n(E),\; \|\PP_\mu\|_{L^2(\mu)\to L^2(\mu)}\leq 1  \big\},\\
		\gamma_{\text{sy},\text{op}}(E) &:= \sup\big\{\mu(E)\;:\; \mu\in \Sigma_n(E),\; \|\PP_{\text{sy},\mu}\|_{L^2(\mu)\to L^2(\mu)}\leq 1  \big\},
	\end{align*}
	where $\PP_\mu$ and $\PP_{\text{sy},\mu}$ are the convolution operators associated to $P$ and $\Psy$ with respect to the measure $\mu$. For any $\varepsilon>0$, we will also denote by $\PP_{\mu,\varepsilon}$ and $\PP_{\text{sy},\mu,\varepsilon}$ the convolution operators associated to the truncated kernels $P(\ox)\chi_{\{|\ox|>\varepsilon\}}$ and $\Psy(\ox)\chi_{\{|\ox|>\varepsilon\}}$ respectively. With this definition, we will understand  
	\begin{equation*}
		\|\PP_\mu\|_{L^2(\mu)\to L^2(\mu)}:= \sup_{\varepsilon>0} \|\PP_{\mu,\varepsilon}\|_{L^2(\mu)\to L^2(\mu)},
	\end{equation*}
	and analogously for $\Psy$. We will frequently use the notation:
	\begin{equation*}
		\PP(f\mu):=\PP_\mu f := P\ast (f\mu),
	\end{equation*}
	and normally $f$ will happen to be the constant function $1$, so the reader has to understand $\PP\mu$ as $\PP_\mu1$. The same considerations will be done for $\PP_{\text{sy}}$ and for $\PP^\ast$, the latter being the convolution operator associated to the conjugate kernel $P^\ast$.
    
	From the results of \cite{MPr} and \cite{He} one can deduce an important property of the above different capacities: most of them are comparable.
	\begin{thm}
		\label{thm2.1}
		The previous capacities defined through positive measures are all comparable. That is, for any compact set $E\subset \mathbb{R}^{n+1}$,
		\begin{align*}
			\gamma_{+}(E) \approx \overline{\gamma}_+(E) \approx \widetilde{\gamma}_{+}(E)  &\approx \gamma_{\text{\normalfont{sy}},+}(E)\\
			&\approx \gamma_{\text{\normalfont{op}}}(E) \approx \gamma_{\text{\normalfont{sy,op}}}(E).
		\end{align*}
	\end{thm}
	\begin{rem}
		\label{rem2.2}
		To be precise, the capacity $\gamma_{\text{\normalfont{sy,op}}}$ is not introduced in the aforementioned references. To deduce the comparability of $\gamma_{\text{\normalfont{sy,op}}}$ to all the rest it suffices to check $\gamma_{\text{sy,op}}\lesssim \gamma_{\text{op}}$, since the reverse inequality can be verified to hold easily. Fix $\mu$ any admissible measure for $\gamma_{\text{\normalfont{sy,op}}}(E)$ with $\gamma_{\text{\normalfont{sy,op}}}(E)\leq 2\mu(E)$. Notice that since $\PP_\mu$ and $\PP^\ast_\mu$ are $n$-dimensional CZ operators, then $\PP_{\text{sy},\mu}$ also satisfies such property. Therefore, following an analogous proof to \cite[Theorem 2.16]{To3} one deduces, from the $L^2(\mu)$-boundedness of $\PP_{\text{sy},\mu}$, its boundedness as an operator from the space of Borel finite signed measures $M(\mathbb{R}^{n+1})$ to $L^{1,\infty}(\mu)$. That is, there exists a constant $C>0$ such that for any $\nu\in M(\mathbb{R}^{n+1})$,
		\begin{equation*}
			\mu\big( \big\{ |\PP_{\text{sy},\varepsilon}\nu(\ox)|>\lambda \big\} \big)\leq C\frac{\|\nu\|}{\lambda}, \qquad \forall \varepsilon>0, \,\forall\lambda>0.
		\end{equation*}
		With this, and using that $\PP_{\text{sy}}=\PP_{\text{sy}}^\ast$ (the operator is its own adjoint), applying for example \cite[Ch.VII, Theorem 23]{Ch}, we can find a function $h:E\to [0,1]$ such that $\int_E h\text{d}\mu \geq \mu(E)/2$ and $\|\PP_{\text{sy}} (h\mu)\|_\infty < 4C$. Then,
		\begin{equation*}
			\gamma_{\text{sy},+}(E)\geq \frac{1}{4C}\int_E h\text{d}\mu \geq \frac{\mu(E)}{8C}\geq \frac{1}{16C}\,\gamma_{\text{sy,op}}(E),
		\end{equation*}
		that implies $\gamma_{\text{sy,op}}(E)\lesssim \gamma_{\text{sy},+}(E)\approx \gamma_{\text{op}}(E)$, and we are done.
	\end{rem}
	
	\section{Even more capacities}
	\label{sec3}
	
	The main goal of this section is to add three additional ways of understanding $1/2$-caloric capacities: one via an $L^2(\mu)$ normalization condition, another using a uniform bound at \textit{all} points of the compact set, and a last one formulated in a variational way. All the results proved in this section are summarized in Theorem \ref{thm3.5}.
	
	\subsection{\mathinhead{L^2(\mu)}{} normalization. Suppressed kernels.} 
	\label{subsec3.1}
	Given a convolution kernel one may define its \textit{suppressed version} in terms of a general nonnegative 1-Lipschitz function (with constant 1). However, for our purposes we do not need to work with such generality, and we will focus on a particular type of functions. Given a closed set $F\subset \mathbb{R}^{n+1}$, we will typically have
	\begin{equation*}
		\Lambda(\ox):= \text{dist}(\ox,F), \quad \forall \ox\in \mathbb{R}^{n+1}.
	\end{equation*}
	The above $\Lambda$ is indeed a 1-Lipschitz function: take any $\ox,\oy\in \mathbb{R}^{n+1}$ and observe that for any $\oz\in F$ we have $\Lambda(\ox)\leq |\ox-\oz|\leq |\ox-\oy|+|\oy-\oz|$. Since $\oz$ is arbitrary, $\Lambda(\ox)\leq |\ox-\oy|+\Lambda(\oy) \, \leftrightarrow \, \Lambda(\ox)-\Lambda(\oy)\leq |\ox-\oy|$. Repeating the same argument changing the roles of $\ox$ and $\oy$ we deduce the desired Lipschitz property. We define the suppressed kernel associated to $P$, that depends on $\Lambda$ and a previously fixed closed set $F$, as:
	\begin{equation}
		\label{eq3.1}
		P_\Lambda(\ox,\oy):=\frac{P(\ox-\oy)}{1+P(\ox-\oy)^2\Lambda(\ox)^n\Lambda(\oy)^n},
	\end{equation}
	for all the pair of points where such expression makes sense. Notice that it is also a nonnegative kernel and differentiable, with respect to $\ox$ or $\oy$, out of a set of null Lebesgue measure. We may also define $P_\Lambda^\ast$ and $P_{\text{sy},\Lambda}$ in an analogous way. In fact, for the sake of notation, we shall prove three basic properties only for $P_\Lambda$, but they can be checked to hold also for $P_\Lambda^\ast$ and $P_{\text{sy},\Lambda}$ with almost exactly the same proofs. Essentially the convolution kernels $P_\Lambda, P_\Lambda^\ast$ and $P_{\text{sy},\Lambda}$ in $\mathbb{R}^{n+1}$ define $n$-dimensional CZ operators. In fact, such suppressed kernels satisfy \textit{3} as an additional property.
	
	\begin{lem}
		\label{lem3.1}
		Let $\Lambda:\mathbb{R}^{n+1}\to \mathbb[0,\infty)$ be a 1-Lipschitz function. Then, kernels $P_\Lambda, P_\Lambda^\ast$ and $P_{\text{\normalfont{sy}},\Lambda}$ define $n$-dimensional CZ operators which satisfy being ``well suppressed" at the points where $\Lambda>0$. Namely, if we fix any $\ox\neq \oy$ points in $\mathbb{R}^{n+1}$, the precise estimates that $P_\Lambda$ satisfies $($as well as $P_\Lambda^\ast$ and $P_{\text{\normalfont{sy}},\Lambda})$ are
		\begin{enumerate}[leftmargin=*,itemsep=0.1cm]
			\item[\textit{1}.] $|P_\Lambda(\ox,\oy)|\leq |\ox-\oy|^{-n}$.
			\item[\textit{2}.] For any $\ox'\in \mathbb{R}^{n+1}$ satisfying $|\ox-\ox'|\leq |\ox-\oy|/2$,
			\begin{equation*}
				\big\rvert P_\Lambda(\ox,\oy) - P_\Lambda(\ox',\oy) \big\rvert \lesssim \frac{|\ox-\ox'|}{|\ox-\oy|^{n+1}} \quad \text{and} \quad  \big\rvert P_\Lambda(\oy,\ox) - P_\Lambda(\oy,\ox') \big\rvert \lesssim \frac{|\ox-\ox'|}{|\ox-\oy|^{n+1}}.
			\end{equation*}
			In particular, where the differentiation is well-defined,
			\begin{equation*}
				\big\rvert \nabla_{\ox} P_\Lambda(\ox,\oy) \big\rvert + \big\rvert \nabla_{\oy} P_\Lambda(\ox,\oy) \big\rvert \lesssim \frac{1}{|\ox-\oy|^{n+1}}.
			\end{equation*}
			\item[\textit{3}.] $|P_\Lambda(\ox,\oy)|\lesssim \min\big\{ \Lambda(\ox)^{-n}, \Lambda(\oy)^{-n} \big\}$.
		\end{enumerate}
	\end{lem}
	\begin{proof}
		The proof of \textit{1} is trivial, since $|P_\Lambda(\ox,\oy)|\leq P(\ox-\oy)$. We move on by proving the first inequality in \textit{2}, and we will do it in a way that the arguments will also be valid for $P^\ast$, implying the validity of the second estimate of \textit{2}. Let us first assume $\Lambda(\ox')\leq \Lambda(\ox)$ and define
		\begin{equation*}
			g(\ox,\oy):=1+P(\ox-\oy)^2\Lambda(\ox)^n\Lambda(\oy)^n.
		\end{equation*}
		Using \cite[Lemma 2.1]{MPr} (that is also valid for $P^\ast$ and hence also for $\Psy$) we get
		\begin{align*}
			|P_\Lambda(\ox,&\oy)-P_\Lambda(\ox',\oy)|\\
			&\leq \frac{|P(\ox-\oy)-P(\ox'-\oy)|}{g(\ox,\oy)}+P(\ox'-\oy)\frac{|P(\ox-\oy)^2\Lambda(\ox)^n-P(\ox'-\oy)^2\Lambda(\ox')^n|\Lambda(\oy)^n}{g(\ox,\oy)g(\ox',\oy)}\\
			&\lesssim \frac{|\ox-\ox'|}{|\ox-\oy|^{n+1}}+ P(\ox'-\oy)\frac{|P(\ox-\oy)^2-P(\ox'-\oy)^2|\Lambda(\ox')^n\Lambda(\oy)^n}{g(\ox,\oy)g(\ox',\oy)}\\
			&\hspace{3.5cm}+P(\ox'-\oy)\frac{P(\ox-\oy)^2|\Lambda(\ox)^n-\Lambda(\ox')^n|\Lambda(\oy)^n}{g(\ox,\oy)g(\ox',\oy)}=:\text{\rom{1}}+\text{\rom{2}}+\text{\rom{3}}.
		\end{align*}
		We only have to deal with \rom{2} and \rom{3}. For the first, we rewrite it as
		\begin{equation*}
			\text{\rom{2}} = |P(\ox-\oy)-P(\ox'-\oy)|\cdot \text{\rom{4}},
		\end{equation*}
		where
		\begin{equation*}
			\text{\rom{4}}:=\frac{P(\ox'-\oy)|P(\ox-\oy)+P(\ox'-\oy)|\Lambda(\ox')^n\Lambda(\oy)^n}{g(\ox,\oy)g(\ox',\oy)}.
		\end{equation*}
		It is clear that by the definition of $g$
		\begin{equation*}
			\text{\rom{4}} \leq 1+\frac{P(\ox'-\oy)P(\ox-\oy)\Lambda(\ox')^n\Lambda(\oy)^n}{g(\ox,\oy)g(\ox',\oy)}.
		\end{equation*}
		For the remaining fraction, if $P(\ox-\oy)\leq P(\ox'-\oy)$ we obtain the same bound by 1. If on the other hand $P(\ox-\oy)>P(\ox'-\oy)$, using that $\Lambda(\ox')\leq \Lambda(\ox)$ we also obtain the same estimate and we are done with \rom{2}. Regarding \rom{3}, observe that there exists a point $\overline{\xi}$ in the line segment joining $\ox$ and $\ox'$ such that
		\begin{equation*}
			\text{\rom{3}} \leq |\ox-\ox'|P(\ox'-\oy)\frac{n\Lambda(\overline{\xi})^{n-1}\Lambda(\oy)^nP(\ox-\oy)^2}{g(\ox,\oy)g(\ox',\oy)}.
		\end{equation*}
		We distinguish two cases: if $\Lambda(\overline{\xi})\leq |\ox-\oy|$. By the Lipschitz property of $\Lambda$, $\Lambda(\oy)\leq 2 |\ox-\oy|$. Thus,
		\begin{align*}
			\text{\rom{3}}&\leq  |\ox-\ox'|\frac{n\Lambda(\overline{\xi})^{n-1}\Lambda(\oy)^nP(\ox-\oy)^2}{g(\ox,\oy)g(\ox',\oy)}\Big[ P(\ox-\oy)+\big\rvert P(\ox-\oy)-P(\ox'-\oy) \big\rvert \Big]\\
			&\lesssim |\ox-\ox'|\frac{|\ox-\oy|^{2n-1}}{|\ox-\oy|^{2n}}\Bigg[ \frac{1}{|\ox-\oy|^n}+\frac{|\ox-\ox'|}{|\ox-\oy|^{n+1}} \Bigg] = \frac{|\ox-\ox'|}{|\ox-\oy|^{n+1}}\Bigg[ 1+\frac{|\ox-\ox'|}{|\ox-\oy|}\Bigg]\\
			&\lesssim \frac{|\ox-\ox'|}{|\ox-\oy|^{n+1}},
		\end{align*}
		that is the desired estimate. If on the other hand we had $\Lambda(\overline{\xi})> |\ox-\oy|$, then notice that by the Lipschitz property of $\Lambda$,
		\begin{equation*}
			\Lambda(\overline{\xi})\leq \Lambda(\ox)+|\ox-\overline{\xi}|\leq \Lambda(\ox)+|\ox-\ox'|\leq \Lambda(\ox)+\frac{|\ox-\oy|}{2}< \Lambda(\ox)+\frac{\Lambda(\overline{\xi})}{2},
		\end{equation*}
		that is $\Lambda(\ox)>\Lambda(\overline{\xi})/2$, and in particular $g(\ox,\oy)\gtrsim g(\overline{\xi},\oy)$. Therefore,
		\begin{align*}
			\text{\rom{3}}&\lesssim P(\ox'-\oy)\frac{|\ox-\ox'|}{|\ox-\oy|}\frac{P(\ox-\oy)^2\Lambda(\overline{\xi})^{n}\Lambda(\oy)^n}{g(\overline{\xi},\oy)g(\ox',\oy)}\leq P(\ox'-\oy)\frac{|\ox-\ox'|}{|\ox-\oy|}\\
			&\leq \frac{1}{2}\big\rvert P(\ox'-\oy) - P(\ox-\oy) \big\rvert + \frac{|\ox-\ox'|}{|\ox-\oy|^{n+1}} \lesssim \frac{|\ox-\ox'|}{|\ox-\oy|^{n+1}},
		\end{align*}
		that is what we wanted to prove. With this, we have proved \textit{2} for the case $\Lambda(\ox')\leq \Lambda(\ox)$. The case $\Lambda(\ox)\leq \Lambda(\ox')$ can be treated analogously, using symmetric considerations with respect to $\ox$ and $\ox'$.
        
		To prove the gradient estimates, let us fix $j\in\{1,\ldots,n+1\}$ and consider $\partial_{x_j}P_{\Lambda}(\ox,\oy)$ at the points where this expression is well defined. Observe that since $\Lambda$ is a Lipschitz function and $P$ can be differentiated except for null Lebesgue sets, $\partial_{x_j}P_{\Lambda}$ makes sense out of sets of null Lebesgue measure. In any case, choose $0<|h|\ll1$ so that $|h|\leq |\ox-\oy|/2$, where $|\ox-\oy|$ is strictly positive since $\ox\neq \oy$ are two fixed different points. Hence, applying the already proved estimates of \textit{2} we get
		\begin{equation*}
			\frac{|P_{\Lambda}(\ox+h\widehat{e}_j,\oy)-P_{\Lambda}(\ox,\oy)|}{|h|} \lesssim \frac{1}{|h|}\frac{|h|}{|\ox-\oy|^{n+1}}= \frac{1}{|\ox-\oy|^{n+1}},
		\end{equation*}
		so taking the limit as $h\to 0$ the result follows.
        
		Finally we prove \textit{3}. Assume that $\Lambda(\ox)=\max\big\{ \Lambda(\ox),\Lambda(\oy) \big\}$. In this setting,
		\begin{align}
			\label{eq3.2}
			\big\rvert P_\Lambda(\ox,\oy) \big\rvert = \frac{1}{P(\ox-\oy)}\Bigg[ \frac{1}{P(\ox-\oy)^{-2}+\Lambda(\ox)^n\Lambda(\oy)^n} \Bigg]\leq \frac{1}{P(\ox-\oy)}\frac{1}{\Lambda(\ox)^n\Lambda(\oy)^n}.
		\end{align}
		Assume that
		\begin{equation}
			\label{eq3.3}
			\frac{1}{P(\ox-\oy)}\leq \frac{\Lambda(\ox)^n}{2^n}.
		\end{equation}
		In this case, $|\ox-\oy|^n\leq P(\ox-\oy)^{-1}\leq \Lambda(\ox)^n/2^n$, that is $|\ox-\oy|\leq \Lambda(\ox)/2$. By the Lipschitz property we also have $\Lambda(\oy)\geq \Lambda(\ox)-|\ox-\oy|$, so $\Lambda(\oy)\geq \Lambda(\ox)/2$, and applying this estimate to \eqref{eq3.2} the result would follow. If \eqref{eq3.3} did not hold, we would simply have 
		\begin{equation*}
			|P_\Lambda(\ox-\oy)|\leq |P(\ox-\oy)| \leq \frac{2^n}{\Lambda(\ox)^n},   
		\end{equation*}
		and we would be also done. Moreover, it is clear that if $\Lambda(\oy)=\max\big\{ \Lambda(\ox),\Lambda(\oy) \big\}$, the above proof can be carried out in an analogous way, obtaining the desired result.
	\end{proof}

	For a fixed $E\subset \mathbb{R}^{n+1}$ compact set, we will use the above suppressed kernels to study the capacity
	\begin{equation*}
		\gamma_{\text{sy},2}(E) := \sup \big\{ \mu(E)\,:\, \mu \in \Sigma_n(E), \, \|P_{\text{\normalfont{sy}}}\ast \mu\|_{L^{2}(\mu)} \leq \mu(E)^{1/2}\big\}.
	\end{equation*} 
	It is clear that $\gamma_{\text{\normalfont{sy}},\text{\normalfont{op}}}(E)\leq  \gamma_{\text{sy},2}(E)$. In the following lemma we prove that the reverse inequality also holds:
	\begin{lem} 
		\label{lem3.2}
		For $E\subset \mathbb{R}^{n+1}$ compact, the following inequality holds:
		\begin{equation*}
			\gamma_{\text{\normalfont{sy}},2}(E)\lesssim \gamma_{\text{\normalfont{sy}},\text{\normalfont{op}}}(E).            
		\end{equation*}
	\end{lem}
	\begin{proof}
		Let $\mu$ be an admissible measure for $\gamma_{\text{\normalfont{sy}},2}(E)$. Then, by definition,
		\begin{equation*}
			\int_{E}|\PP_{\text{sy},\varepsilon}\mu(\ox)|^2\text{d}\mu(\ox)\leq \mu(E), \quad \forall \varepsilon>0.
		\end{equation*}
		Consider the maximal operator
		\begin{equation*}
			\PP_{\text{sy},\ast}\mu(\ox) := \sup_{\varepsilon>0} |\PP_{\text{sy},\varepsilon}\mu(\ox)|,
		\end{equation*}
		that by the nonnegativity of $\Psy$ and $\mu$ is such that
		\begin{equation*}
			\PP_{\text{sy},\ast}\mu(\ox) = \lim_{\varepsilon\to 0} \PP_{\text{sy},\varepsilon}\mu(\ox),
		\end{equation*}
		which is well-defined for every point $\ox\in \mathbb{R}^{n+1}$. In fact, by the monotone convergence theorem, it is clear that the above $L^2(\mu)$ uniform estimate with respect to $\varepsilon>0$, can be rewritten as
		\begin{equation*}
			\int_{E}|\PP_{\text{sy},\ast}\mu(\ox)|^2\text{d}\mu(\ox)\leq \mu(E).
		\end{equation*}
		Now we define
		\begin{equation*}
			F:= \Big\{ \ox\in E \,:\, \PP_{\text{sy},\ast}\mu(\ox)\leq \sqrt{2} \Big\},
		\end{equation*}
		and apply Chebyshev's inequality to obtain
		\begin{equation*}
			\mu(E\setminus{F})\leq \frac{1}{2}\int_E |\PP_{\text{sy},\ast}\mu(\ox)|^2\text{d}\mu(\ox)\leq \frac{\mu(E)}{2} \quad \text{and then} \quad \mu(F) \geq \frac{\mu(E)}{2}.
		\end{equation*}
		If we are able to prove
		\begin{equation*}
			\|\PP_{\text{sy},\mu|_F}\|_{L^2(\mu|_F) \to L^2(\mu|_F)}\lesssim 1    
		\end{equation*}
		we will be done, since the latter would mean $\gamma_{\text{sy},\text{op}}(E)\gtrsim \mu(F) \geq \mu(E)/2$, and by the arbitrariness of $\mu$ the result would follow.
        
		To proceed, let us observe that $\PP_{\text{sy},\ast}\mu$ can be understood as a non-decreasing limit of continuous functions. Indeed, we may rewrite, for each $\varepsilon>0$ and $\ox\in \mathbb{R}^{n+1}$,
		\begin{equation*}
			\PP_{\text{sy},\varepsilon}\mu(\ox)=\int_{\mathbb{R}^{n+1}}\Psy(\ox-\oy)\chi_{\{|\ox-\oy|>\varepsilon\}}\dd\mu(\oy)=:P_{\text{sy},\chi_{\varepsilon}}\ast \mu(\ox),
		\end{equation*}
		where $P_{\text{sy},\chi_{\varepsilon}}(\cdot):=\Psy(\cdot)\chi_{\{|\cdot|>\varepsilon\}}$. Now, let us fix $\psi$ any test function such that $0\leq \psi \leq 1$, $\psi\equiv 0$ in $B_1(0)$ and $\psi\equiv 1$ in $\mathbb{R}^{n+1}\setminus{B_2(0)}$. Assume also that $\|\nabla \psi\|_\infty \leq 1$. Write $\psi_\varepsilon(\cdot):=\psi(\cdot/\varepsilon)$ and define, for each $\ox\in \mathbb{R}^{n+1}$,
		\begin{equation*}
			P_{\text{sy},\psi_{\varepsilon}}\ast \mu(\ox):=\int_{\mathbb{R}^{n+1}}\Psy(\ox-\oy)\psi_\varepsilon(\ox-\oy)\text{d}\mu(\oy).
		\end{equation*}
		Observe that since $\Psy\geq 0$,
		\begin{equation*}
			P_{\text{sy},\chi_{2\varepsilon}}\ast \mu (\ox) \leq P_{\text{sy},\psi_{\varepsilon}}\ast \mu (\ox) \leq P_{\text{sy},\chi_{\varepsilon}}\ast \mu(\ox), \hspace{0.25cm} \forall \ox \in \mathbb{R}^{n+1},
		\end{equation*}
		which implies, by definition,
		\begin{equation*}
			\lim_{\varepsilon\to 0} P_{\text{sy},\psi_{\varepsilon}}\ast \mu (\ox) = \PP_{\text{sy},\varepsilon}\mu(\ox), \hspace{0.25cm} \forall \ox \in \mathbb{R}^{n+1}.
		\end{equation*}
		Now, since $\mu$ is in particular a finite measure, it is not difficult to prove using Fatou's lemma and its reverse version, that for each $\varepsilon>0$ fixed, $P_{\text{sy},\psi_{\varepsilon}}\ast \mu$ is a continuous function. Moreover, since $\mu$ and $\Psy$ are nonnegative, it is clear that for $\varepsilon_1<\varepsilon_2$, then  $P_{\text{sy},\psi_{\varepsilon_2}}\ast \mu \leq P_{\text{sy},\psi_{\varepsilon_1}}\ast \mu$. Hence, indeed, $\PP_{\text{sy},\ast}\mu$ can be defined as a non-decreasing limit of continuous functions. Therefore, $\PP_{\text{sy},\ast}\mu$ is lower semicontinuous \cite[Ch.4]{Ba} and thus $F$ is a closed set. Now consider $B$ a large open ball containing $E$ inside and define $G:= B\setminus{F}$. For any $\oz\in G\cap E=E\setminus{F}$ we define
		\begin{equation*}
			\varepsilon(\oz):=\text{dist}(\oz,F) = \Lambda(\oz),
		\end{equation*}
		recalling the notation used when defining the suppressed kernels. First we estimate for any $\varepsilon\geq \varepsilon(\oz)$:
		\begin{align*}
			\PP_{\text{sy},\varepsilon}\mu(\oz) = \int_{|\oy-\oz|\geq 2\varepsilon} \Psy(\oz-\oy)\dd \mu(\oy)+\int_{\varepsilon<|\oy-\oz|\leq 2\varepsilon} \Psy(\oz-\oy)\dd \mu(\oy) =: \text{\rom{1}}+\text{\rom{2}}.
		\end{align*}
		Notice that by the $n$-growth of $\mu$ we have
		\begin{align*}
			|\text{\rom{2}}| \leq \frac{\mu(B_{2\varepsilon}(\oz))}{\varepsilon^n} \lesssim 1.
		\end{align*}
		To deal with \rom{1}, let $\ox_0$ be a closest point to $\oz$ on $F$, so that $\varepsilon(\oz)=|\ox_0-\oz|$. Then
		\begin{align*}
			|\text{\rom{1}}|\leq \int_{|\oy-\oz|\geq 2\varepsilon} \Psy(\ox_0-\oy)\dd \mu(\oy)+\int_{|\oy-\oz|\geq 2\varepsilon} |\Psy(\ox_0-\oy)-\Psy(\oz-\oy)|\dd \mu(\oy)=: \text{\rom{3}}+\text{\rom{4}}.
		\end{align*}
		It is clear that $\text{\rom{3}}\leq \sqrt{2}$, since $\ox_0\in F$. On the other hand, since $|(\oz-\oy)-(\ox_0-\oy)|\leq \varepsilon \leq |\oz-\oy|/2$, integration over annuli yields
		\begin{align*}
			|\text{\rom{4}}|\lesssim \varepsilon \int_{|\oy-\oz|\geq \varepsilon} \frac{\dd\mu(\oy)}{|\oz-\oy|^{n+1}}\lesssim 1.
		\end{align*}
		So we have $|\PP_{\text{sy},\varepsilon}\mu(\oz)|\lesssim 1$, for any $\varepsilon\geq \varepsilon(\oz)$. Now we shall prove
		\begin{equation*}
			A:=\big\rvert \PP_{\text{sy},\Lambda,\varepsilon}\mu(\oz)-\PP_{\text{sy},\varepsilon}\mu(\oz) \big\rvert\lesssim 1, \quad \forall \varepsilon\geq \varepsilon(\oz),
		\end{equation*}
		where $\PP_{\text{sy},\Lambda,\varepsilon}\mu(\oz) := \int_{|\oz-\oy|>\varepsilon} P_{\text{sy},\Lambda}(\oz,\oy) \dd \mu(\oy)$. Since $\varepsilon(\oz)=\Lambda(\oz)$ by definition,
		\begin{align*}
			A \leq \int_{|\oy-\oz|\geq \Lambda(\oz)} \big\rvert P_{\text{sy},\Lambda}(\oz,\oy)-\Psy(\oz-\oy) \big\rvert \dd\mu(\oy).
		\end{align*}
		Using the definition of $P_{\text{sy},\Lambda}$ we know
		\begin{equation}
			\label{eq3.4}
			\big\rvert P_{\text{sy},\Lambda}(\oz,\oy)-\Psy(\oz-\oy) \big\rvert \leq \Psy(\oz-\oy)\Bigg[ \frac{\Psy(\oz-\oy)^2\Lambda(\oz)^n\Lambda(\oy)^n}{1+\Psy(\oz-\oy)^2\Lambda(\oz)^n\Lambda(\oy)^n} \Bigg].
		\end{equation}
		In the region of integration $\Lambda(\oz)\leq |\oy-\oz|$, then $\Lambda(\oy)\leq 2|\oy-\oz|$, by the Lipschitz property. Therefore, returning to \eqref{eq3.4} we get
		\begin{align*}
			\big\rvert P_{\text{sy},\Lambda}(\oz,\oy)-\Psy(\oz-\oy) \big\rvert \leq \Psy(\oz-\oy)&\Big[ \Psy(\oz-\oy)^2\Lambda(\oz)^n\Lambda(\oy)^n \Big]\\
			&\leq 2^n \Psy(\oz-\oy)^2\Lambda(\oz)^n\lesssim \frac{\Lambda(\oz)^n}{|\oz-\oy|^{2n}},
		\end{align*}
		so by the $n$-growth of $\mu$ we obtain (again, integrating over annuli)
		\begin{equation*}
			A\lesssim \int_{|\oy-\oz|\geq \Lambda(\oz)} \frac{\Lambda(\oz)^n}{|\oz-\oy|^{2n}}\dd\mu(\oy)\lesssim 1,
		\end{equation*}
		that is what we wanted to prove. Hence, combining this last estimate with $|\PP_{\text{sy},\varepsilon}\mu(\oz)|\lesssim 1$, for any $\varepsilon\geq \varepsilon(\oz)$ we deduce
		\begin{equation*}
			|\PP_{\text{sy},\Lambda,\varepsilon}\mu(\oz)|\lesssim 1, \quad \forall \varepsilon \geq \varepsilon (\oz).
		\end{equation*}
		Now, fixing $\eta \in (0,\varepsilon(\oz))$, using property \textit{3} in Lemma \ref{lem3.1} together with $\eta < \Lambda(\oz)$, we get
		\begin{equation*}
			\bigg\rvert \int_{|\oy-\oz|\leq \eta} P_{\text{sy},\Lambda}(\oz,\oy)\dd\mu(\oy) \bigg\rvert \leq \frac{1}{\Lambda(\oz)^n}\mu(B_\eta(\oz))\leq 1.
		\end{equation*}
		All in all, we have proved
		\begin{equation*}
			\big\rvert \PP_{\text{sy},\Lambda,\ast}\mu(\oz) \big\rvert \lesssim 1, \quad \forall \oz\in E\setminus{F}.
		\end{equation*}
		In fact, this last estimate also holds for $\oz\in F$, since in this case $\Lambda(\oz)=0$ and $P_{\text{sy},\Lambda}=P_{\text{sy}}$, implying
		\begin{align*}
			\big\rvert \PP_{\text{sy},\Lambda,\ast}\mu(\oz) \big\rvert &= \sup_{\eta>0} \bigg\rvert \int_{|\oy-\oz|\geq \eta} P_{\text{sy},\Lambda}(\oz,\oy)\dd\mu(\oy) \bigg\rvert\\
			&=\sup_{\eta>0} \bigg\rvert \int_{|\oy-\oz|\geq \eta} \Psy(\oz-\oy)\dd\mu(\oy) \bigg\rvert = \PP_{\text{sy},\ast}\mu(\oz) \leq \sqrt{2},
		\end{align*}
		by definition of $F$. Hence, we get $\|\PP_{\text{sy},\Lambda}\mu\|_{L^{\infty}(\mu)}\lesssim 1$, where this estimate has to be understood as $\|\PP_{\text{sy},\Lambda,\varepsilon}\mu\|_{L^{\infty}(\mu)}\lesssim 1$, uniformly on $\varepsilon>0$. The same estimate also holds for $\PP_{\text{sy},\Lambda}^\ast \mu$, since $P_{\text{sy},\Lambda}(\ox,\oy)=P_{\text{sy},\Lambda}(\oy,\ox)$, by the symmetry of $\Psy$. Moreover, it is also clear that for any cube $Q\subset \mathbb{R}^{n+1}$ and any $\varepsilon>0$,
		\begin{align*}
			\big\rvert \big\langle \PP_{\text{sy},\Lambda,\mu,\varepsilon}\chi_Q, \chi_Q \big\rangle \big\rvert = \int_Q \bigg( \int_{Q\cap \{|\oy-\oz|>\varepsilon\}} &P_{\text{sy},\Lambda}(\oz,\oy)\dd\mu(\oy) \bigg)\dd\mu(\oz)\\
			&\leq \int_Q \PP_{\text{sy},\Lambda,\varepsilon}\mu(\oz)\dd \mu(\oz)\lesssim \mu(Q),
		\end{align*}
		by the nonnegativity of $P_{\text{sy},\Lambda}$. So bearing in mind Remark Lemma \ref{lem3.1}, we may apply a suitable $T1$-theorem, namely \cite[Theorem 1.3]{To1}, to deduce $\|\PP_{\text{sy},\Lambda,\mu}\|_{L^2(\mu)\to L^2(\mu)}\lesssim 1$. But for $f,g\in L^2(\mu)$ supported on $F$ one has $\langle \PP_{\text{sy},\Lambda}f,g \rangle = \langle \PP_{\text{sy}}f,g \rangle$, meaning that
		\begin{equation*}
			\|\PP_{\text{sy},\mu|_F}\|_{L^2(\mu|_F)\to L^2(\mu|_F)}\lesssim 1,
		\end{equation*}
		and we are done.
	\end{proof}
	\begin{rem}
		\label{rem3.1}
		Due to the above Lemma \ref{lem3.2}, we may add $\gamma_{\text{sy},2}$ to the statement of Theorem \ref{thm2.1}. In fact, the arguments above also prove that we can also add the following capacity to Theorem \ref{thm2.1},
		\begin{equation*}
			\widetilde{\gamma}_{2}(E):= \sup \big\{ \mu(E)\,:\, \mu \in \Sigma_n(E), \, \|P\ast \mu\|_{L^{2}(\mu)} \leq \mu(E)^{1/2}, \|P^\ast \ast \mu\|_{L^{2}(\mu)} \leq \mu(E)^{1/2}\big\}.
		\end{equation*} 
	\end{rem}
	\bigskip
	
	\subsection{Normalization by a uniform bound at all points of the support}
	\label{subsec3.2}
	For a fixed $E\subset \mathbb{R}^{n+1}$ compact set, let us consider now the capacity
	\begin{equation*}
		\gamma^\star_+(E):= \sup\big\{ \mu(E)\,:\, \mu\in \Sigma_n(E),\, |P\ast \mu(\ox)|\leq 1,  \,\forall \ox\in E\big\},
	\end{equation*}
	as well as all the variants $\overline{\gamma}^\star_+, \widetilde{\gamma}^\star_+$ and $\gamma^{\star}_{\text{sy},+}$. It is clear that $\gamma^\star_+(E) \leq \gamma_+(E)\approx \gamma_{\text{op}}(E)$. We claim that the following holds:
	\begin{equation*}
		\gamma_{\text{op}}(E)\lesssim \gamma_+^\star(E).
	\end{equation*}
	Indeed, given $\mu\in \Sigma_n(E)$ admissible for $\gamma_{\text{op}}(E)$ with $\gamma_{\text{op}}(E)\leq 2\mu(E)$, proceeding as in Remark \ref{rem2.2} we can find a function $h:E\to [0,1]$ such that $\int_E h\text{d}\mu \geq \mu(E)/2$ and $\|\PP (h\mu)\|_\infty \leq C$, for some $C>0$ constant. We know that the latter estimate implies $\|\PP (h\mu)\|_{L^\infty(\mu)} \leq C'$ for some other constant $C'>0$ (by a Cotlar type inequality analogous to that of \cite[Lemma 5.4]{MaP}). Applying now Cotlar's inequality of \cite[Theorem 2.18]{To3}, for example, we have
	\begin{equation*}
		\sup_{\varepsilon>0}\big\rvert \PP_{\varepsilon}(h\mu)(\ox) \big\rvert \leq C''\big( \widetilde{M}_\mu|\PP (h\mu)|^\delta \big)^{1/\delta}(\ox) + C'''(\widetilde{M}_\mu h)(\ox), \quad \forall \ox\in \mathbb{R}^{n+1},
	\end{equation*}
	where
	\begin{equation*}
		\widetilde{M}_\mu h(\ox):= \sup_{r>0} \frac{1}{\mu(B_{3r}(\ox))}\int_{B_r(\ox)}|h(\oy)|\dd\mu(\oy),
	\end{equation*}
	$\delta\in(0,1)$ is arbitrarily fixed and $C'', C'''$ are positive constants that depend on the dimension, $\delta$ and the $L^2(\mu)$-norm of $\PP_\mu$, that equals 1 by hypothesis. Since $0\leq h \leq 1$, it is clear that $\widetilde{M}_\mu h \leq 1$; and since $\|\PP(h\mu)\|_{L^\infty(\mu)} \leq C'$, we deduce for $\delta = 1/2$
	\begin{equation*}
		\widetilde{M}_\mu|\PP(h\mu)|^{1/2}(\ox) = \sup_{r>0} \frac{1}{\mu(B_{3r}(\ox))}\int_{B_r(\ox)} \big\rvert \PP(h\mu)(\ox) \big\rvert^{1/2} \dd\mu(\oy) \leq (C')^{1/2}.
	\end{equation*}
	Therefore, setting $\widetilde{C}:=\max\{1,C'C''+C'''\}$ we get
	\begin{equation*}
		\big\rvert \PP(\widetilde{C}^{-1}h\mu)(\ox) \big\rvert = \sup_{\varepsilon>0}\big\rvert \PP_{\varepsilon}(\widetilde{C}^{-1}h\mu)(\ox) \big\rvert \leq 1, \quad \forall\ox\in \mathbb{R}^{n+1}.
	\end{equation*}
	So in particular $\widetilde{C}^{-1}h\mu$ is an admissible measure (up to a dimensional factor that makes it an $n$-growth measure with constant 1) for $\gamma^\star_+(E)$. Thus,
	\begin{equation*}
		\gamma^\star_+(E)\gtrsim \frac{1}{\widetilde{C}}\int_E h\dd \mu \geq \frac{1}{2\widetilde{C}}\mu(E),
	\end{equation*}
	and by the arbitrariness of $\mu$ we conclude that $\gamma_{\text{op}}(E)\lesssim \gamma_+^\star(E)$.
	
	\begin{rem}
		\label{rem3.2}
		Therefore, we conclude that the capacity $\gamma_+^\star$ can be added to the statement of Theorem \ref{thm2.1}. In fact, since the operator $L^2(\mu)$-norm of $\PP_{\mu}^\ast$ is the same as that of $\PP_\mu$, following an analogous argument we may also add $\overline{\gamma}^\star_+$ as well as $\widetilde{\gamma}^\star_+$. Further, the same arguments can also be followed to compare $\gamma^\star_{\text{sy},+}$ with $\gamma_{\text{sy},\text{op}}$, allowing us to also consider $\gamma^\star_{\text{sy},+}$ in Theorem \ref{thm2.1}.
	\end{rem}
	
	\subsection{An equivalent variational capacity}
	\label{subsec3.3}
	We shall consider a construction already presented in the proof of Lemma \ref{lem3.2}: let $\psi$ be a radial test function with $0\leq \psi \leq 1$, $\psi\equiv 0$ in $B_{1/2}(0)$, $\psi\equiv 1$ in $\mathbb{R}^{n+1}\setminus{B_1(0)}$ and such that $\|\nabla \psi\|_\infty \leq 1$. Fix $E\subset \mathbb{R}^{n+1}$ compact set and $\mu\in \Sigma_n(E)$. For each $\tau>0$, write $\psi_\tau(\cdot):=\psi(\cdot/\tau)$ and define, for each $\ox\in \mathbb{R}^{n+1}$ and $f\in L^1_{\text{loc}}(\mu)$.
	\begin{equation*}
		\PP_{\text{sy},\psi_{\tau}}(f\mu)(\ox):= \Psy\psi_\tau \ast (f\mu) (\ox) = \int_{\mathbb{R}^{n+1}}\Psy(\ox-\oy)\psi_\tau(\ox-\oy)f(\oy)\text{d}\mu(\oy).
	\end{equation*}
	
	Let us prove, using \cite[Lemma 2.1]{MPr} (also valid for $P^\ast$ and $\Psy$) , that the regularized (continuous) kernel $\Psy\psi_\tau$ defines a CZ convolution operator with constants not depending on $\tau$. It is clear that $|\Psy\psi_\tau(\ox)|\leq |\ox|^{-n}$. Then, by the symmetry of $\Psy\psi_\tau$, it suffices to check that for any $\ox,\ox'\in\mathbb{R}^{n+1}$ with $\ox\neq 0$ and $|\ox-\ox'|\leq |\ox|/2$,
	\begin{equation}
		\label{eq3.5}
		\big\rvert  \Psy\psi_\tau(\ox)-\Psy\psi_\tau(\ox') \big\rvert \leq C \frac{|\ox-\ox'|}{|\ox|^{n+1}},
	\end{equation}
	where $C>0$ is a constant independent of $\tau$. To prove this, we distinguish two cases: if $\tau\geq |\ox|/4$, 
	\begin{align*}
		\big\rvert  \Psy\psi_\tau(\ox)-\Psy\psi_\tau(\ox') \big\rvert &\leq \big\rvert \Psy(\ox)-\Psy(\ox') \big\rvert\psi_\tau(\ox')+\big\rvert \psi_\tau(\ox)-\psi_\tau(\ox') \big\rvert \Psy(\ox)\\
		&\lesssim \frac{|\ox-\ox'|}{|\ox|^{n+1}}+\frac{|\ox-\ox'|}{\tau}\frac{1}{|\ox|^n} \lesssim  \frac{|\ox-\ox'|}{|\ox|^{n+1}},
	\end{align*}
	where we have applied \cite[Lemma 2.1]{MPr} and $\|\nabla \psi\|_\infty \leq 1$. If on the other hand $\tau< |\ox|/4$, by definition of $\psi_\tau$ we have $\psi_\tau(\ox)=1$. In addition, by the triangle inequality,
	\begin{equation}
		\label{eq3.6}
		|\ox'|\geq \big\rvert |\ox-\ox'|-|\ox| \big\rvert
	\end{equation}
	If $||\ox-\ox'|-|\ox|| = |\ox-\ox'|-|\ox|$, then $|\ox-\ox'|\geq |\ox|$, so $|\ox-\ox'|\approx |\ox|$. So in this case,
	\begin{align*}
		\big\rvert  \Psy\psi_\tau(\ox)-\Psy\psi_\tau(\ox') \big\rvert&\leq \big\rvert \Psy(\ox)-\Psy(\ox') \big\rvert\psi_\tau(\ox')+\big\rvert 1-\psi_\tau(\ox') \big\rvert \Psy(\ox)\\
		&\lesssim \frac{|\ox-\ox'|}{|\ox|^{n+1}}+\frac{2}{|\ox|^n} \leq \frac{3|\ox-\ox'|}{|\ox|^{n+1}},
	\end{align*}
	and we are done. If on the other hand $||\ox-\ox'|-|\ox|| = |\ox|-|\ox-\ox'|$, then $|\ox'|\geq |\ox|-|\ox-\ox'| \geq |\ox|/2$ which implies $\psi_\tau(\ox')=1$. Then, in this case,
	\begin{align*}
		\big\rvert  \Psy\psi_\tau(\ox)-\Psy\psi_\tau(\ox') \big\rvert = \big\rvert \Psy(\ox)-\Psy(\ox') \big\rvert \lesssim \frac{|\ox-\ox'|}{|\ox|^{n+1}},
	\end{align*}
	and the proof of \eqref{eq3.5} is complete.
    
	Let us also observe that for any $f\in L^1_{\text{loc}}(\mu)$ any $\tau>0$ and any $\ox\in E$,
	\begin{align*}
		\big\rvert \PP_{\text{sy},\tau}(f\mu)(\ox) -\PP_{\text{sy},\psi_{\tau}}(f\mu)(\ox) \big\rvert &\leq \int_{\tau/2 < |\ox-\oy|<\tau} \Psy(\ox-\oy) \big\rvert 1-\psi_\tau(\ox-\oy) \big\rvert|f(\oy)|\dd\mu(\oy)\\
		&\lesssim \sup_{\tau>0} \frac{1}{\tau^n} \int_{|\ox-\oy|<\tau} |f(\oy)|\dd\mu(\oy).
	\end{align*}
	In particular, choosing $f\equiv 1$, by the $n$-growth of $\mu$ we get
	\begin{equation}
		\label{eq3.7}
		\big\rvert \PP_{\text{sy},\tau}\mu(\ox) -\PP_{\text{sy},\psi_{\tau}}\mu(\ox) \big\rvert \lesssim 1, \quad \forall \ox\in E,\, \tau>0.
	\end{equation}
	Let us define the auxiliary capacity, for each $\tau>0$,
	\begin{equation*}
		\gamma^\star_{\text{sy},\psi_\tau,+}(E):= \sup\big\{ \mu(E)\,:\, \mu\in \Sigma_n(E),\, |\Psy\psi_\tau \ast \mu(\ox)|\leq 1,  \,\forall \ox\in E\big\}.
	\end{equation*}
	\begin{lem}
		\label{lem3.3}
		The following estimates hold:
		\begin{enumerate}[leftmargin=*,itemsep=0.1cm]
			\item[\textit{1}.] $\limsup\limits_{\tau\to 0} \gamma^\star_{\text{\normalfont{sy}},\psi_\tau,+}(E) \lesssim \gamma^\star_{\text{\normalfont{sy}},+}(E) $.
			\item[\textit{2}.] $\liminf\limits_{\tau\to 0} \gamma^\star_{\text{\normalfont{sy}},\psi_\tau,+}(E) \gtrsim \gamma^\star_{\text{\normalfont{sy}},+}(E) $.
		\end{enumerate}
	\end{lem}
	\begin{proof}
		We begin by proving \textit{1}. Let $(\tau_k)_k$ be a monotonically decreasing sequence to 0, and let $\mu_k\in \Sigma_n(E)$ be admissible for $\gamma^\star_{\text{\normalfont{sy}},\psi_{\tau_k},+}(E)$ and such that $\gamma^\star_{\text{\normalfont{sy}},\psi_{\tau_k},+}(E)\leq 2\mu_k(E)$. Passing to a subsequence if necessary, assume $\mu_k \rightharpoonup \mu_0$ weakly \cite[Theorem 1.23]{Ma}, i.e.
		\begin{equation*}
			\lim_{k\to\infty} \int \varphi\dd\mu_k = \int \varphi \dd\mu_0, \quad \forall\varphi \in \pazocal{C}_c(\mathbb{R}^{n+1}).
		\end{equation*}
		It is not hard to prove that $\mu_0\in \Sigma_n(E)$. Now fix $m>0$ integer and $\tau\in [\tau_{m+1},\tau_m]$, and also assume $k>m$. By \eqref{eq3.7} we get that for $\ox\in E$,
		\begin{align*}
			\big\rvert \PP_{\text{sy},\psi_{\tau}}\mu_k(\ox) \big\rvert \leq \big\rvert \PP_{\text{sy},\psi_{\tau_k}}\mu_k(\ox) \big\rvert + \big\rvert \PP_{\text{sy},\psi_\tau}\mu_k(\ox) -\PP_{\text{sy},\psi_{\tau_k}}\mu_k(\ox) \big\rvert \leq 1+C,
		\end{align*}
		for some finite constant $C>0$. Since $\mu_k\rightharpoonup \mu_0$ we obtain $\big\rvert \PP_{\text{sy},\psi_{\tau}}\mu_0\big\rvert \leq 1+C$ on $E$ and for each $\tau>0$. In addition, using \eqref{eq3.7} again we get
		\begin{equation*}
			\big\rvert \PP_{\text{sy},\tau}\mu_0(\ox)\big\rvert \leq 1+2C,\quad \forall \ox\in E, \, \tau >0,
		\end{equation*}
		that is $ |\PP_{\text{sy}}\mu_0(\ox)\big| \leq 1+2C$ for all $\ox\in E$. Thus, by the semicontinuity properties of weak convergence of \cite[Theorem 1.24]{Ma} we get
		\begin{align*}
			\gamma^{\star}_{\text{sy},+}(E) \geq \frac{\mu_0(E)}{1+2C} \geq \frac{1}{1+2C} \limsup_{k\to\infty} \mu_k(E) \gtrsim \limsup_{\tau\to 0} \gamma^\star_{\text{\normalfont{sy}},\psi_{\tau_k},+}(E),
		\end{align*}
		and \textit{1} follows. To prove \textit{2}, take $\mu$ admissible for $\gamma^\star_{\text{sy},+}(E)$. Using \eqref{eq3.7},
		\begin{equation*}
			\big\rvert \PP_{\text{sy},\psi_\tau}\mu(\ox)\big\rvert \leq 1+C,\quad \forall \ox\in E, \, \tau >0,
		\end{equation*}
		implying that $\gamma^{\star}_{\text{sy},\psi_\tau,+}(E)\geq \frac{\mu(E)}{1+C}$, and by the arbitrariness of $\mu$ we finally deduce
		\begin{equation*}
			\inf_{\tau>0} \gamma^{\star}_{\text{sy},\psi_\tau,+}(E) \gtrsim \gamma^{\star}_{\text{sy},+}(E).
		\end{equation*}
	\end{proof}
	
	\begin{rem}
		\label{rem3.3}
		Let us observe that, for each $\tau>0$, using the regularized (continuous) kernel $\Psy\psi_\tau$ we may also define the rest of corresponding capacities, requiring different normalization conditions over the potentials. In fact, the following chain of estimates holds for any compact set $E\subset \mathbb{R}^{n+1}$:
		\begin{equation*}
			\gamma_{\text{sy},\psi_\tau,+}^\star(E) \underset{(1)}{\leq}  \gamma_{\text{sy},\psi_\tau,2}(E) \underset{(2)}{\lesssim} \gamma_{\text{sy},\psi_\tau,\text{op}}(E) \underset{(3)}{\lesssim} \gamma_{\text{sy},\psi_\tau,+}(E) \underset{(4)}{\leq} \gamma_{\text{sy},\psi_\tau,+}^\star(E).
		\end{equation*}
		Inequality (1) is trivial. To verify (2) proceed as in the proof of $\gamma_{\text{sy},2}(E) \lesssim \gamma_{\text{sy},\text{op}}(E)$ in Lemma \ref{lem3.2}. The proof of (3) goes as in Remark \ref{rem2.2}, and it relies on \cite[Theorem 2.16]{To3} and \cite[Ch.VII, Theorem 23]{Ch}. Notice that the independence of the CZ constants of $\Psy\psi_\tau$ of $\tau$ is crucial to ensure the validity of (2) and (3). Finally, (4) holds by the continuity of $\PP_{\text{sy},\psi_\tau}\mu$ for each $\tau>0$, that has already been argued in the proof of Lemma \ref{lem3.2}. Observe that the previous chain of estimates also holds changing every $\gamma_{\text{sy}}$ for $\widetilde{\gamma}$, where we ask the respective normalization conditions in each case for both kernels $P\psi_{\tau}$ and $P^\ast \psi_{\tau}$.
	\end{rem}
	Therefore, in particular, we shall restate Lemma \ref{lem3.3} as
	\begin{enumerate}[itemsep=0.1cm]
		\item[\textit{1}.] $\limsup\limits_{\tau\to 0} \gamma_{\text{\normalfont{sy}},\psi_\tau,2}(E) \lesssim \gamma_{\text{\normalfont{sy}},2}(E) $,
		\item[\textit{2}.] $\liminf\limits_{\tau\to 0} \gamma_{\text{\normalfont{sy}},\psi_\tau,2}(E) \gtrsim  \gamma_{\text{\normalfont{sy}},2}(E) $.
	\end{enumerate}
	Now, bearing in mind that we have found positive constants $C_1,C_2$ so that
	\begin{itemize}
		\item $\limsup\limits_{\tau\to 0} \gamma_{\text{\normalfont{sy}},\psi_\tau,2}(E) \leq C_1\,\gamma_{\text{\normalfont{sy}},2}(E)$,
		\item $\gamma_{\text{\normalfont{sy}},2}(E) \leq C_2 \, \gamma_{\text{sy},\text{op}}(E)$,
	\end{itemize}
	we define the variational capacity as follows:
	\begin{defn}
		\label{def3.1}		
		Let $E\subset \mathbb{R}^{n+1}$ be a compact set and choose $\tau_0$ small enough so that $\gamma_{\text{\normalfont{sy}},\psi_{\tau_0},2}(E) \leq 2C_1C_2\,  \gamma_{\text{sy},\text{op}}(E).$ Let $\pSS$ be the convolution operator associated to the kernel $\Psy\psi_{\tau_0}$ (that depends on $\tau_0$ and thus on $E$). We define the \textit{variational capacity of} $E$ as
		\begin{equation*}
			\gamma_{\text{var}}(E):=\sup\Bigg\{ \frac{\mu(E)^2}{\mu(E)+\int_E |\pSS\mu|^2\dd\mu}\,:\, \mu\in \Sigma_n(E) \Bigg\}.
		\end{equation*}
		We convey that expressions of the form $\frac{0}{0}$ equal 0.
	\end{defn}
	
	\begin{lem}
		\label{lem3.4}
		The supremum in $\gamma_{\text{\normalfont{var}}}(E)$ is attained. Moreover, $\gamma_{\text{\normalfont{var}}}(E) \approx \gamma_{\text{\normalfont{sy}},\text{\normalfont{op}}}(E)$, where we remark that the implicit constants do not depend on $E$.
	\end{lem}
	\begin{proof}
		For each $\mu\in\Sigma_n(E)$ define
		\begin{equation*}
			F(\mu):= \frac{\mu(E)^2}{\mu(E)+\int_E |\pSS\mu |^2\dd\mu}.
		\end{equation*}
		As it has already been mentioned in the proof of Lemma \ref{lem3.3}, it is not difficult to prove that $\Sigma_n(E)$ is closed and sequentially compact with respect to weak convergence of measures. In fact, $\Sigma_n(E)$ is compact (since it is contained in the space of finite signed Radon measures on $E$, which is metrizable, thought as the dual of the separable space $\pazocal{C}_c(\mathbb{R}^{n+1})$). Moreover, it is clear that if $\mu_k\rightharpoonup\mu_0$ on $\Sigma_n(E)$, then $\mu_k(E)\to \mu_0(E)$. In this setting, we also claim
		\begin{equation}
			\label{eq3.8}
			\lim_{k\to\infty} \int_E |\pSS\mu_k|^2\dd\mu_k = \int_E |\pSS\mu_0|^2\dd\mu_0.
		\end{equation}
		To prove the claim \eqref{eq3.8} we argue as follows: given continuous functions $\varphi_i,\psi_i$ for $i=1,\ldots N$ in $\pazocal{C}_c(\mathbb{R}^{n+1})$, by definition of weak convergence of measures we get
		\begin{align}
			\label{eq3.9}
			\lim_{k\to\infty} &\int_E\bigg(\int_E \sum_{i=1}^N\varphi_i(\ox)\psi_i(\oy)\dd\mu_k(\ox)\bigg)^2\dd\mu_k(\oy) \\
			\nonumber
			&= \lim_{k\to\infty} \bigg\{ \sum_{i=1}^N \int_E\bigg(\int_E \varphi_i(\ox)\psi_i(\oy)\dd\mu_k(\ox)\bigg)^2\dd\mu_k(\oy)\\
			\nonumber
			&\hspace{1.75cm}+2\sum_{1\leq i<j\leq N} \int_E\bigg(\int_E \varphi_i(\ox)\psi_i(\oy) \dd\mu_k(\ox) \bigg)\bigg(\int_E \varphi_j(\ox)\psi_j(\oy) \dd\mu_k(\ox) \bigg)\dd\mu_k(\oy)\bigg\}\\
			\nonumber
			&= \lim_{k\to\infty} \bigg\{ \sum_{i=1}^N \bigg( \int_E \varphi_i(\ox)\dd\mu_k(\ox) \bigg)^2\bigg( \int_E \psi_i^2(\oy)\dd\mu_k(\oy) \bigg)\\
			\nonumber
			&\hspace{1.1cm}+2\sum_{1\leq i<j\leq N} \bigg( \int_E \varphi_i(\ox)\dd\mu_k(\ox) \bigg)\bigg( \int_E \varphi_j(\ox)\dd\mu_k(\ox) \bigg)\bigg( \int_E \psi_i(\oy)\psi_j(\oy)\dd\mu_k(\oy) \bigg) \bigg\}\\
			\nonumber
			&= \int_E\bigg(\int_E \sum_{i=1}^N\varphi_i(\ox)\psi_i(\oy)\dd\mu_0(\ox)\bigg)^2\dd\mu_0(\oy).
		\end{align}
		Observe that the collection of continuous functions on the compact set $E\times E$ of the form $\sum_{i=1}^N\varphi_i(\ox)\psi_i(\oy)$ is an algebra that contains the constant functions (on $E$) and separates points. Therefore, by the Stone-Weierstrass theorem, every continuous function $f:E\times E \to \mathbb{R}$ can be uniformly approximated by such sums. Namely, consider
		\begin{equation*}
			\big(g_j(\ox,\oy)\big)_j:=\bigg( \sum_{i=1}^{N_j}\varphi_{i,j}(\ox)\psi_{i,j}(\oy) \bigg)_j,
		\end{equation*}
		so that $\|f-g_j\|_{\infty}:=\|f-g_j\|_{L^\infty(E\times E)}\to 0$ as $j\to \infty$. Observe:
		\begin{align*}
			\bigg\rvert &\int_E\bigg(\int_E f(\ox,\oy)\dd\mu_k(\ox)\bigg)^2\dd\mu_k(\oy)-\int_E\bigg(\int_E f(\ox,\oy)\dd\mu_0(\ox)\bigg)^2\dd\mu_0(\oy) \bigg\rvert\\
			&\leq \bigg\rvert \int_E\bigg(\int_E f(\ox,\oy)\dd\mu_k(\ox)\bigg)^2\dd\mu_k(\oy)-\int_E\bigg(\int_E g_j(\ox,\oy)\dd\mu_k(\ox)\bigg)^2\dd\mu_k(\oy) \bigg\rvert\\
			&\hspace{1cm}+\bigg\rvert \int_E\bigg(\int_E g_j(\ox,\oy)\dd\mu_k(\ox)\bigg)^2\dd\mu_k(\oy)-\int_E\bigg(\int_E g_j(\ox,\oy)\dd\mu_0(\ox)\bigg)^2\dd\mu_0(\oy) \bigg\rvert\\
			&\hspace{1cm}+\bigg\rvert \int_E\bigg(\int_E g_j(\ox,\oy)\dd\mu_0(\ox)\bigg)^2\dd\mu_0(\oy)-\int_E\bigg(\int_E f(\ox,\oy)\dd\mu_0(\ox)\bigg)^2\dd\mu_0(\oy) \bigg\rvert\\
			&\hspace{1cm}=:\text{\rom{1}}+\text{\rom{2}}+\text{\rom{3}}.
		\end{align*}
		Choose $j$ large enough so that $\|g_j\|_\infty\leq 2\|f\|_\infty$. We estimate \rom{1} as follows
		\begin{align*}
			\text{\rom{1}} &= \Bigg\rvert \int_E\bigg(\int_E (f+g_j)(\ox,\oy)\dd\mu_k(\ox)\bigg)\bigg(\int_E (f-g_j)(\ox,\oy)\dd\mu_k(\ox)\bigg)\dd\mu({\oy}) \Bigg\rvert\\
			&\leq 3\|f\|_\infty \cdot \text{diam}(E)^{3n} \|f-g_j\|_\infty \xrightarrow[]{j\to\infty} 0.
		\end{align*}
		Term \rom{3} can be treated the same way. For \rom{2} use that $\mu_k\rightharpoonup \mu_0$ as $k\to\infty$ and apply \eqref{eq3.9} to deduce that \rom{2} also tends to 0 as $j\to\infty$. The latter can be done since the supports of the measures $\mu_k$ are all contained in $E$, so we may understand each factor $\varphi_{i,j}$ and $\psi_{i,j}$ of the summands that define $g_j$ once extended continuously onto $\pazocal{C}_c(\mathbb{R}^{n+1})$. Hence, by the arbitrariness of $\varepsilon$ we deduce
		\begin{equation*}
			\lim_{k\to\infty} \int_E\bigg(\int_E f(\ox,\oy)\dd\mu_k(\ox)\bigg)^2\dd\mu_k(\oy) = \int_E\bigg(\int_E f(\ox,\oy)\dd\mu_0(\ox)\bigg)^2\dd\mu_0(\oy),
		\end{equation*}
		for any $f$ continuous on $E\times E$. Therefore, applying this result to $f:=\Psy\psi_{\tau_0}|_E\geq 0$ we get \eqref{eq3.8}. All in all, we have proved that $F$ defines a continuous functional on a compact space, meaning that it attains its maximum and thus the supremum that defines $\gamma_{\text{var}}$ is indeed attained.
        
		In order to prove $\gamma_{\text{\normalfont{var}}}(E) \approx \gamma_{\text{\normalfont{sy}},\text{\normalfont{op}}}(E)$, we claim that it suffices to prove
		\begin{equation*}
			\gamma_{\text{\normalfont{var}}}(E) \approx \gamma_{\text{\normalfont{sy}},\psi_{\tau_0},2}(E).
		\end{equation*}
		Indeed, since $0\leq \Psy\psi_{\tau_0}\leq \Psy$, we trivially have $\gamma_{\text{\normalfont{sy}},\text{\normalfont{op}}}(E)\leq \gamma_{\text{\normalfont{sy}},\psi_{\tau_0},2}(E)$; and by the choice of $\tau_0$, we also have $\gamma_{\text{\normalfont{sy}},\psi_{\tau_0},2}(E)\leq 2C_1C_2\,\gamma_{\text{\normalfont{sy}},\text{\normalfont{op}}}(E)$, so the claim follows.
        
		Further, observe that we can restrict ourselves to the case that there exists $\mu\in \Sigma_n(E)$ such that $F(\mu)>0$. If this was not the case, since $\int_E |\pSS\mu|\dd \mu \lesssim \text{diam}(E)^{2n}/\tau_0^n < \infty$, we would have, necessarily, $\mu(E)=0$, implying, by definition, $\gamma_{\text{var}}(E)=0=\gamma_{\text{\normalfont{sy}},\psi_{\tau_0},2}(E)$.\medskip\\
		So let us assume that there exists $\mu\in \Sigma_n(E)$ such that $F(\mu)>0$, meaning $\gamma_{\text{var}}(E)>0$. Let us begin by proving that any extremal $\mu_0$ for $F$ in this setting satisfies
		\begin{equation}
			\label{eq3.10}
			\int_E |\pSS\mu_0|^2\dd\mu_0 \leq \mu_0(E).
		\end{equation}
		If this was not the case, then $\int_E |\pSS\mu_0|^2\dd\mu_0 = M\mu_0(E)$ for some $M>1$. Define $\mu_1:=M^{-1/2}\mu_0 \in \Sigma_n(E)$ and notice that
		\begin{align*}
			F(\mu_1)&=\frac{M^{-1}\mu_0(E)^2}{M^{-1/2}\mu_0(E)+\int_E |\pSS\mu_1|^2\dd\mu_1} = \frac{\mu_0(E)}{2M^{1/2}}>\frac{\mu_0(E)}{1+M}=F(\mu_0),
		\end{align*}
		that is a contradiction. Then, \eqref{eq3.10} holds, and it implies, by definition, $\mu_0(E)\leq \gamma_{\text{\normalfont{sy}},\psi_{\tau_0},2}(E)$. So we have
		\begin{equation*}
			\gamma_{\text{var}}(E) = F(\mu_0) < \mu_0(E) \leq \gamma_{\text{\normalfont{sy}},\psi_{\tau_0},2}(E).
		\end{equation*}
		On the other hand, for any $\mu$ admissible for $\gamma_{\text{\normalfont{sy}},\psi_{\tau_0},2}(E)$,
		\begin{equation*}
			\frac{\mu(E)}{2} = \frac{\mu(E)^2}{\mu(E)+\mu(E)} \leq \frac{\mu(E)^2}{\mu(E)+\int_E |\pSS\mu|^2\dd\mu} \leq \gamma_{\text{var}}(E),
		\end{equation*}
		and the proof is complete.
	\end{proof}
	
	\begin{rem}
		\label{rem3.4}
		Let us notice that \eqref{eq3.10} holds in general for any extremal measure $\mu_0$ for $F$. Indeed, in the case that for any $\mu\in \Sigma_n(E)$ one had $F(\mu)=0$, then $\mu(E)=0$. So any measure would be extremal and \eqref{eq3.10} would hold trivially.
	\end{rem}
	
	Finally, we are ready to give an improved version of Theorem \ref{thm2.1}:
	\begin{thm}
		\label{thm3.5}
		For any $E\subset \mathbb{R}^{n+1}$ compact set,
		\begin{equation*}
			\gamma_+(E)\approx \gamma^{\star}_+(E),
		\end{equation*}
		where the respective bounds of the potentials can be taken indifferently and independently for $P, P^\ast, \Psy$ or both $P$ and $P^\ast$ simultaneously. Moreover, all the previous capacities are also comparable to $\gamma_{\text{\normalfont{sy}},2}(E),\; \gamma_{\text{\normalfont{op}}}(E),\;  \gamma_{\text{\normalfont{sy,op}}}(E)$ and $ \gamma_{\text{\normalfont{var}}}(E)$.\medskip\\
		In addition, for each $\tau>0$ the following holds
		\begin{align*}
			\gamma_{\text{\normalfont{sy}},\psi_\tau,+}(E) &\approx 
			\gamma_{\text{\normalfont{sy}},\psi_\tau,+}^\star(E)  \approx \gamma_{\text{\normalfont{sy}},\psi_\tau,2}(E)\approx \gamma_{\text{\normalfont{sy}},\psi_\tau,\text{\normalfont{op}}}(E)
		\end{align*}
		as well as $\limsup_{\tau\to 0} \gamma_{\text{\normalfont{sy}},\psi_\tau,+}(E) \lesssim \gamma_{\text{\normalfont{sy}},+}(E) \lesssim \liminf_{\tau\to 0} \gamma_{\text{\normalfont{sy}},\psi_\tau,+}(E)$. The latter relations for the capacities depending on $\psi_\tau$ also hold changing each $\gamma_{\text{\normalfont{sy}}}$ for $\widetilde{\gamma}$.
	\end{thm}
	
	\section{The construction of the cubes}
	\label{sec4}
	The goal of this section is to use Theorem \ref{thm3.5} to carry out the construction done in \cite[\textsection 5.2]{Vo} for Riesz kernels. The aim of this section is to prove the following result: 
	\begin{thm}
		\label{thm4.1}
		Let $E\subset B_1(0) \subset \mathbb{R}^{n+1}$ be a compact set consisting of a finite union of closed cubes, with sides parallel to the axes. Then, there exists a finite collection of dyadic cubes $\{\QQ_1,\ldots,\QQ_N\}$ that cover $E$ and such that $\frac{1}{2}\QQ_1,\ldots, \frac{1}{2}\QQ_N$ have disjoint interiors. Moreover, if $\FF:= \cup_{i=1}^N \QQ_i$,
		\begin{enumerate}[itemsep=0.1cm]
			\item[$\Pone$.] $\frac{5}{8}\QQ_i\cap E \neq \varnothing$, for each $i=1,\ldots,N$.
			\item[$\Ptwo$.] $\widetilde{\gamma}_+(\FF) \leq C_0 \widetilde{\gamma}_+(E)$.
			\item[$\Pthree$.] $\sum_{i=1}^N\widetilde{\gamma}_+(2\QQ_i\cap E) \leq C_1 \widetilde{\gamma}_+(E)$.
			\item[$\Pfour$.] If $\widetilde{\gamma}_+(E)\leq C_+\text{\normalfont{diam}}(E)^n$, then $\text{\normalfont{diam}}(\QQ_i)\leq \frac{1}{10}\text{\normalfont{diam}}(E)$, for each $i=1, \ldots, N$.
			\item[$\Pfive$.] The family $\{5\QQ_1,\ldots, 5\QQ_N\}$ has bounded overlap with constant $C_2$.
		\end{enumerate}
		Letters $C_0,C_1,C_2,C_+$ refer to constants.
	\end{thm}
	Before proceeding, let us clarify the from this point on $E\subset \mathbb{R}^{n+1}$ will be a fixed compact set and $\mu_0$ will always denote a maximizer of $\gamma_{\text{var}}(E)$. Let us also write explicitly the following expression, that will appear repeatedly, for the sake of clarity: for $\mu,\nu$ say positive finite Borel measures supported on $E$, we have, in light of Definition \ref{def3.1},
	\begin{equation*}
		\pSS_{\mu}(\pSS \nu)(\ox)=\int_E\int_E \Psy(\ox-\oy)\Psy(\oy-\oz)\psi_{\tau_0}(\ox-\oy)\psi_{\tau_0}(\oy-\oz)\dd\nu(\oz)\dd\mu(\oy).
	\end{equation*}
	
	\begin{lem}
		\label{lem4.1}
		Let $H$ be a positive Borel measure supported on $E$ such that for any $\lambda_0>0$ and any $\lambda\in [0,\lambda_0]$, $\mu_\lambda:=\mu_0+\lambda H\in \Sigma_n(E)$. Then,
		\begin{equation*}
			H(E)\mu_0(E)\bigg( \mu_0(E)+2\int_E|\pSS\mu_0|^2\dd\mu_0 \bigg) \leq \mu_0(E)^2\int_E \Big( |\pSS\mu_0|^2+2\pSS_{\mu_0}(\pSS\mu_0) \Big)\dd H.
		\end{equation*}
	\end{lem}
	\begin{proof}
		Let us begin by noticing that
		\begin{equation*}
			F(\mu_\lambda) = \frac{\big[ \mu_0(E)+\lambda H(E) \big]^2}{\mu_0(E)+\lambda H(E)+\int_E|\pSS \mu_\lambda |^2\dd\mu_\lambda},
		\end{equation*}
		where the integral in the denominator can be expanded as
		\begin{align*}
			\int_E|\pSS \mu_\lambda&|^2\dd\mu_\lambda =\int_{E}(\pSS\mu_0)^2\dd\mu_{0}+\lambda\int_{E}(\pSS\mu_0)^2\dd H +2\lambda \int_{E}(\pSS\mu_0)(\pSS H)\dd\mu_{0}\\
			&\hspace{2.5cm}+2\lambda^2\int_{E}(\pSS\mu_0)(\pSS H)\dd H + \lambda^2\int_{E}(\pSS H)^2\dd\mu_0 + \lambda^3\int_{E}(\pSS H)^2\dd H.
		\end{align*}
		Observe that for $\nu\in\{\mu_0, H\}$, since $\Psy$ and $\psi_{\tau_0}$ are symmetric functions,
		\begin{align*}
			\int_{E}\pSS\mu_0&1(\ox)\pSS H(\ox)\dd \nu(\ox)\\
			&=\int_E\bigg( \int_E \Psy(\ox-\oy) \psi_{\tau_0}(\ox-\oy)\dd\mu_0(\oy) \bigg)\bigg( \int_E \Psy(\ox-\oz) \psi_{\tau_0}(\ox-\oz)\dd H(\oz) \bigg)\dd\nu(\ox)\\
			&=\int_E\bigg(  \int_E \bigg( \int_E \Psy(\ox-\oy) \psi_{\tau_0}(\ox-\oy)\dd\mu_0(\oy) \bigg) \Psy(\oz-\ox) \psi_{\tau_0}(\oz-\ox) \dd\nu(\ox) \bigg) \dd H(\oz)\\
			&=\int_E \pSS_\nu (\pSS\mu_0) (\oz)\dd H (\oz).
		\end{align*}
		Therefore, we may rewrite $\int_E|\pSS \mu_\lambda|^2\dd\mu_\lambda$ as
		\begin{align*}
			\int_E|\pSS \mu_\lambda&|^2\dd\mu_\lambda = \int_{E}(\pSS\mu_0)^2\dd\mu_{0} + \lambda \bigg[ \int_{E}(\pSS\mu_0)^2\dd H +2\int_E \pSS_{\mu_0} (\pSS\mu_0) \dd H  \bigg]\\
			&\hspace{2.5cm}+\lambda^2\bigg[ \int_{E}(\pSS H)^2\dd\mu_0 +2 \int_E \pSS_{H} (\pSS\mu_0) \dd H \bigg] + \lambda^3\int_{E}(\pSS H)^2\dd H.
		\end{align*}
		So condition $F(\mu_0)\geq F(\mu_\lambda)$ can be written as
		\begin{align*}
			\frac{\mu_0(E)^2}{\mu_0(E)+\int_E (\pSS\mu_0)^2\dd \mu_0} \\
			&\hspace{-1cm}\geq \frac{\mu_0(E)^2+2\lambda\mu_0(E)H(E)+\lambda^2H(E)^2}{\mu_0(E)+\lambda H(E)+\int_E(\pSS\mu_0)^2\dd\mu_0+\lambda\big[  \cdots \big]+\lambda^2\big[  \cdots \big]+\lambda^3\big[  \cdots \big]},
		\end{align*}
		or equivalently
		\begin{align*}
			\lambda \Bigg\{ \mu_0(E)^2\bigg( -H(E)&+\int_E(\pSS\mu_0)^2\dd H + 2\int_E\pSS_{\mu_0}(\pSS\mu_0) \dd H \bigg) \\
			&-2\mu_0(E)H(E)\int_E(\pSS\mu_0)^2\dd \mu_0 \Bigg\} + \lambda^2\big[ \cdots \big] + \lambda^3 \int_{E}(\pSS H)^2\dd H \geq 0.
		\end{align*}
		Assume that $\lambda>0$ and divide both sides of the previous inequality by this factor and make $\lambda\to 0$ to obtain the desired estimate. Let us notice that the kernel associated to $\pSS$ is continuous, meaning that the terms in $[\cdots]$ and $\int_E (\pSS H)^2\dd H$ are finite, since $H(E)<\infty$ due to the fact that $\mu_\lambda \in \Sigma_n(E)$ for any $\lambda \in [0,\lambda_0]$ with $\lambda_0>0$.
	\end{proof}
	We remark that in the particular case we will apply this lemma, $E$ will  be a finite union of cubes, so $\mu_0(E)>0$. Moreover, we will choose a specific measure $H$ so that $H(E)>0$. Therefore, in this setting, the inequality in Lemma \ref{lem4.2} can be rewritten as
	\begin{equation}
		\label{eq4.1}
		\frac{\mu_0(E)+2\int_E|\pSS\mu_0|^2\dd\mu_0}{\mu_0(E)} \leq \frac{1}{H(E)} \int_E \Big( |\pSS\mu_0|^2+2\pSS_{\mu_0}(\pSS\mu_0) \Big)\dd H,
	\end{equation}
	which is a similar estimate that resembles that of \cite[Lemma 5.6]{Vo}. 
    
	We shall apply \eqref{eq4.1} to prove an auxiliary lemma, but first let us introduce the following notation for a certain maximal function applied to a measure:
	\begin{equation*}
		M\nu(\ox):= \sup_{r>0}\frac{\nu(B_r(\ox))}{r^n}
	\end{equation*}
	\begin{lem}
		\label{lem4.2}
		Let $E$ be a finite union of cubes. Then, for the potential
		\begin{equation*}
			\pazocal{U}^{\mu_0} := M\mu_0+\pSS\mu_0+\pSS_{\mu_0}\pSS\mu_0,
		\end{equation*}
		there is a constant $\alpha_0>0$ such that
		\begin{equation*}
			\pazocal{U}^{\mu_0}(\ox)\geq \alpha_0, \quad \forall \ox \in E.
		\end{equation*}
	\end{lem}
	\begin{proof}
		Choose $\alpha_1\in (0,80^{-n})$. If $\ox_0\in E$ is such that $M\mu_0(\ox_0) \geq \alpha_1$ then we are done at $\ox_0$. So assume $M\mu_0(\ox_0) <\alpha_1$. Pick $\varepsilon_1\leq \tau_0, R_0:=\varepsilon_1/10$ and $B_0:=B_{R_0}(\ox_0)$. Set $\mu_{00}:=\mu_0|_{2B_0}$ and define
		\begin{equation*}
			G:=\big\{ \oy\in B_0\cap E \, :\, M\mu_0(\oy)\leq 4^{-n} \big\}.
		\end{equation*}
		We shall prove that the complement of $G$ in $B_0\cap E$ is small. If $\oy\in (B_0\cap E)\setminus{G}$, then there exists $r=r(\oy)>0$ so that
		\begin{equation*}
			\frac{\mu_0(B_r(\oy))}{r^n}>4^{-n}.
		\end{equation*}
		If we had $r>\varepsilon_1/20 = R_0/2$, then
		\begin{equation*}
			\frac{1}{80^n}\leq \frac{\mu_0(B_r(\oy))}{20^nr^n}\leq \frac{\mu_0(B_{20r}(\ox_0))}{(20r)^n}<\alpha_1, \quad \text{since}\; M\mu_0(\ox_0)<\alpha_1,
		\end{equation*}
		and this cannot be. Then $r\leq \varepsilon_1/20$, implying $B_r(\oy)\cap (2B_0)=B_r(\oy)$, which in turn implies $\mu_{00}(B_r(\oy))=\mu_0(B_r(\oy))$. So we have found
		\begin{equation}
			\label{eq4.2}
			\forall \oy\in (B_0\cap E)\setminus{G}, \, \exists r(\oy)>0 \; \text{ such that } \; \frac{\mu_{00}(B_r(\oy))}{r^n}>4^{-n}.
		\end{equation}
		We continue by choosing $\varepsilon_1$ (that was already smaller than $\tau_0$) small enough so that
		\begin{equation*}
			\pazocal{L}^{n+1}(B_0\cap E) \geq a(n)R_0^{n+1}=:aR_0^{n+1},
		\end{equation*}
		for some dimensional constant $a>0$. This can be done since $E$ is a finite union of cubes. Notice also that the dependence on $E$ of the previous estimate is in $R_0$, since $\varepsilon_1$ depends on $\tau_0$, that depends on $E$. Using \eqref{eq4.2} and \cite[Theorem 2.1]{Ma} we obtain a countable covering $\{ B_{5r_j}(\oy_j) \}$ of $(B_0\cap E)\setminus{G}$ with $\{ B_{r_j}(\oy_j) \}$ disjoint and also satisfying \eqref{eq4.2}. Now,
		\begin{align*}
			\pazocal{L}^{n+1}\big( (B_0\cap E) \setminus{G} \big) &\leq 5^{n+1} \sum_{j}r_j^{n+1} \leq 2\cdot 5^{n+1} R_0 \sum_j r_j^n\\
			&<10\cdot 4^n \cdot 5^{n+1}R_0 \sum_{j}\mu_{00}\big( B_{r_j}(\oy_j) \big)=10\cdot 20^nR_0\,\mu_{00}\bigg( \bigcup_j  B_{r_j}(\oy_j) \bigg)\\
			& \leq 10\cdot 20^nR_0\,\mu_{0}(2B_0) \leq 10\cdot 40^nR_0^{n+1}\alpha_1=:A(n)R_0^{n+1}\alpha_1\\
			&\leq \frac{A}{a}\alpha_1 \pazocal{L}^{n+1}(B_0\cap E),
		\end{align*}
		where we have applied the hypothesis $M\mu_0(\ox_0)<\alpha_1$. Therefore, by definition of $G$, 
		\begin{equation*}
			\pazocal{L}^{n+1}(G)\geq \bigg( \frac{a}{A\alpha_1}-1 \bigg)\pazocal{L}^{n+1}\big( (B_0\cap E) \setminus{G} \big).
		\end{equation*}
		We pick $\alpha_1\in (0,80^{-n})$ depending only on $n$ so that
		\begin{equation*}
			\pazocal{L}^{n+1}(G) := \pazocal{L}^{n+1}\big(\big\{ \oy\in B_0\cap E \, :\, M\mu_0(\oy)\leq 4^{-n} \big\}\big)>0.
		\end{equation*}
		We will now prove that for all sufficiently small $\lambda>0$, $\mu_0+\lambda \pazocal{L}^{n+1}|_G$ belongs to $\Sigma_n(E)$, i.e.
		\begin{equation*}
			M\big(\mu_0+\lambda \pazocal{L}^{n+1}|_G\big) \leq 1, \quad \text{on }\, E\, \text{ if }\, \lambda\in (0,\lambda_0).
		\end{equation*}
		Let us fix $\oz\in E, r>0$ and write $\mu_\lambda := \mu_0+\lambda \pazocal{L}^{n+1}|_G$. Distinguish three cases: first, if $B_r(\oz)\cap G = \varnothing$ then it is clear that
		\begin{equation*}
			\frac{\mu_{\lambda}(B_r(\oz))}{r^n} = \frac{\mu_{0}(B_r(\oz))}{r^n}\leq 1,
		\end{equation*}
		that is the desired estimate. If $B_r(\oz)\cap G \neq \varnothing$ and $\oz\in G$, then
		\begin{align*}
			\frac{\mu_{\lambda}(B_r(\oz))}{r^n} = \frac{\mu_{0}(B_r(\oz))}{r^n}+\lambda \frac{\pazocal{L}^{n+1}(B_r(\oz)\cap G)}{r^n} \leq \frac{1}{4^n} +\lambda \frac{\min\big\{r^{n+1},R_0^{n+1}\big\}}{r^n}.
		\end{align*}
		If $r\leq R_0$, then
		\begin{equation*}
			\frac{\mu_{\lambda}(B_r(\oz))}{r^n} \leq \frac{1}{4^n}+\lambda r\leq \frac{1}{4^n}+\lambda R_0 < 1, \quad \text{for $\lambda$ small enough}.
		\end{equation*}
		If on the other hand $r>R_0$, we also have
		\begin{equation*}
			\frac{\mu_{\lambda}(B_r(\oz))}{r^n} \leq \frac{1}{4^n}+\lambda \frac{R_0^{n+1}}{r^{n}}\leq \frac{1}{4^n}+\lambda R_0 < 1, \quad \text{for $\lambda$ small enough}.
		\end{equation*}
		Finally, for the third case, that is, if $B_r(\oz)\cap G \neq \varnothing$ and $\oz\not\in G$, take $\oy\in B_r(\oz)\cap G$ so that
		\begin{align*}
			\frac{\mu_{\lambda}(B_r(\oz))}{r^n} &= \frac{\mu_{0}(B_r(\oz))}{r^n}+\lambda \frac{\min\big\{r^{n+1},R_0^{n+1}\big\}}{r^n}\leq 2^n\frac{\mu_{0}(B_{2r}(\oy))}{(2r)^n}+\lambda R_0 \\
			&\leq \frac{1}{2^n}+(2^n+1)\lambda R_0,
		\end{align*}
		where for the last inequality we have applied the reasoning of the second case. Hence, in general, we deduce that for $\lambda>0$ small enough, $\mu_\lambda\in \Sigma_n(E)$.
        
		Now we argue as follows: choose $\varepsilon_1^{(k)}\leq \varepsilon_1$ a sequence converging to 0. For each $k$ set $R_0^{(k)}:=\varepsilon_1^{(k)}/10$ and build a set $G_k$ in $B_{R_0^{(k)}}(\ox_0)$ as above. Let $H_k:=\pazocal{L}^{n+1}|_{G_k}$ and apply Lemma \ref{lem4.1}, or in this particular case \eqref{eq4.1}, to get
        
		\begin{equation*}
			1\leq \frac{\mu_0(E)+2\int_E|\pSS\mu_0|^2\dd\mu_0}{\mu_0(E)} \leq \frac{1}{H_k(E)} \int_E \Big( |\pSS\mu_0|^2+2\pSS_{\mu_0}(\pSS\mu_0) \Big)\dd H_k.
		\end{equation*}
		Since the measures $H_k/H_k(E)$ converge weakly to a point mass probability measure at $\ox_0$, by the continuity of the kernel associated to $\pSS$ we have
		\begin{equation*}
			1\leq |\pSS\mu_0|^2(\ox_0)+2\pSS_{\mu_0}(\pSS\mu_0)(\ox_0) \quad \rightarrow \quad \frac{1}{2}\leq |\pSS\mu_0|^2(\ox_0)+\pSS_{\mu_0}(\pSS\mu_0)(\ox_0).
		\end{equation*}
		Therefore, either $|\pSS\mu_0|^2(\ox_0)\geq 1/4$ or $\pSS_{\mu_0}(\pSS\mu_0)(\ox_0)\geq 1/4$, so in any case we get
		\begin{equation*}
			\pSS\mu_0(\ox_0)+\pSS_{\mu_0}(\pSS\mu_0)(\ox_0)\geq \frac{1}{4},
		\end{equation*}
		by the nonnegativity of each term. Since this last estimate follows if $M\mu_0(\ox_0)$ is smaller than $\alpha_1$, chosen to depend only on $n$, setting $\alpha_0:=\min\{\alpha_1,1/4\}$ the result follows.
	\end{proof}
	
	\subsection{The construction}
	\label{subsec4.1}
	Both of the above lemmas are enough to carry out the construction of the desired cubes in a similar way  to \cite[pp. 38-42]{Vo}. Fix $E\subset B_1(0)\subset 2Q_0$ compact consisting of a finite union of cubes, and define the auxiliary potential
	\begin{equation*}
		\widetilde{\pazocal{U}}^{\mu_0}(\ox) := \pazocal{U}^{\mu_0}(\ox)+M(\pSS\mu_0\dd \mu_0)(\ox).
	\end{equation*}
	Apply Lemma \ref{lem4.2} to $E$ and obtain its corresponding constant $\alpha_0$, and pick $0<\beta \ll 1$ depending only on $n$, to be fixed later on, and set
	\begin{equation*}
		\pazocal{G}:=\big\{ \oy\in \mathbb{R}^{n+1}\,:\, \widetilde{\pazocal{U}}^{\mu_0}(\oy)>\beta \alpha_0 \big\} \supset E.
	\end{equation*}
	Consider $\{Q_j\}_{j}$ a Whitney decomposition of $\pazocal{G}$ (see \cite[pp. 167-169]{St}, for example), that is: a countable family of closed dyadic cubes with disjoint interiors that cover $\pazocal{G}$ and such that for some constant $A>0$,
	\begin{itemize}[nolistsep]
		\item $20 Q_j \subset \pazocal{G}$ for each $j$,
		\item $(AQ_j)\cap \pazocal{G}^c \neq \varnothing$ for each $j$,
		\item The family $\{10Q_j\}_{j}$ has bounded overlap. Moreover, if $10Q_j \cap 10Q_i \neq \varnothing$, then $\ell(Q_j) \approx \ell(Q_i)$.
	\end{itemize}
	From those $Q_j$ satisfying $\frac{5}{4}Q_j\cap E \neq \varnothing$, choose a finite subcovering and enumerate them as $Q_{j_1},\ldots,Q_{j_N}$. We shall check that
	\begin{equation*}
		\QQ_i := 2Q_{j_i}, \quad i=1,\ldots,N
	\end{equation*}
	is the desired family of cubes that satisfies the properties in Theorem \ref{thm4.1}. It is clear that properties $\Pone$ and $\Pfive$ are satisfied by construction, so we are only left to verify $\Ptwo$, $\Pthree$ and $\Pfour$. The first two will be checked for $\gamma_{\text{sy},+}$ instead of $\widetilde{\gamma}_+$, which is enough due to the comparability of both capacities.
	\begin{proof}[\textit{Proof of $\Ptwo$} in Theorem \ref{thm4.1}.]
		We shall prove that property $\widetilde{\pazocal{U}}^{\mu_0}(\ox)\geq \beta\alpha_0$ on $\FF$, implies
		\begin{equation*}
			\gamma_{\text{sy},\text{op}}(\FF) \leq \frac{C}{\beta \alpha_0} \mu_0(E),
		\end{equation*}
		where $\mu_0$ is an extremal measure for $\gamma_{\text{var}}(E)$ and $C>0$ constant. If this is the case, we are done, since we have
		\begin{equation*}
			\gamma_{\text{sy},+}(\FF) \approx \gamma_{\text{sy},\text{op}}(\FF) \leq \frac{C}{\beta\alpha_0} \gamma_{\text{var}}(E) \leq  \frac{C}{\beta\alpha_0} \gamma_{\text{sy},\psi_{\tau_0},2}(E) \leq \frac{C}{\beta\alpha_0}\,2C_1C_2 \gamma_{\text{sy},\text{op}}(E) \approx \gamma_{\text{sy},+}(E),
		\end{equation*}
		where for the third estimate recall Remark \ref{rem3.4} and the fact that relation \eqref{eq3.10} holds in general, and in the fourth inequality we have used the definition of $\tau_0$.
        
		So let us choose $\mu$ admissible for $\gamma_{\text{sy},\text{op}}(\FF)$ with $\gamma_{\text{sy},\text{op}}(\FF)\leq 2\mu(\FF)$ and observe that
		\begin{equation*}
			\|\pSS_\mu\|_{L^{2}(\mu)\to L^2(\mu)} \leq \|\PP_{\text{sy},\mu}\|_{L^{2}(\mu)\to L^2(\mu)}\leq 1,
		\end{equation*}
		since $\Psy$ and the measure $\mu$ are nonnegative. Therefore, $\pSS_\mu$ becomes a CZ convolution operator (with continuous kernel), so that it satisfies the weak estimate
		\begin{equation*}
			\mu \big( \{ \pSS\mu_0 \geq \lambda \} \big) \leq C'\frac{\mu_0(E)}{\lambda},
		\end{equation*}
		where $C'>0$ is a constant. By definition, for each $\ox\in \FF$,
		\begin{align*}
			\widetilde{\pazocal{U}}^{\mu_0}(\ox)&:= M\mu_0(\ox)+\pSS\mu_0(\ox)+\pSS_{\mu_0}(\pSS\mu_0)(\ox)+M(\pSS\mu_0\dd \mu_0)(\ox)\\
			&=: \text{\rom{1}}(\ox)+\text{\rom{2}}(\ox)+\text{\rom{3}}(\ox)+\text{\rom{4}}(\ox),
		\end{align*}
		so at least one of the four terms is larger than $\frac{\beta\alpha_0}{4}=:a$. In fact, if we name
		\begin{align*}
			\FF_1:=\big\{ \ox\in F\,:\, \text{\rom{1}}(\ox)\geq a \big\}, &\qquad \FF_2:=\big\{ \ox\in F\,:\, \text{\rom{2}}(\ox)\geq a \big\},\\
			\FF_3:=\big\{ \ox\in F\,:\, \text{\rom{3}}(\ox)\geq a \big\}, &\qquad \FF_4:=\big\{ \ox\in F\,:\, \text{\rom{4}}(\ox)\geq a \big\},
		\end{align*}
		we have $\FF=\cup_{i=1}^4 \FF_i$. So in particular $\mu(\FF)\leq \sum_{i=1}^4 \mu(\FF_i)$ and then, necessarily, there exists $i\in \{1,2,3,4\}$ such that $\mu(\FF_i)\geq \mu(\FF)/4$. Let us study separately each possibility:
		\begin{enumerate}[leftmargin=*]
			\item[\textit{a})] If $i=1$. In this case $M\mu_0\geq a$ on $\FF_1$ with $\mu(\FF_1)\geq \mu(\FF)/4$. For each $\ox\in \FF_1$ choose a ball $B_{\ox}:=B_{r(\ox)}(\ox)$ so that $\mu_0(B_{\ox})\geq a r(\ox)^n$. Apply Besicovitch's covering theorem \cite[Theorem 2.7]{Ma} to extract a countable collection of balls $\{B_{\ox_j}\}_j$ with bounded multiplicity of overlapping (with constant depending only on the dimension) that covers $\FF_1$. Then,
			\begin{align*}
				\hspace{1.5cm}\frac{\mu(\FF)}{4} \leq \mu(\FF_1) \leq \sum_j \mu(B_{\ox_j}) \leq \sum_j r^n(\ox_j)\leq \frac{1}{a}\sum_j \mu_0(B_{\ox_j}) \lesssim \frac{\mu_0(\FF_1)}{a}\leq \frac{\mu_0(E)}{a},
			\end{align*}
			since $\mu_0$ is supported on $E$.
			\item[\textit{b})] If $i=2$. Now $\pSS\mu_0\geq a$ on $\FF_2$ with $\mu(\FF_2)\geq \mu(\FF)/4$. Using the weak estimate presented above we get the desired estimate
			\begin{equation*}
				\frac{\mu(\FF)}{4}\leq \mu(\FF_2) =  \mu \big( \{ \pSS\mu_0 \geq a \} \big) \leq \frac{C'}{a}\mu_0(E).
			\end{equation*}
			\item[\textit{c})] If $i=3$. Now $\pSS_{\mu_0}(\pSS\mu_0)\geq a$ on $\FF_3$ with $\mu(\FF_3)\geq \mu(\FF)/4$. By \eqref{eq3.10} we obtain
			\begin{align*}
				\frac{\mu(\FF)}{4}\leq \mu(\FF_3) &= \mu \big( \{ \pSS_{\mu_0}(\pSS\mu_0) \geq a \} \big) \leq \frac{C'}{a}\int_E \pSS\mu_0\dd \mu_0 \\
				&\leq \frac{C'}{a}\mu_0(E)^{1/2}\bigg(\int_E (\pSS\mu_0)^2\dd \mu_0\bigg)^2 \leq \frac{C'}{a}\mu_0(E).
			\end{align*}
			\item[\textit{d})] If $i=4$. Now $M(\pSS\mu_0\dd \mu_0)\geq a$ on $\FF_4$ with $\mu(\FF_4)\geq \mu(\FF)/4$. Proceeding analogously as in the first case, we deduce
			\begin{equation*}
				\frac{\mu(\FF)}{4} \lesssim \frac{1}{a}\int_E \pSS\mu_0\dd \mu_0,
			\end{equation*}
			and arguing as in the third case we obtain the desired result.
		\end{enumerate}
	\end{proof}
	To prove $\Pthree$ we will need the following lemma:
	\begin{lem}
		\label{lem4.3}
		For $i=1,\ldots, N$ define $\mu_i:=\mu_0|_{5Q_{j_i}}$ and consider the auxiliary potential
		\begin{equation*}
			\pazocal{W}^{\mu_i}:= M\mu_i+\pSS_{\mu_i}1+\pSS_{\mu_i}(\pSS\mu_0)+M(\pSS\mu_0\dd \mu_i).
		\end{equation*}
		Then, there exists a constant $\alpha_0'>0$ such that for each $i=1,\ldots, N$,
		\begin{equation*}
			\pazocal{W}^{\mu_i}(\ox)\geq \alpha_0'\alpha_0, \qquad \forall \ox\in 4Q_{j_i}\cap E.
		\end{equation*}
	\end{lem}
	\begin{proof}
		For the sake of notation, in this proof we rename $Q_{i}:=Q_{j_i}$. Fix $i=1,\ldots,N$, $\ox\in 4Q_{i}\cap E$ and $\oz\in (AQ_i)\cap \pazocal{G}^c$, which is non-empty by construction. So in particular
		\begin{equation}
			\label{eq4.3}
			M\mu_0(\oz)\leq \beta\alpha_0, \quad \pSS\mu_0(\oz)\leq \beta\alpha_0.
		\end{equation}
		Now choosing, for example, $R:=2(4+A)\text{diam}(Q_i)$ it is clear that we have
		\begin{equation*}
			|\ox-\oz|\leq (4+A)\text{diam}(Q_i)\leq \frac{R}{2}.
		\end{equation*}
		Therefore, using the CZ estimates for $\Psy\psi_{\tau_0}$, the nonnegativity of the latter kernel and relations \eqref{eq4.3} we have
		\begin{align}
			\nonumber
			\pSS_{R}\mu_0(\ox)& := \int_{|\ox-\oy|>R}\Psy(\ox-\oy)\psi_{\tau_0}(\ox-\oy)\dd\mu_0(\oy)\\
			\nonumber
			&\leq \beta\alpha_0+\int_{|\ox-\oy|>R}\Big\rvert \Psy(\ox-\oy)\psi_{\tau_0}(\ox-\oy)-\Psy(\oz-\oy)\psi_{\tau_0}(\oz-\oy)\Big\rvert\dd\mu_0(\oy)\\
			\nonumber
			&\lesssim \beta\alpha_0+|\ox-\oz|\int_{|\ox-\oy|>R}\frac{\dd\mu_0(\oy)}{|\ox-\oy|^{n+1}} \lesssim \beta\alpha_0+|\ox-\oz|\int_{|\oz-\oy|>R/4}\frac{\dd\mu_0(\oy)}{|\oz-\oy|^{n+1}} \\
			&\leq \beta\alpha_0+\frac{R}{2} \sum_{j=0}^{\infty}\frac{\mu_0(B_{2^{j+1}R}(\oz))}{(2^jR)^{n+1}} \leq \beta\alpha_0+\frac{R}{2} \sum_{j=0}^{\infty}\frac{\beta\alpha_0(2^{j+1}R)^n}{(2^jR)^{n+1}}\lesssim \beta\alpha_0. \label{eq4.4}
		\end{align}
		We also notice that there exists $\widetilde{A}=\widetilde{A}(n)>0$ such that
		\begin{equation}
			\label{eq4.5}
			M\mu_0(\ox)\leq \max\big\{\widetilde{A}\,\beta\alpha_0, M\mu_i(\ox)\big\}.    
		\end{equation}
		Indeed, if $r\geq \ell(Q_i)/20$ then
		\begin{equation*}
			\frac{\mu_0(B_r(\overline{x}))}{r^n}\leq \frac{\mu_0(B_{40\sqrt{n}Ar}(\oz))}{r^n}\leq \widetilde{A}(n)\beta\alpha_0.
		\end{equation*}
		If $r< \ell(Q_i)/20$ use that $\ox\in 4Q_i$ and $\mu_i:=\mu_0|_{5Q_i}$ to simply have
		\begin{equation*}
			\frac{\mu_0(B_r(\overline{x}))}{r^n} =  \frac{\mu_j(B_r(\overline{x}))}{r^n},
		\end{equation*}
		so indeed $ M\mu_0(\ox)\leq \max\big\{\widetilde{A}\,\beta\alpha_0, M\mu_i(\ox)\big\}$. Having made this observations, we recall that by Lemma \ref{lem4.2} we always have
		\begin{equation*}
			M\mu_0+\pSS\mu_0+\pSS_{\mu_0}(\pSS\mu_0)\geq \alpha_0, \quad \text{on }\, E.
		\end{equation*}
		Now we distinguish three cases:
		\begin{enumerate}[leftmargin=*]
			\item[\textit{a})] If $M\mu_0(\ox)\geq \alpha_0/3$. Then \eqref{eq4.5} implies
			\begin{equation*}
				\frac{\alpha_0}{3}\leq \max\big\{\widetilde{A}\,\beta\alpha_0, M\mu_i(\ox)\big\},
			\end{equation*}
			so choosing $\beta \leq (6\widetilde{A})^{-1}$ we deduce $M\mu_i(\ox)\geq \alpha_0/6$ and we would be done.
			\item[\textit{b})] If $\pSS\mu_0(\ox)\geq \alpha_0/3$, by \eqref{eq4.4} there exists $C>0$ constant such that
			\begin{equation*}
				\pSS\big(\mu|_{B_R(\ox)}\big)(\ox) \geq \frac{\alpha_0}{3}-C\beta\alpha_0 \geq \frac{\alpha_0}{4},
			\end{equation*}
			for $\beta$ small enough. In addition, by definition of $R$,
			\begin{equation*}
				\pSS\big(\mu|_{B_R(\ox)\setminus{5Q_i}}\big)(\ox) \leq C' \frac{\mu_0(B_{2R}(\oz))}{\ell(Q_i)^n}\leq 2^nC'\beta\alpha_0 \frac{R^n}{\ell(Q_i)^n} = C''\beta\alpha_0.
			\end{equation*}
			Hence
			\begin{equation*}
				\pSS\mu_i(\ox) = \pSS\big(\mu_0|_{5Q_i}\big)(\ox)\geq \frac{\alpha_0}{4}-C''\beta\alpha_0 \geq \frac{\alpha_0}{5},
			\end{equation*}
			for $\beta$ small enough, and we are done.
			\item[\textit{c})] If $\pSS_{\mu_0}(\pSS\mu_0)(\ox)\geq \alpha_0/3$. Let us observe that for $\oz\in (AQ_i)\cap \pazocal{G}^c$, by definition of $\pazocal{G}$ we also have
			\begin{equation*}
				M(\pSS\mu_0\dd\mu_0)(\oz)\leq \beta\alpha_0,\qquad \pSS_{\mu_0}(\pSS\mu_0)(\oz)\leq \beta\alpha_0.
			\end{equation*}
			So choosing again $R:=2(4+A)\text{diam}(Q_i)$ we would be able to prove analogously as in \eqref{eq4.4} that
			\begin{equation*}
				\pSS_{\mu_0,R}(\pSS\mu_0)(\ox)\leq C\beta\alpha_0,
			\end{equation*}
			for some $C>0$ constant. Moreover,
			\begin{align*}
				\pSS_{\mu_0|_{B_R(\ox)\setminus{5Q_i}}}(\pSS\mu_0)(\ox)\leq C' \frac{\int_{B_{2R}(\oz)}\pSS\mu_0\dd\mu_0}{\ell(Q_i)^n} \leq C''\beta\alpha_0,
			\end{align*}
			where we have used, again, the definition of $R$ and that $M(\pSS\mu_0\dd\mu_0)(\oz)\leq \beta\alpha_0$. Therefore,
			\begin{align*}
				\pSS_{\mu_i}(\pSS\mu_0)(\ox)&=\pSS_{\mu_0}(\pSS\mu_0)(\ox)-\pSS_{\mu_0|_{B_R(\ox)\setminus{5Q_i}}}(\pSS\mu_0)(\ox)-\pSS_{\mu_0,R}(\pSS\mu_0)(\ox)\\
				&\geq \frac{\alpha_0}{3}-C''\beta\alpha_0-C\beta\alpha_0\geq \frac{\alpha_0}{5},
			\end{align*}
			for $\beta$ small enough, and the proof is completed.
		\end{enumerate}
	\end{proof}
	\begin{proof}[\textit{Proof of $\Pthree$ in Theorem \ref{thm4.1}}.]
		Repeating the same arguments presented for the proof of $\Ptwo$, one deduces that property $\pazocal{W}^{\mu_i}(\ox)\geq \alpha_0'\alpha$, for every $\ox \in 4Q_{j_i}\cap E$ and every $i=1,\ldots, N$, implies
		\begin{equation*}
			\gamma_{\text{sy},\text{op}}(4Q_{j_i}\cap E) \leq \frac{C}{\alpha_0'\alpha_0} \bigg( \mu_0(5Q_{j_i})+\int_{5Q_{j_i}}\pSS\mu_0\dd\mu_0 \bigg).
		\end{equation*}
		But notice that the bounded multiplicity of overlapping of the family $\{10Q_j\}_j$ and the fact that $\mu_0$ is supported on $E$ yield
		\begin{equation*}
			\sum_{i=1}^N\bigg(\mu_0(Q_{j_i})+\int_{5Q_{j_i}}\pSS\mu_0\dd\mu_0\bigg) \lesssim \mu_0(E) + \int_{E}\pSS\mu_0\dd\mu_0 \leq 2\mu_0(E),
		\end{equation*}
		where for the last inequality we have used \eqref{eq3.10}. Therefore, by the latter relation and using the definition of $\tau_0$ we get
		\begin{align*}
			\sum_{i=1}^N \gamma_{\text{sy},+}(4Q_{j_i}\cap E) &\lesssim \sum_{i=1}^N \gamma_{\text{sy},\text{op}}(4Q_{j_i}\cap E) \lesssim \mu_0(E) \\
			&= \gamma_{\text{var}}(E) \leq \gamma_{\text{sy},\psi_{\tau_0},2}(E)\leq 2C_1C_2 \gamma_{\text{sy},\text{op}}(E) \lesssim \gamma_{\text{sy},+}(E),
		\end{align*}
		that is the desired estimate.
	\end{proof}
	
	Finally, we prove $\Pfour$. Let us first observe that the assumption
	\begin{equation*}
		\widetilde{\gamma}_+(E)\leq C_+\text{diam}(E)^n, \quad \text{ for $C_+>0$ constant to be fixed below},
	\end{equation*}
	is not at all restrictive. Indeed, if it failed, \cite[Lemma 4.1]{MPr} would imply
	\begin{equation*}
		\widetilde{\gamma}(E)\leq \gamma(E) \lesssim \pazocal{H}_\infty^n(E)\leq \text{diam}(E)^n < \frac{1}{C_+}\widetilde{\gamma}_+(E),
	\end{equation*}
	and we would be done.
	\begin{proof}[\textit{Proof of $\Pfour$ in Theorem \ref{thm4.1}}.]
		Fix $\ox\not\in E$ and observe that the following estimates hold
		\begin{align*}
			\pSS\mu_0(\ox) &\leq \frac{\mu_0(E)}{\text{dist}(\ox,E)^n},\\
			\pSS_{\mu_0}(\pSS\mu_0)(\ox) &\leq  \frac{1}{\text{dist}(\ox,E)^n}\int_E\int_E\Psy(\oy-\oz)\psi_{\tau_0}(\oy-\oz)\dd\mu_0(\oz)\dd\mu_0(\oy)\\
			&=\frac{1}{\text{dist}(\ox,E)^n}\int_E \pSS\mu_0(\oy) \dd\mu_0(\oy)\leq \frac{\mu_0(E)}{\text{dist}(\ox,E)^n},
		\end{align*}
		where for the last inequality we have applied \eqref{eq3.10}. Since $\text{supp}(\mu_0)\subset E$, we have
		\begin{equation*}
			\sup_{r>0}\frac{\mu_0(B_r(\ox))}{r^n} = \sup_{r>\text{dist}(\ox,E)}\frac{\mu_0(B_r(\ox))}{r^n} \leq \frac{\mu_0(E)}{\text{dist}(\ox,E)^n},
		\end{equation*}
		and moreover, due to \eqref{eq3.10} again,
		\begin{align*}
			\sup_{r>0}\frac{1}{r^n}\int_{B_r(\ox)}\pSS\mu_0(\oy)\dd \mu_0(\oy)&= \sup_{r>\text{dist}(\ox,E)}\frac{1}{r^n}\int_{B_r(\ox)}\pSS\mu_0(\oy)\dd \mu_0(\oy)\\
			&\leq \frac{1}{\text{dist}(\ox,E)^n}\int_{E}\pSS\mu_0(\oy)\dd \mu_0(\oy) \leq \frac{\mu_0(E)}{\text{dist}(\ox,E)^n}.
		\end{align*}
		All in all, we get
		\begin{equation*}
			\widetilde{\pazocal{U}}^{\mu_0}(\ox)\leq \frac{4\mu_0(E)}{\text{dist}(\ox,E)}.
		\end{equation*}
		Let us pick $\ox\in \pazocal{G}\cap E^c$ (which is non empty, since $E\subset \pazocal{G}$ with $\pazocal{G}$ open, by the lower semicontinuity of the potential $\widetilde{\pazocal{U}}^{\mu_0}$) so that $\widetilde{\pazocal{U}}^{\mu_0}(\ox)>\beta \alpha_0$ and observe that using the definition of $\mu_0$ and $\tau_0$ as in the previous proofs, we get
		\begin{align*}
			\text{dist}(\ox,E)^n &\leq \frac{8C_1C_2}{\alpha_0\beta}\gamma_{\text{sy},\text{op}}(E)\leq \frac{8CC_1C_2}{\alpha_0\beta}\gamma_{\text{sy},+}(E)\\
			&\leq \frac{16CC_1C_2}{\alpha_0\beta}\widetilde{\gamma}_{+}(E) \leq \frac{16CC_1C_2}{\alpha_0\beta}C_+\text{diam}(E)^n.
		\end{align*}
		Choosing $C_+$ appropriately to neglect the effect of all the above constants that have already been fixed, the result follows. Indeed, this is because the previous estimate is valid for all $\ox \in \pazocal{G}\cap E^c$, implying that one can make $\partial \pazocal{G}$ say $\frac{\text{diam}(E)}{1000}$ close to $E$. Since we also know that $20Q_j\subset \pazocal{G}$ we have
		\begin{equation*}
			\text{diam}(2Q_j)\leq \frac{1}{10}\text{dist}(\partial\pazocal{G},Q_j),
		\end{equation*}
		where $\text{dist}(\partial\pazocal{G},Q_j)$ is comparable to $\text{dist}(\partial\pazocal{G},E)$ since the cubes $Q_j$ have been also chosen so that $\frac{5}{4}Q_j\cap E\neq \varnothing$.
	\end{proof}
	
	\section{A comparability result under an additional assumption}
	\label{sec5}
	
	Let $E\subset \mathbb{R}^{n+1}$ be a compact set. As it is stated in $\Pfour$ in Theorem \ref{thm4.1}, we shall work under the following additional assumption:
	\begin{enumerate}
		\item[$\Aone$:] $\widetilde{\gamma}_+(E)\leq C_+\text{diam}(E)^n$, with $C_+>0$ constant.
	\end{enumerate}
	By the translation invariance and $\widetilde{\gamma}(\lambda E) = \lambda^n\widetilde{\gamma}(E)$, $\widetilde{\gamma}_+(\lambda E) = \lambda^n\widetilde{\gamma}_+(E)$ (see \cite[Proposition 3.1]{He}), it is clear that we may assume $E\subset B_1(0)$ without loss of generality. In fact, in order to apply Theorem \ref{thm4.1} let us check that we can assume:
	\begin{itemize}
		\item[$\Atwo$:] $E$ \textit{is contained in the unit ball and consists of a finite union of closed dyadic cubes belonging to a dyadic grid in $\mathbb{R}^{n+1}$ $($with sides parallel to the coordinate axes$\,)$, all with the same size and with disjoint interiors.}
	\end{itemize}
	Let us verify that if we deduce the comparability between $\widetilde{\gamma}$ and $\widetilde{\gamma}_+$ for $E$ satisfying $\Atwo$, we obtain the same result for a general $E$. So fix $E\subset \mathbb{R}^{n+1}$ any compact set and let $\overline{\pazocal{U}_\delta(E)}$ be the closed $\delta$-neighborhood of $E$. Consider a grid of dyadic cubes in $\mathbb{R}^{n+1}$ with sides parallel to the axes and diameter smaller than $\delta/2$. Name $E_0$ the collection of dyadic cubes that intersect $E$ and notice that $E\subset E_0 \subset \overline{E_0} \subset \overline{\pazocal{U}_\delta(E)}$. Now we would have
	\begin{equation*}
		\widetilde{\gamma}(E)\leq \widetilde{\gamma}(\overline{E_0})\lesssim \widetilde{\gamma}_+(\overline{E_0})\leq \widetilde{\gamma}_+(\overline{\pazocal{U}_\delta(E)}) \leq \gamma_+(\overline{\pazocal{U}_\delta(E)}).
	\end{equation*}
	Letting $\delta\to 0$ and using the outer regularity of $\gamma_+$ we deduce, by Theorem \ref{thm3.5},
	\begin{equation*}
		\widetilde{\gamma}(E) \lesssim \gamma_+(E) \lesssim \widetilde{\gamma}_+(E),
	\end{equation*}
	and the result follows.
    
	Finally, we present a last assumption that motivates the title of this section:
	\begin{enumerate}
		\item[$\Athree$.] \textit{Let} $\{\QQ_1,\ldots, \QQ_N\}$ \textit{be the family of cubes provided by Theorem} \ref{thm4.1} \textit{and} $C_1$ \textit{the constant appearing in point $\Pthree$ of the same theorem. We will assume that}
		\begin{equation*}
			\widetilde{\gamma}(E)\geq C_1^{-1} \sum_{i=1}^N \widetilde{\gamma}(2\QQ_i\cap E).
		\end{equation*}
	\end{enumerate}
	The aim of this section is to prove the comparability between $\widetilde{\gamma}$ and $\widetilde{\gamma}_+$ under this last assumption $\Athree$.

	\subsection{Basic definitions and properties}
	\label{subsec5.1}
	Let $E\subset \mathbb{R}^{n+1}$ be a compact set verifying $\Aone, \Atwo$ and $\Athree$. We begin our argument by fixing a distribution $T_0$ admissible for $\widetilde{\gamma}(E)$ so that $|\langle T_0, 1 \rangle|=\widetilde{\gamma}(E)/2$. We know that $T_0$ has $n$-growth with constant 1 and
	\begin{equation*}
		\|P\ast T_0\|_\infty\leq 1, \qquad \|P^\ast\ast T_0\|_\infty\leq 1.
	\end{equation*}
	Let $\FF:=\cup_{i=1}^N\QQ_i$ be the covering by cubes provided by Theorem \ref{thm4.1}. Recall that $\QQ_i=2Q_i$, where $\{Q_i\}_{i=1}^N$ is the finite subset of dyadic cubes with disjoint interiors of a Whitney decomposition of $\pazocal{G}\supset E$, with the property $\frac{5}{4}Q_i\cap E \neq \varnothing$. We write
	\begin{equation*}
		F:= \bigcup_{i=1}^N Q_i\subset \FF.
	\end{equation*}
	It is clear that $E\subset F$ and moreover, by monotonicity and the properties in Theorem \ref{thm4.1} we also have
	\begin{enumerate}[itemsep=0.1cm]
		\item[$\Ptwo'$.] $\widetilde{\gamma}_+(F)\leq C_0 \widetilde{\gamma}_+(E)$,
		\item[$\Pthree'$.] $\sum_{i=1}^N\widetilde{\gamma}_+(2Q_i\cap E)\leq C_1 \widetilde{\gamma}_+(E)$ and by $\Athree$ we also have $\sum_{i=1}^N\widetilde{\gamma}(2Q_i\cap E)\leq C_1 \widetilde{\gamma}(E)$.
		\item[$\Pfour'$.] $\text{\normalfont{diam}}(Q_i)\leq \frac{1}{20}\text{\normalfont{diam}}(E)$, for each $i=1, \ldots, N$,
		\item[$\Pfive'$.] The family $\{10Q_1,\ldots, 10Q_N\}$ has bounded overlap. Moreover, if $10Q_j \cap 10Q_i \neq \varnothing$, then $\ell(Q_j) \approx \ell(Q_i)$.
	\end{enumerate}
	In order to simplify the arguments, in this section we will work with the cubes $Q_1,\ldots, Q_N$ instead of $\QQ_1,\ldots, \QQ_N$, and with $F$ instead of $\FF$. Let us choose for each $Q_i$ a ball $B_i\subset \frac{1}{2}Q_i$ concentric with $Q_i$ with radius $r_i$ comparable to $\widetilde{\gamma}(2Q_i\cap E)^{1/n}$. Notice that
	\begin{equation*}
		\widetilde{\gamma}(2Q_i\cap E)^{1/n}\leq \gamma(2Q_i)^{1/n} \approx \ell(Q_i).
	\end{equation*}
	Therefore, we may choose $r_i\approx \widetilde{\gamma}(2Q_i\cap E)^{1/n}$ with constants depending at most on the dimension and still have $B_i\subset \frac{1}{2}Q_i$. Notice that $\text{dist}(B_i, B_j) \geq \frac{1}{2}\min\{\ell(Q_i),\ell(Q_j)\}$, for $i\neq j$. We define the positive measure
	\begin{equation}
		\label{eq5.1}
		\mu:=\sum_{i=1}^N \mu_i := \sum_{i=1}^N \frac{r_i^n}{\pazocal{L}^{n+1}(B_i)}\pazocal{L}^{n+1}|_{B_i}.
	\end{equation}
	Apply \cite[Lemma 3.1]{HPo} to pick test functions $\varphi_i\in \pazocal{C}^{\infty}_c(2Q_i), 0\leq \varphi_i \leq 1$, satisfying $\|\nabla \varphi_i\|_\infty\leq \ell(2Q_i)^{-1}$ and with $\sum_{i=1}^{N}\varphi_i \equiv 1$ in $F$. We also define the signed measure
	\begin{equation}
		\label{eq5.2}
		\nu := \sum_{i=1}^N \frac{\langle T_0, \varphi_i \rangle}{\pazocal{L}^{n+1}(B_i)}\pazocal{L}^{n+1}|_{B_i}, \qquad \text{that is} \qquad  \nu = \sum_{i=1}^N \frac{\langle T_0, \varphi_i \rangle}{r_i^n}\mu_i.
	\end{equation}
	Since $\text{supp}(\nu)\subset \text{supp}(\mu)\subset \FF$ and $\nu\ll \mu$, have $\nu = b\mu$, where $b$ is the Radon-Nikodym derivative.
	\begin{lem}
		\label{lem5.1}
		The following hold for some positive constants $c_1,c_3$ and $c_4$:
		\begin{enumerate}[leftmargin=*,itemsep=0.1cm]
			\item[\textit{1}.] $\|b\|_\infty \leq c_1$.
			\item[\textit{2}.] $|\nu(F)|=\widetilde{\gamma}(E)/2$.
			\item[\textit{3}.] $\widetilde{\gamma}(E)\leq c_3 \mu(F)$.
			\item[\textit{4.}] $|\nu(Q)|\leq c_4\ell(Q)^n$ for any cube $Q\subset \mathbb{R}^{n+1}$.
		\end{enumerate}
	\end{lem}
	\begin{proof}
		Let us prove \textit{1}. By the localization result \cite[Theorem 3.1]{MPr} we get $\|P\ast (\varphi_i T_0)\|_\infty \leq A_1$, for every $i$, where $A_1=A_1(n)>0$. Analogously, the same localization result holds for $P^\ast$, so there also exists $A_2=A_2(n)>0$ with $\|P^\ast \ast (\varphi_i T_0)\|_\infty \leq A_2$, $\forall i$. Therefore, since $\text{supp}(\varphi_iT_0)\subset 2Q_i\cap E$ we have
		\begin{equation}
			\label{eq5.3}
			|\langle T_0, \varphi_i \rangle | = |\langle \varphi_i T_0, 1 \rangle| \lesssim \widetilde{\gamma}(2Q_i\cap E) \approx r_i^n, \quad \forall i=1,\ldots,n,
		\end{equation}
		which yields the desired result.
        
		To prove \textit{2} we simply use the definition of $\nu$, $B_i\cap F = B_i$ and $\text{supp}(T_0)\subset E \subset F$:
		\begin{align*}
			|\nu(F)|= \Bigg\rvert \sum_{i=1}^N \frac{\langle T_0, \varphi_i \rangle}{\pazocal{L}^{n+1}(B_i)} \int_{B_i} \dd \pazocal{L}^{n+1} \Bigg\rvert = \Bigg\rvert \Bigg\langle T_0, \sum_{i=1}^N \varphi_i \Bigg\rangle \Bigg\rvert = | \langle T_0, 1 \rangle |= \frac{1}{2}\widetilde{\gamma}(E).
		\end{align*}
		Finally, \textit{3} follows by the choice of radii $r_i$ and the admissibility of $\varphi_iT_0$ for $\widetilde{\gamma}(2Q_i\cap E)$,
		\begin{align*}
			\widetilde{\gamma}(E)=2| \langle T_0, 1 \rangle |\leq 2\sum_{i=1}^N | \langle T_0, \varphi_i \rangle | \lesssim \sum_{i=1}^N \widetilde{\gamma}(2Q_i\cap E) \approx \sum_{i=1}^N r_i^n = \mu(F).
		\end{align*}
		The proof of inequality \textit{4} can be followed analogously to that given for property $(g)$ of \cite[Lemma 6.8]{To3}. We refer the reader to this last reference for the detailed arguments.
	\end{proof}
	\begin{rem}
		\label{rem5.1}
		In the proof of inequality \textit{4} it is used that for a Borel signed measure $\nu$ satisfying $\|P\ast \nu\|_\infty\leq 1$ one has $|\nu(Q)|\leq C \ell(Q)^n$, for some constant $C>0$ and any cube $Q\subset \mathbb{R}^{n+1}$. In order to justify that this holds in our context we will prove the following: given $\varphi\in\pazocal{C}^\infty_c(\mathbb{R}^{n+1})$ such that $\|P\ast \varphi\|_\infty \leq 1$, then $|\int_Q \varphi\dd\pazocal{L}^{n+1}|\leq C\ell(Q)^n$, for some $C>0$ constant. If this is the case, we will be done by mollifying $\nu$ to be $\nu_\varepsilon := \nu\ast \psi_\varepsilon$ with $(\psi_\varepsilon)_\varepsilon$ a proper approximation of the identity. Let us estimate $|\int_Q \varphi\dd\pazocal{L}^{n+1}|$ as follows:
		\begin{align*}
			\bigg\rvert \int_Q \varphi(\ox) \dd\ox \bigg\rvert \leq \bigg\rvert \int_Q \partial_t (P\ast \varphi)(\ox) \dd\ox \bigg\rvert + \bigg\rvert \int_Q (\Delta_x)^{1/2} (P\ast \varphi)(\ox) \dd\ox \bigg\rvert
		\end{align*}
		For the first term of the right-hand side we set $Q=Q_1\times I_Q \subset \mathbb{R}^{n}\times\mathbb{R}$ where $I_Q$ is the \textit{temporal} interval, that we write as $I_Q=[a,b]$ for some $a,b\in \mathbb{R}$ with $a<b$. This way, by Fubini's theorem,
		\begin{align*}
			\bigg\rvert \int_Q \partial_t (P\ast \varphi)(\ox) \dd\ox \bigg\rvert
			&= \bigg\rvert \int_{Q_1} \Big( P\ast \varphi (x,b)- P\ast \varphi(x,a) \Big)\dd x \bigg\rvert \lesssim \|P\ast \varphi\|_{\infty} \pazocal{L}^n(Q_1)\leq \ell(Q)^n.
		\end{align*}
		For the second term we recall that we have the following representation
		\begin{equation*}
			(-\Delta)^{1/2} \approx \sum_{j=1}^n \partial_jR_j
		\end{equation*}
		where $R_j, \, 1\leq j \leq n$ are the Riesz transforms with Fourier multiplier $\xi_j/|\xi_j|$. Notice that if we prove that $P\ast (R_j\varphi) \in L^\infty(\mathbb{R}^{n+1})$ for each $j=1,\ldots, n$, we are done, since we may argue as above decomposing the cube $Q$ as a cartesian product of an interval and an $n$-dimensional cube. To study $P\ast (R_j\varphi)$ we consider its Fourier transform $\frac{\xi_j}{|\xi|}\widehat{P}\widehat{\varphi}$. Notice that $P\in (L^1+L^2)(\mathbb{R}^{n+1})$ (just decompose it as $P\chi_{|\ox|\leq 1}+P\chi_{|\ox|>1}$), and therefore $\widehat{P}\in (L^\infty+L^2)(\mathbb{R}^{n+1})$, and we write its decomposition as $\widehat{P}_\infty+\widehat{P}_2$. Now it is immediate to check that $\frac{\xi_j}{|\xi|}\widehat{P}_\infty\widehat{\varphi}\in L^1(\mathbb{R}^{n+1})$ and we also have $\frac{\xi_j}{|\xi|}\widehat{P}_2\widehat{\varphi}\in L^1(\mathbb{R}^{n+1})$ using Cauchy-Schwarz inequality. Therefore the Fourier transform of $P\ast (R_j\varphi)$ is integrable, implying $P\ast (R_j\varphi) \in L^\infty(\mathbb{R}^{n+1})$ that is what we wanted to see.
	\end{rem}
	
	\subsection{The exceptional sets \mathinhead{\boldsymbol{H_{\DD}}}{} and \mathinhead{\boldsymbol{T_{\DD}}}{}}
	\label{subsec5.2}
	In light of the third property in Lemma \ref{lem5.1}, assume we find a compact set $G\subset F$ such that
	\begin{itemize}[nolistsep]
		\item[\textit{i}.] $\mu(F)\lesssim \mu(G)$,
		\item[\textit{ii}.] $\mu|_G$ has $n$-growth,
		\item[\textit{iii}.] $\PP_{\mu|_G}$ is a bounded operator in $L^2(\mu|_G)$.
	\end{itemize}
	Then, $\mu|_G$ is admissible for $\gamma_{\text{op}}(G)$, and we are done, since we have
	\begin{equation}
		\label{eq5.4}
		\widetilde{\gamma}(E)\lesssim \mu(F) \lesssim \mu(G) \leq \gamma_{\text{op}}(G)\leq \gamma_{\text{op}}(F)\approx \widetilde{\gamma}_+(F) \lesssim \widetilde{\gamma}_+(E),
	\end{equation}
	where the first inequality is due to Lemma \ref{lem5.1} and the last is by property $\Ptwo'$ above.
    
	The construction of $G$ is inspired by the $Tb$-theorems found in \cite[Theorem 5.1]{To3} and  \cite[Theorem 7.1]{Vo}, which are presented for the particular case of Cauchy and Riesz potentials respectively. To apply this kind of results we need to introduce the notions of Ahlfors ball, Ahlfors radius and Ahlfors point in our context. Due to \cite[Lemma 4.1]{MPr} there are positive dimensional constants $a_1,a_2$ such that
	\begin{equation*}
		\widetilde{\gamma}(E)\leq  a_1\,\text{diam}(E)^n,
	\end{equation*}
	and also
	\begin{equation*}
		a_2^{-1}r_i^n \leq \widetilde{\gamma}(2Q_i\cap E) \leq a_2r_i^n,
	\end{equation*}
	for each $i=1,\ldots,N$, where $r_i$ is the radius of the ball $B_i\subset Q_i$ constructed at the beginning of Subsection \ref{subsec5.1}. Let
	\begin{equation*}
		L:=100^na_1a_2C_1,
	\end{equation*}
	where $C_1$ and $C_2$ are the constants appearing in Theorem \ref{thm4.1}. The factor 100 is arbitrarily chosen; we may pick, for convenience, any other \textit{large} constant. Assumption $\Athree$ implies
	\begin{align*}
		\mu(F) = \sum_{i=1}^N r_i^n \leq a_2\sum_{i=1}^N \widetilde{\gamma}(2Q_i\cap E) \leq a_2C_1 \widetilde{\gamma}(E) \leq a_1a_2C_1 \, \text{diam}(E)^n.
	\end{align*}
	Therefore:
	\begin{equation*}
		\forall \ox \in F, \, \forall R > \frac{1}{100}\text{diam}(F)\,:\quad \mu(B_R(\ox))\leq LR^n.
	\end{equation*}
	A ball $B_R(\ox)$, $\ox \in F$, is an \textit{Ahlfors ball} if precisely the estimate $\mu(B_R(\ox))\leq LR^n$ holds. Notice that if $R>\frac{1}{100}\text{diam}(F)$, then the ball is an Ahlfors ball.
	We also define the \textit{Ahlfors radius},
	\begin{equation}
		\label{eq5.5}
		\pazocal{R}(\ox):=\sup\big\{r>0: B_r(\ox) \text{ is non-Ahlfors}\,\big\}.
	\end{equation}
	\textit{Ahlfors points} are those for which $\pazocal{R}(\ox)=0$. For every $\ox\in F$, we have $\pazocal{R}(\ox)\leq \frac{1}{100}\text{diam}(F)$. Set
	\begin{equation*}
		H':=\bigcup_{\ox\in F} B_{\pazocal{R}(\ox)}(\ox),
	\end{equation*}
	where we convey $B_0(\ox):=\varnothing$, and apply a $5r$-covering theorem \cite[Theorem 2.1]{Ma} to choose a countable family of disjoint balls $\{B_{\pazocal{R}_k}\}_{k}:=\{B_{\pazocal{R}(\ox_k)} (\ox_k)\}_{k}$ such that
	\begin{equation}
		\label{eq5.6}
		H'\subset H := \bigcup_{k}B_{5\pazocal{R}_k},
	\end{equation}
	so that all non-Ahlfors balls are contained in $H$. Observe that since $\mu(B_{\pazocal{R}_k})\geq L\pazocal{R}_k^n$ and the balls $B_{\pazocal{R}_k}$ are disjoint, we have
	\begin{equation*}
		\sum_{k} \pazocal{R}_k^n \leq \frac{1}{L}\sum_{k} \mu(B_{\pazocal{R}_k}) \leq \frac{1}{L}\mu(F).
	\end{equation*}
	Let $\DD^0$ be the usual dyadic lattice in $\mathbb{R}^{n+1}$
	and for $\ox\in \mathbb{R}^{n+1}$ we write
	\begin{equation*}
		\DD(\ox):=\ox+\DD^0,
	\end{equation*}
	the translation of $\DD^0$ by the vector $\ox$.  Let $\DD=\DD(\ox)$ be any fixed dyadic lattice from the family of lattices $\{\DD(\oy)\}_{\oy\in \mathbb{R}^{n+1}}$. Consider the subcollection of dyadic cubes $\DD_H\subset \DD$ defined as follows: $Q\in \DD_H$ if there is a ball $B_{5\pazocal{R}_k}(\ox_k)$ satisfying
	\begin{equation}
		\label{eq5.7}
		B_{5\pazocal{R}_k}(\ox_k)\cap Q \neq \varnothing \qquad \text{and} \qquad 10\pazocal{R}_k < \ell(Q) \leq 20\pazocal{R}_k.
	\end{equation}
	Then, it is clear that
	\begin{equation*}
		\bigcup_{k} B_{5\pazocal{R}_k}(\ox_k)\subset \bigcup_{Q\in \DD_H} Q.
	\end{equation*}
	Pick a subfamily of disjoint maximal cubes $\{Q_k\}_{k}$ from $\DD_H$ so that we may define the exceptional set $H_\DD$ as follows
	\begin{equation*}
		\bigcup_{Q\in \DD_H} Q = \bigcup_{k} Q_k=: H_{\DD} = H_{\DD(\ox)}.
	\end{equation*}
	By construction one has 
	\begin{equation*}
		H\subset H_{\DD(\ox)}, \quad  \forall \ox\in \mathbb{R}^{n+1},
	\end{equation*}
	and since for each ball $B_{5\pazocal{R}_k}(\ox_k)$ conditions \eqref{eq5.7} can only occur for a number of cubes in $\DD_H$ which is bounded by a dimensional constant, there is $c_H=c_H(n)>0$ so that
	\begin{equation}
		\label{eq5.8}
		\sum_{k} \ell(Q_k)^n \leq c_H \sum_{k} \pazocal{R}_k^n \leq \frac{c_H}{L}\mu(F).
	\end{equation}
	We define the exceptional set $T_\DD$ in a similar manner but imposing a proper \textit{accretivity condition} on the Radon-Nikodym derivative $b$. Consider the family $\DD_T\subset \DD$ of cubes satisfying
	\begin{equation}
		\label{eq5.9}
		\mu(Q)\geq c_T|\nu(Q)|,
	\end{equation}
	with $c_T=c_T(n)>0$ big constant to be fixed below. Let $\{Q_k\}_{k}$ be the subcollection of maximal (and thus disjoint) dyadic cubes from $\DD_T$. We define the exceptional set $T_\DD=T_{\DD(\ox)}$ as
	\begin{equation*}
		T_\DD:=\bigcup _{k} Q_k.
	\end{equation*}
	The next result shows that $\mu(F\setminus{(H_{\pazocal{D}}\cup T_{\pazocal{D}})})$ is comparable to $\mu(F)$. The reader can find its proof in \cite[Lemma 6.12]{To3}.
	\begin{lem}
		\label{lem5.2}
		If $L$ and $c_T$ are chosen big enough, then
		\begin{equation*}
			|\nu(H_{\pazocal{D}}\cup T_{\pazocal{D}})|\leq \frac{1}{2}|\nu(F)| \qquad \text{and} \qquad \mu(H_{\pazocal{D}}\cup T_{\pazocal{D}})\leq \delta_0\mu(F),
		\end{equation*}
		with $\delta_0< 1$ constant.
	\end{lem}
	
	\subsection{Verifying the hypotheses of a \mathinhead{Tb\,}{}-theorem}
	\label{subsec5.3}
	The goal of this subsection is to prove the lemma below, analogous to \cite[Lemma 6.8 (Main Lemma)]{To3}. We remark that in order to apply the $Tb$-theorem found in \cite{NTVo1} we will still need to check an additional weak boundedness property for a particular suppressed kernel, due to the lack of anti-symmetry of our kernel $P$. In the proof of the aforementioned reference or that of \cite[Theorem 5.1]{To3}, such condition needs not to be verified, by the anti-symmetry of the Cauchy and Riesz kernels.
	
	\begin{lem}
		\label{lem5.3}
		Let $E\subset \mathbb{R}^{n+1}$ be a compact set verifying $\Aone, \Atwo$ and $\Athree$. Let $\FF:=\cup_{i=1}^N\QQ_i$ be the covering provided by Theorem \ref{thm4.1}, and $F:=\cup_{i=1}^NQ_i= \cup_{i=1}^N \frac{1}{2}\QQ_i$, that still satisfies $E\subset F$. Let $\mu$ and $\nu$ be the measures defined in \eqref{eq5.1} and \eqref{eq5.2}. Then, there exists a subset $H\subset F$, such that
		\begin{enumerate}[leftmargin=*,itemsep=0.1cm]
			\item[\textit{1}.] $ \nu=b \mu$ for some $b$ with $\|b\|_\infty \leq c_1$.
			\item[\textit{2}.] $|\nu(F)| = \widetilde{\gamma}(E)/2$.
			\item[\textit{3}.] $c_3^{-1}\widetilde{\gamma}(E) \leq \mu(F) \leq c_3 \widetilde{\gamma}(E)$.
			\item[\textit{4}.] For any cube $Q\subset \mathbb{R}^{n+1}$ , $|\nu(Q)|\leq c_4\ell(Q)^n$.
			\item[\textit{5}.] If $\PP_{\ast}$ and $\PP^\ast_{\ast}$ denote the maximal convolution operators associated to $P$ and $P^\ast$ respectively, we have 
			\begin{equation*}
				\int_{F\setminus{H}}\PP_{\ast}\nu(\ox)\dd\mu(\ox)\leq c_5 \,\mu(F), \qquad \int_{F\setminus{H}}\PP^\ast_{\ast}\nu(\ox)\dd\mu(\ox)\leq c_5\,\mu(F).
			\end{equation*}
			\item[\textit{6}.] If $\mu(B_r(\ox))>Lr^n$ $(\,$for some big constant $L)$, then $B_r(\ox)\subset H$. In particular, for any $\ox\in F\setminus{H}$ and $r>0$, one has $\mu(B_r(\ox))\leq Lr^n$. 
			\item[\textit{7}.]  $H$ is of the form $\bigcup_{k\in I_H}B_{r_k}(\ox_k)$, for some countable set of indices $I_H$. Moreover $\sum_{k\in I_H} r_k^n \leq \frac{5^n}{L}\mu(F)$.
		\end{enumerate}
		The constants $c_1, c_3, c_4$ and $c_5$ are fixed, while $L$ can be chosen arbitrarily large.
	\end{lem}
	Many of the hypothesis in Lemma \ref{lem5.3} are verified for our particular choice of $\mu$ and $\nu$. Properties \textit{1}, \textit{2} and \textit{4} are proved in Lemma \ref{lem5.1}, and \textit{3} follows from the third property of the same lemma, and by assumption $\Athree$ combined with $r_i^n\approx \widetilde{\gamma}(2Q_i\cap E)$. Moreover, taking $H\subset F$ to be as in \eqref{eq5.6}, properties \textit{6} and \textit{7} are also guaranteed by definition. So it remains to check hypothesis \textit{5}. To do so, let us first present a series of auxiliary results.
	\begin{lem}
		\label{lem5.4}
		Let $T$ be a distribution supported on a compact set  $E\subset \mathbb{R}^{n+1}$ with $\|P\ast T\|_\infty\leq 1$ or $\|P^\ast\ast T\|_\infty\leq 1$. Let $Q\subset \mathbb{R}^{n+1}$ be a cube and $\varphi\in \pazocal{C}^1_{c,\pazocal{N}}(2Q)$ such that $\|\varphi\|_\infty\leq D\,\ell(Q)$ and $\|\nabla\varphi\|_\infty\leq A\,D$. Then
		\begin{equation*}
			|\langle T, \varphi \rangle|\lesssim D\,\ell(Q)^{n+1}.
		\end{equation*}
	\end{lem}
	\begin{proof}
		Recall Remark \ref{rem2.1} to notice that the proof is simply an application of \cite[Corollary 3.3]{MPr} to $\psi:=\big[ 2AD\ell(Q) \big]^{-1}\varphi\in\pazocal{C}^1_{c,\pazocal{N}}(2Q)$ that is such that $\|\nabla \psi\|_\infty \leq \ell(2Q)^{-1}$. Observe that \cite[Corollary 3.3]{MPr} is also valid for $P^\ast$, which can be proved analogously using that $P^\ast$ is fundamental solution of the conjugate operator $\overline{\Theta}^{1/2}$.
	\end{proof}
	\begin{lem}
		\label{lem5.5}
		Let $E\subset \mathbb{R}^{n+1}$ be compact set and $Q\subset \mathbb{R}^{n+1}$ a cube. If $T$ is a distribution supported on $E\subset Q$, then the following holds:
		\begin{enumerate}[leftmargin=*,itemsep=0.1cm]
			\item[\textit{1}.] If $\|P\ast T\|_\infty\leq 1$ or $\|P^\ast\ast T\|_\infty\leq 1$ and $\varphi\in \pazocal{C}^1_{\pazocal{N}}(2Q)$ satisfies $\|\varphi\|_{L^\infty(2Q)}\leq D\,\ell(Q)$ and $\|\nabla \varphi \|_{L^\infty(2Q)}\leq D$, then
			\begin{equation*}
				|\langle T, \varphi \rangle|\lesssim D\,\ell(Q)^{n+1}.
			\end{equation*}
			\item[\textit{2}.] If $\|P\ast T\|_\infty\leq 1$, $\|P^\ast\ast T\|_\infty\leq 1$ and $\varphi$ is as above, then
			\begin{equation*}
				|\langle T, \varphi \rangle|\lesssim D\,\ell(Q)\cdot \widetilde{\gamma}(E).
			\end{equation*}
		\end{enumerate}
	\end{lem}
	\begin{proof}
		We prove \textit{1} under the assumption  $\|P\ast T\|_\infty\leq 1$ (the proof is analogous for $P^\ast$). Consider a test function $\psi\in \pazocal{C}^\infty_c(2Q)$ with $\|\psi\|_\infty\leq 1, \|\nabla\psi\|_\infty \leq \ell(2Q)^{-1}$ and $\psi|_Q\equiv 1$. Notice that $\eta := \varphi\psi$ belongs to $\pazocal{C}^1_{c,\pazocal{N}}(2Q)$ and $\|\eta\|_\infty\leq D\,\ell(Q)$ and $\|\nabla \eta\|_\infty \leq 2D$. Then, by Lemma \ref{lem5.4} and the fact that $\text{supp}(T)\subset E \subset Q$,
		\begin{equation*}
			|\langle T, \varphi \rangle| = |\langle T, \eta \rangle| \lesssim D\,\ell(Q)^{n+1},
		\end{equation*}
		To prove \textit{2}, consider $\psi$ as before and define
		\begin{equation*}
			\eta:= \big[ 4D\ell(Q) \big]^{-1}\varphi\psi\in \pazocal{C}^1_{c,\pazocal{N}}(2Q),
		\end{equation*}
		that is such that $\|\nabla \eta\|_\infty\leq \ell(2Q)^{-1}$. By Remark \ref{rem2.1}, we shall apply \cite[Theorem 3.1]{MPr} to deduce $\|P\ast \eta T\|_\infty\lesssim 1$ and $\|P^\ast\ast \eta T\|_\infty\lesssim 1$. Hence $|\langle \eta T, 1 \rangle| \lesssim \widetilde{\gamma}(E)$, meaning
		\begin{equation*}
			|\langle T, \varphi \rangle|\lesssim D\,\ell(Q)\cdot \widetilde{\gamma}(E).
		\end{equation*}
	\end{proof}
	Let us now consider a distribution $T$ supported on $E\subset Q = Q(c_Q, \ell(Q))$ with $\|P\ast T\|_\infty \leq 1$ and $\|P^\ast \ast T\|_\infty \leq 1$. Fix any $\oz\in \mathbb{R}^{n+1}\setminus{3Q}$ and set
	\begin{equation*}
		\varphi_{\oz}(\cdot):=P(\oz-\cdot)-P(\oz-c_Q), \qquad \varphi^\ast_{\oz}(\cdot):=P^\ast(\oz-\cdot)-P^\ast(\oz-c_Q).
	\end{equation*}
	Write $\oz=(z,\tau)$ and let $\oy=(y,s)$ be a generic point in $\mathbb{R}^{n+1}$. We define the null set $\pazocal{N}=\pazocal{N}(\oz)$ to be the horizontal hyperplane
	\begin{equation*}
		\pazocal{N}:=\big\{ \oy\in\mathbb{R}^{n+1}\,:\, s=\tau \big\}.
	\end{equation*}
	This way $\varphi_{\oz}, \varphi_{\oz}^\ast \in \pazocal{C}^\infty_{\pazocal{N}}(2Q)$. Moreover, for any $\oy\in 2Q$, by \cite[Lemma 2.1]{MPr},
	\begin{equation*}
		|\varphi_{\oz}(\oy)|=|P(\oz-\oy)-P(\oz-c_Q)|\lesssim \frac{|\oy-c_Q|}{|\oz-c_Q|^{n+1}}\lesssim \frac{\ell(Q)}{\text{dist}(\oz,E)^{n+1}},
	\end{equation*}
	and for any $\oy\in 2Q\setminus{\pazocal{N}}$,
	\begin{equation*}
		|\nabla \varphi_{\oz}(\oy)|=|\nabla P(\oz-\oy)|\lesssim \frac{1}{|\oz-\oy|^{n+1}}\leq \frac{3^{n+1}}{|\overline{z}-c_Q|}\lesssim \frac{1}{\text{dist}(\oz,E)^{n+1}}.
	\end{equation*}
	and the same bounds clearly also hold for $\varphi_{\oz}^\ast$. Therefore we have
	\begin{align*}
		\|\varphi_{\oz}\|_{L^\infty(2Q)}&\lesssim \frac{\ell(Q)}{\text{dist}(\oz,E)^{n+1}}, \qquad \|\varphi_{\oz}^\ast \|_{L^\infty(2Q)}\lesssim \frac{\ell(Q)}{\text{dist}(\oz,E)^{n+1}},\\
		\|\nabla \varphi_{\oz}\|_{L^\infty(2Q)}&\lesssim \frac{\ell(Q)}{\text{dist}(\oz,E)^{n+1}}, \qquad \|\nabla \varphi_{\oz}^\ast \|_{L^\infty(2Q)}\lesssim \frac{1}{\text{dist}(\oz,E)^{n+1}}.
	\end{align*}
	Hence, using the identities
	\begin{align*}
		\big\rvert P\ast T(\oz) - \langle T, 1 \rangle P(\oz-c_Q) \big\rvert &= |\langle T, \varphi_{\oz} \rangle|,\\
		\big\rvert P^\ast\ast T(\oz) - \langle T, 1 \rangle P^\ast(\oz-c_Q) \big\rvert &= |\langle T, \varphi_{\oz}^\ast \rangle|,
	\end{align*}
	a direct application of the second statement in Lemma \ref{lem5.5} yields
	\begin{cor}
		\label{cor5.6}
		Let $E\subset \mathbb{R}^{n+1}$ be a compact set contained in the cube $Q = Q(c_Q, \ell(Q))$, and $T$ a distribution admissible for $\widetilde{\gamma}(E)$. Then, for any $\oz\in \mathbb{R}^{n+1}\setminus{3Q}$,
		\begin{align*}
			\big\rvert P\ast T(\oz) - \langle T, 1 \rangle P(\oz-c_Q) \big\rvert &\lesssim \frac{\ell(Q)}{\text{\normalfont{dist}}(\oz,E)^{n+1}}\widetilde{\gamma}(E),\\
			\big\rvert P^\ast\ast T(\oz) - \langle T, 1 \rangle P^\ast(\oz-c_Q) \big\rvert &\lesssim \frac{\ell(Q)}{\text{\normalfont{dist}}(\oz,E)^{n+1}}\widetilde{\gamma}(E).
		\end{align*}
	\end{cor}
	
	\begin{lem}
		\label{lem5.7}
		Let
		\begin{equation*}
			\widetilde{\varphi_i T_0} := \frac{\langle T_0, \varphi_i \rangle}{\pazocal{L}^{n+1}(B_i)}\pazocal{L}^{n+1}|_{B_i}, \quad i=1,\ldots,N.
		\end{equation*}
		For each $i=1,\ldots, N$, we have
		\begin{enumerate}[leftmargin=*,itemsep=0.1cm]
			\item[\textit{1}.] $\|P\ast \varphi_iT_0 \|_\infty \lesssim 1$ and $\|P\ast \widetilde{\varphi_iT_0} \|_\infty \lesssim 1$.
			\item[\textit{2}.] For any $\oz\in \mathbb{R}^{n+1}\setminus{3Q_i}$,
			\begin{equation*}
				\big\rvert  P\ast \varphi_iT_0(\oz)-P\ast \widetilde{\varphi_iT_0}(\oz) \big\rvert \lesssim \frac{\ell(2Q_i)}{\text{\normalfont{dist}}(\oz,2Q_i)^{n+1}}\widetilde{\gamma}(2Q_i\cap E).
			\end{equation*}
		\end{enumerate}
		The same results hold changing $P$ by $P^\ast$.
	\end{lem}
	\begin{proof}
		The first estimate in \textit{1} is just a consequence of the localization result \cite[Theorem 3.1]{MPr}. Regarding the second, fix $\ox\in\mathbb{R}^{n+1}$ and compute:
		\begin{align*}
			|P\ast \widetilde{\varphi_iT_0}|(\ox) \leq \frac{|\langle T_0, \varphi_i \rangle|}{\pazocal{L}^{n+1}(B_i)}\int_{B_i} \frac{\dd\oy}{|\ox-\oy|^{n}} \lesssim \frac{1}{r_i}\int_{B_i} \frac{\dd\oy}{|\ox-\oy|^{n}},
		\end{align*}
		where we have used that $|\langle T_0, \varphi_i \rangle|\lesssim r_i^n$, that has already been argued in \eqref{eq5.3}. We deal with remaining integral as follows:
		\begin{align*}
			\int_{B_i}\frac{\dd\oy}{|\ox-\oy|^{n}}&= \int_{B_i\cap \{|\ox-\oy|\geq r_i\}}\frac{\dd\oy}{|\ox-\oy|^{n}}+\int_{B_i\cap \{|\ox-\oy|< r_i\}}\frac{\dd\oy}{|\ox-\oy|^{n}}\\
			&\leq \frac{\pazocal{L}^{n+1}(B_i)}{r_i^n}+\int_{B_{4r_i}(\ox)}\frac{\dd\oy}{|\ox-\oy|^{n}} \lesssim r_i.
		\end{align*}
		Let us prove \textit{2}: fix $\oz\in \mathbb{R}^{n+1}\setminus{3Q_i}$ and notice
		\begin{equation*}
			\big\rvert  P\ast \varphi_iT_0(\oz)-P\ast \widetilde{\varphi_iT_0}(\oz) \big\rvert = \bigg\rvert  P\ast \varphi_iT_0(\oz)-\frac{\langle \varphi_iT_0,1 \rangle}{r_i^n}P\ast \mu_i(\oz) \bigg\rvert,
		\end{equation*}
		where
		\begin{equation*}
			\frac{1}{r_i^n}P\ast \mu_i(\oz)=\int_{B_i} \big( P(\oz-\oy)-P(\oz-c_{B_i}) \big)\frac{\dd\mu_i(\oy)}{r_i^n} + P(\oz-c_{B_i}).
		\end{equation*}
		Therefore
		\begin{align}
			\nonumber
			\big\rvert  P\ast \varphi_iT_0(\oz)&-P\ast \widetilde{\varphi_iT_0}(\oz) \big\rvert \\
			\nonumber
			&\leq \big\rvert P\ast \varphi_iT_0(\oz)-\langle \varphi_iT_0,1\rangle P(\oz-c_{B_i}) \big\rvert\\
			&\hspace{2cm}+ \frac{|\langle \varphi_iT_0,1 \rangle|}{\pazocal{L}^{n+1}(B_i)}\int_{B_i} \big\rvert P(\oz-\oy)-P(\oz-c_{B_i}) \big\rvert \dd\oy.
			\label{eq5.10}
		\end{align}
		Apply Corollary \ref{cor5.6} to $T:=\varphi_iT_0$, admissible for $\widetilde{\gamma}(2Q_i\cap E)$ to deduce
		\begin{align}
			\big\rvert P\ast \varphi_iT_0(\oz)-\langle \varphi_iT_0,1\rangle P(\oz-c_{B_i}) \big\rvert &\lesssim \frac{\ell(2Q_i)}{\text{dist}(\oz,2Q_i\cap E)^{n+1}}\widetilde{\gamma}(2Q_i\cap E)\nonumber \\
			&\leq \frac{\ell(2Q_i)}{\text{dist}(\oz,2Q_i)^{n+1}}\widetilde{\gamma}(2Q_i\cap E).
			\label{eq5.11}
		\end{align}
		By \cite[Lemma 2.1]{MPr} and estimates in the proof of Lemma \ref{lem5.1} we also have
		\begin{align}
			\frac{|\langle \varphi_iT_0,1 \rangle|}{\pazocal{L}^{n+1}(B_i)}\int_{B_i} \big\rvert P(\oz-\oy)-&P(\oz-c_{B_i}) \big\rvert \dd\oy\nonumber \\
			&\lesssim \frac{1}{r_i}\frac{1}{|\oz-c_{B_i}|^{n+1}}\int_{B_i}|\oy-c_{B_i}|\dd\oy\nonumber \\
			&\lesssim \frac{\pazocal{L}^{n+1}(B_i)}{\text{dist}(\oz, 2Q_i)^{n+1}} \approx \frac{r_i}{\text{dist}(\oz, 2Q_i)^{n+1}} \widetilde{\gamma}(2\QQ_i\cap E)\nonumber \\
			&\leq \frac{\ell(2Q_i)}{\text{dist}(\oz, 2Q_i)^{n+1}} \widetilde{\gamma}(2Q_i\cap E).
			\label{eq5.12}
		\end{align}
		Hence, using \eqref{eq5.11} and \eqref{eq5.12} in \eqref{eq5.10} we finish the proof of \textit{2}.
	\end{proof}
	Our next goal will be to obtain a regularized version of statement \textit{2} in Lemma \ref{lem5.7}. Consider $\psi$ a smooth radial function $\psi$ supported on $B_1(0)$ with $0\leq \psi \leq 1$, $\int\psi =1$ and $\|\nabla \psi\|_\infty \leq 1$, and set
	\begin{equation*}
		\Psi_{\varepsilon}(\ox):= \frac{1}{\varepsilon^{n+1}}\psi \bigg( \frac{\ox}{\varepsilon} \bigg), \quad \varepsilon>0.
	\end{equation*}
	Notice that $\int \Psi_\varepsilon = 1$ for every $\varepsilon>0$. Define the regularized kernels
	\begin{equation*}
		R_\varepsilon := \Psi_\varepsilon \ast P, \qquad R_\varepsilon^\ast := \Psi_\varepsilon \ast P^\ast,
	\end{equation*}
	as well as $\pazocal{R}_{\mu,\varepsilon}$ and $\pazocal{R}_{\mu,\varepsilon}^\ast$, its associated convolution operators with respect to the finite Borel measure $\mu$.
	\begin{rem}
		\label{rem5.2}
		Observe that, in particular, for a point $\ox_0=(x_0,t_0)$ with $|\ox_0|>\varepsilon$ and $t_0\neq 0$ we have 
		\begin{equation*}
			R_\varepsilon(\ox_0) = P(\ox_0)\quad \text{ and } \quad R_\varepsilon^\ast(\ox_0) = P^\ast(\ox_0).
		\end{equation*}
		This follows as in \cite[Equation 3.32]{Vo}, using the harmonicity of $P$ and $P^\ast$ in $\mathbb{R}^{n+1}\setminus{\{t=0\}}$ and the mean value property in that domain. Notice that we need the spherical symmetry of $\Psi_\varepsilon$ and $\int \Psi_\varepsilon = 1$ to ensure this.
        
		On the other hand, for any $\ox_0$ such that $|\ox_0|<\varepsilon$ we have
		\begin{equation*}
			|R_\varepsilon(\ox_0)|\lesssim \frac{1}{\varepsilon^n}\quad \text{ and } \quad |R_\varepsilon^\ast(\ox_0)|\lesssim \frac{1}{\varepsilon^n}.
		\end{equation*}
		Indeed, using the definition of $\Psi$ and the fact that $0\leq \psi\leq 1$, we get
		\begin{equation*}
			|R_\varepsilon(\ox_0)|\leq \frac{1}{\varepsilon^{n+1}}\int_{B_\varepsilon(0)}\frac{\dd\ox}{|\ox_0-\ox|^n}\leq \frac{1}{\varepsilon^{n+1}}\int_{2B_\varepsilon(\ox_0)}\frac{\dd\ox}{|\ox_0-\ox|^n} \lesssim \frac{1}{\varepsilon^n},
		\end{equation*}
		and analogously for $R_\varepsilon^\ast(\ox_0)$.
	\end{rem}
	
	\begin{lem}
		\label{lem5.8}
		For any $\oz\in \mathbb{R}^{n+1}\setminus{4Q_i}$, any $\varepsilon>0$ and each $i=1,\ldots,N$,
		\begin{equation*}
			\big\rvert  R_\varepsilon \ast \varphi_iT_0(\oz)-R_{\varepsilon}\ast \widetilde{\varphi_iT_0}(\oz) \big\rvert \lesssim \frac{\ell(2Q_i)}{\text{\normalfont{dist}}(\oz,2Q_i)^{n+1}}\widetilde{\gamma}(2Q_i\cap E).
		\end{equation*}
		The same result holds changing $R_\varepsilon$ by $R_\varepsilon^\ast$.
	\end{lem}
	\begin{proof}
		We will proof the result only for $R_\varepsilon$, since the arguments for $R_\varepsilon^\ast$ are analogous. Fix $\oz\in  \mathbb{R}^{n+1}\setminus{4Q_i}$ and $\varepsilon<\frac{1}{2}\text{dist}(\oz,2Q_i)$. This way, since for any $\oy \in B_\varepsilon(\oz)$ we have $\text{dist}(\oy,2Q_i) \approx \text{dist}(\oz,2Q_i)$, by Lemma \ref{lem5.7} and the fact that $\int_{B_{\varepsilon}(\oz)}\Psi_\varepsilon(\oz-\oy)\dd \oy=1$ we deduce
		\begin{align*}
			\big\rvert  R_\varepsilon \ast \varphi_iT_0(\oz)-R_{\varepsilon}\ast \widetilde{\varphi_iT_0}(\oz) \big\rvert  &\leq \int_{B_{\varepsilon}(\oz)}\Psi_\varepsilon(\oz-\oy)  \big\rvert P \ast \varphi_iT_0(\oy)-P\ast \widetilde{\varphi_iT_0}(\oy) \big\rvert \dd \oy\\
			&\lesssim \frac{\ell(2Q_i)}{\text{\normalfont{dist}}(\oz,2Q_i)^{n+1}}\widetilde{\gamma}(2Q_i\cap E),
		\end{align*}
		the desired inequality. Hence, we are only left to study the case $\varepsilon \geq \frac{1}{2}\text{dist}(\oz, 2Q_i)$. Observe that in this setting, since $\oz\not\in 4Q_i$ we have $\varepsilon\gtrsim \ell(Q_i)$. Write
		\begin{equation*}
			\alpha_i:= \varphi_iT_0-\widetilde{\varphi_iT_0},
		\end{equation*}
		so that
		\begin{align*}
			|R_\varepsilon\ast \alpha_i(\oz)|\leq  \int_{\text{supp}(\Psi_\varepsilon\ast \alpha_i)} P(\oz-\oy)  \big\rvert \Psi_\varepsilon \ast \alpha_i(\oy) \big\rvert \dd\oy.
		\end{align*}
		Observe that $\text{supp}(\widetilde{\varphi_i T_0})\subset B_i, \, \text{supp}( \varphi_i T_0 )\subset 2Q_i\cap E$ and $\text{supp}(\Psi_\varepsilon)\subset B_\varepsilon(0)$. Then, the support of $\Psi_\varepsilon\ast \alpha_i$ is contained in $\pazocal{U}_\varepsilon(2Q_i)$, an open $\varepsilon$-neighborhood of $2Q_i$. This implies
		\begin{align*}
			|R_\varepsilon\ast \alpha_i(\oz)| &\leq \|\Psi_\varepsilon\ast \alpha_i\|_\infty \int_{\pazocal{U}_\varepsilon(2Q_i)} \frac{\dd\oy}{|\oz-\oy|^{n}} \lesssim (\ell(Q_i)+\varepsilon)\cdot \|\Psi_\varepsilon\ast \alpha_i\|_\infty \\
			&\lesssim \varepsilon\cdot \|\Psi_\varepsilon\ast \alpha_i\|_\infty,
		\end{align*}
		by integrating over decreasing annuli centered at $\oz$, and using that $\text{dist}(\oz, 2Q_i)\leq 2\varepsilon$ and $\varepsilon\gtrsim \ell(Q_i)$. On the other hand, if $\eta_i\in \pazocal{C}_c^{\infty}(3Q_i)$ with $0\leq \eta_i \leq 1$, $\eta_i\equiv 1$ on $2Q_i$ (where the support of $\alpha_i$ is contained) and $\|\nabla \eta_i\|_\infty \leq \ell(Q_i)^{-1}$, we have
		\begin{align*}
			\Psi_\varepsilon \ast \alpha_i (\ow) &= \langle \Psi_\varepsilon (\ow -\cdot),\alpha_i \rangle = \langle \Psi_\varepsilon (\ow -\cdot) - \Psi_\varepsilon (\ow - c_{Q_i}) ,\alpha_i \rangle\\
			&=\big\langle \big( \Psi_\varepsilon (\ow -\cdot) - \Psi_\varepsilon (\ow - c_{Q_i})\big) \eta_i,\alpha_i \big\rangle,
		\end{align*}
		where we have used $\langle \varphi_iT_0, 1 \rangle = \langle \widetilde{\varphi_iT_0},1 \rangle$.
        
		We claim that for each $\ow\in \mathbb{R}^{n+1}$, the function
		\begin{equation*}
			\varphi (\overline{\xi}) := \big( \Psi_\varepsilon (\ow -\xi) - \Psi_\varepsilon (\ow - c_{Q_i})\big) \eta_i(\xi),
		\end{equation*}
		satisfies $\|\varphi\|_\infty \lesssim \varepsilon^{-(n+2)}\ell(Q_i)$ and $\|\nabla \varphi\|_\infty \lesssim \varepsilon^{-(n+2)}$. Then, using that statement \textit{1} in Lemma \ref{lem5.7} implies that $\varphi T_0$ is admissible for $\widetilde{\gamma}(2Q_i\cap E)$ and that $\widehat{\varphi T_0}$ is admissible for $\widetilde{\gamma}(B_i)$, we have, by Lemma \ref{lem5.5},
		\begin{align*}
			| \Psi_\varepsilon \ast \alpha_i|=|\langle \varphi, \alpha_i \rangle | &\lesssim \frac{\ell(Q_i)}{\varepsilon^{n+2}} \big\{ \widetilde{\gamma}(2Q_i\cap E) + \widetilde{\gamma}(B_i) \big\}\\
			&\lesssim \frac{\ell(Q_i)}{\varepsilon^{n+2}} \big\{ \widetilde{\gamma}(2Q_i\cap E) + r_i^n \big\} \approx \frac{\ell(Q_i)}{\varepsilon^{n+2}} \widetilde{\gamma}(2Q_i\cap E) .
		\end{align*}
		Therefore,
		\begin{align*}
			|R_\varepsilon\ast \alpha_i(\oz)|  \lesssim \varepsilon\cdot \|\Psi_\varepsilon\ast \alpha_i\|_\infty \lesssim \frac{\ell(Q_i)}{\varepsilon^{n+1}} \widetilde{\gamma}(2Q_i\cap E) \lesssim \frac{\ell(2Q_i)}{\text{\normalfont{dist}}(\oz,2Q_i)^{n+1}}\widetilde{\gamma}(2Q_i\cap E),
		\end{align*}
		where we have used $\text{dist}(\oz, 2Q_i)\leq 2\varepsilon$, and we deduce the desired estimate. Hence, we are left to prove the claim. Let us fix $\ow\in \mathbb{R}^{n+1}$ and compute: on the one hand
		\begin{align*}
			\|\varphi\|_\infty = \|\varphi\|_{L^\infty(3Q_i)} \leq \frac{1}{\varepsilon^{n+2}} \sup_{\overline{\xi}\in 3Q_i} |\overline{\xi}-c_{Q_i}|\cdot \|\nabla \psi\|_\infty \lesssim \frac{\ell(Q_i)}{\varepsilon^{n+2}},
		\end{align*}
		while on the other hand
		\begin{align*}
			\|\nabla \varphi\|_\infty &= \|\nabla \varphi\|_{L^\infty(3Q_i)}\\
			&\leq \frac{1}{\ell(Q_i)}\sup_{\overline{\xi}\in 3Q_i} \big\rvert \Psi_\varepsilon(\ow-\overline{\xi})- \Psi_\varepsilon(\ow-c_{Q_i}) \big\rvert + \frac{1}{\varepsilon^{n+2}}\|\nabla \psi\|_\infty \lesssim \frac{1}{\varepsilon^{n+2}}.
		\end{align*}
		Hence, we are done.
	\end{proof}
	Let us finally prove property \textit{5} in Lemma \ref{lem5.3}:
	\begin{lem}
		\label{lem5.9}
		The following estimates hold:
		\begin{equation*}
			\int_{F\setminus{H}}\PP_{\ast}\nu(\oy)\dd\mu(\oy)\lesssim \mu(F), \qquad \int_{F\setminus{H}}\PP^\ast_{\ast}\nu(\oy)\dd\mu(\oy)\lesssim \mu(F).
		\end{equation*}
	\end{lem}
	\begin{proof}
		We only prove the first estimate, since the proof of the second is analogous. Begin by noticing that
		\begin{align*}
			\int_{F\setminus{H}}\PP_{\ast}\nu(\oy)&\dd\mu(\oy)= \int_{F\setminus{H}}\sup_{\varepsilon>0}\big\rvert \PP_{\varepsilon}\nu(\oy) \big\rvert\dd\mu(\oy)\\
			&\leq \int_{F\setminus{H}}\sup_{\varepsilon>0}\big\rvert \PP_{\varepsilon}\nu(\oy)-\pazocal{R}_{\varepsilon}\nu(\oy) \big\rvert\dd\mu(\oy)+\int_{F\setminus{H}}\sup_{\varepsilon>0}\big\rvert \pazocal{R}_{\varepsilon}\nu(\oy) \big\rvert\dd\mu(\oy)=: I_1+I_2.
		\end{align*}
		For $I_1$, observe that for each $\varepsilon>0$ and $\overline{y}=(y,s)\in F\setminus{H}$, since $\nu \ll \mu \ll \pazocal{L}^{n+1}$, we have
		\begin{equation*}
			\nu\big( \{(z,\tau)\in \mathbb{R}^{n+1}\,:\, \tau = s\}\big )=0,
		\end{equation*}
		so applying Remark \ref{rem5.2} we deduce
		\begin{align*}
			\big\rvert \PP_{\varepsilon}\nu&(\oy)-\pazocal{R}_{\varepsilon}\nu(\oy) \big\rvert\\
			&=\Bigg\rvert \int_{|\oy-\oz|\leq \varepsilon}R_\varepsilon(\oy-\oz)\dd\nu(\oz)+\int_{|\oy-\oz|>\varepsilon}R_\varepsilon(\oy-\oz)\dd\nu(\oz) - \int_{|\oy-\oz|>\varepsilon}P(\oy-\oz)\dd\nu(\oz) \Bigg\rvert\\
			&=\Bigg\rvert \int_{|\oy-\oz|\leq \varepsilon}R_\varepsilon(\oy-\oz)\dd\nu(\oz) \Bigg\rvert \leq \frac{1}{\varepsilon^n} |\nu|(B_\varepsilon(\oy)) \lesssim \frac{1}{\varepsilon^n}\mu(B_\varepsilon(\oy)) \leq M\mu(\oy).
		\end{align*}
		Then, using that $\oy\in F\setminus{H}$ we get $|I_1|\lesssim \mu(F)$. So we are left to estimate $I_2$. We introduce the notation
		\begin{equation*}
			R_\ast \ast \nu (\oy) := \sup_{\varepsilon>0} |R_\varepsilon \ast \nu (\oy)|,
		\end{equation*}
		so that
		\begin{align*}
			I_2:=\int _{F\setminus{H}} R_\ast \ast \nu (\oy)\dd \mu (\oy) \leq \int _{F\setminus{H}} R_\ast \ast  T_0 (\oy)\dd\mu (\oy) + \int _{F\setminus{H}} R_\ast \ast (\nu-T_0) (\oy)\dd\mu (\oy). 
		\end{align*}
		Since $\|P\ast T_0\|_\infty \leq 1$ by construction, we also have $\| R_\varepsilon\ast T_0 \|_\infty \leq 1$, uniformly on $\varepsilon>0$. For the second integral in $I_2$,
		\begin{align*}
			\int _{F\setminus{H}} &R_\ast \ast (\nu-T_0) (\oy)\dd\mu (\oy) \leq \sum_{i=1}^N \int_{F\setminus{H}} R_\ast \ast \big(\widetilde{\varphi_iT_0}-\varphi_iT_0 \big) (\oy)\dd\mu (\oy) \\
			&\hspace{-0.2cm}= \sum_{i=1}^N \bigg( \int_{4Q_i} R_\ast \ast \big(\widetilde{\varphi_iT_0}-\varphi_iT_0 \big) (\oy)\dd\mu (\oy) + \int_{F\setminus{(4Q_i\cup H)}} R_\ast \ast \big(\widetilde{\varphi_iT_0}-\varphi_iT_0 \big) (\oy)\dd\mu (\oy)\bigg)\\
			&\hspace{-0.2cm}=:\sum_{i=1}^N I_{i,1}+I_{i,2},
		\end{align*}
		where in the first inequality we have used that $\sum_i \varphi_i \equiv 1$ on $F$ and $\nu=\sum_i \widetilde{\varphi_i T_0}$. For each $i=1,\ldots, N$ we set $\alpha_i:= \widetilde{\varphi_i T_0}-\varphi_iT_0$ and apply the first statement in Lemma \ref{lem5.7} to deduce $\|P\ast \alpha_i\|_\infty \lesssim 1$. Hence, $I_{i,1}\lesssim \mu(4Q_i)$.  To estimate $I_{i,2}$ we will use Lemma \ref{lem5.8}. Let $N$ be the smallest integer such that
		\begin{equation*}
			A_N:=\big(4^{N+1}Q_i\setminus{4^NQ_i}\big)\setminus{H}\neq \varnothing.
		\end{equation*}
		Then,
		\begin{align*}
			I_{i,2}&\lesssim \ell(2Q_i) \, \widetilde{\gamma}(2Q_i\cap E)\sum_{k=N}^\infty \int_{A_N}\frac{\dd\mu(\oy)}{\text{dist}(\oy,2Q_i)^{n+1}} \lesssim \ell(2Q_i) \, \widetilde{\gamma}(2Q_i\cap E) \sum_{k=N}^\infty \frac{\mu\big( 4^{k+1}Q_i \big)}{(4^{k}\ell(Q_i))^{n+1} }.
		\end{align*}
		Observe that for any $\ox \in A_N$, we have for each $k\geq N$
		\begin{equation*}
			\mu\big( 4^{k+1}Q_i \big) \leq \mu\big(B_{2\cdot \text{diam}(4^{k+1}Q_i)}(\ox)\big)\lesssim \big( 4^k\ell(Q_i) \big)^n,
		\end{equation*}
		where we have used that $\ox\not\in H$. Therefore,
		\begin{align*}
			I_{i,2} \lesssim \ell(2Q_i) \, \widetilde{\gamma}(2Q_i\cap E) \sum_{k=N}^\infty \frac{1}{4^k\ell(Q_i) } \lesssim \widetilde{\gamma}(2Q_i\cap E)\approx r_i^n = \mu(B_i).
		\end{align*}
		Since the cubes $10Q_1,\ldots, 10Q_N$ have bounded overlap we conclude that
		\begin{align*}
			\int _{F\setminus{H}} &R_\ast \ast (\nu-T_0) (\oy)\dd\mu (\oy) \lesssim \sum_{i=1}^N \mu(4Q_i) \lesssim \mu(F).
		\end{align*}
	\end{proof}
	The proof of \textit{5} concludes that of Lemma \ref{lem5.3}. In order to apply the $Tb$-theorem of \cite{NTVo1} or \cite[Theorem 5.1]{To3} we still need to check an additional weak boundedness property for our $n$-dimensional CZ kernel, which fails being anti-symmetric.

	\subsection{The exceptional sets \mathinhead{\pSS}{} and \mathinhead{W_{\DD}}{}. The additional eighth property}
	\label{subsec5.4}
	Assume that $F\subset B_{2^{N-3}}(0)$ for some integer $N$ large enough. Observe that assumption $\Atwo$ together with properties \textit{1} and \textit{4} in Theorem \ref{thm4.1} imply that we can take $N=4$, for example. We write
	\begin{equation*}
		\Omega:=\big[-2^{N-1},2^{N-1}\big)^{n+1}
	\end{equation*}
	and consider the \textit{random} cube $Q^0(\ow):=\ow+[-2^N,2^N)^{n+1}$, with $\ow\in \Omega$. Observe that $F\subset Q^0(\ow)$ for any $\ow\in \Omega$. Let $\mathbb{P}$ be the uniform probability measure on $\Omega$, that is, the normalized Lebesgue measure on the cube $\Omega$.\medskip\\
	When carefully reviewing the proofs of the the non-homogeneous $Tb$-theorems of \cite[\textsection 5]{To3} and \cite{Vo}, one encounters expressions that, in our context, would be of the form
	\begin{equation*}
		\langle \PP_{\Lambda,\mu}(\chi_\Delta b), (\chi_\Delta b) \rangle := \int_{\Delta}\int_{\Delta}P_{\Lambda}(\ox,\oy)\dd\nu(\oy)\dd\nu(\ox),
	\end{equation*}
	where $\Lambda$ is a certain 1-Lipschitz function associated to the suppressed kernel $P_\Lambda$, $\Delta = P\cap S$ is a parallelepiped obtained as the intersection of two cubes $P\in \DD(\ow_1)=:\DD_1$, $S\in \DD(\ow_2)=:\DD_2$, with $\ow_1,\ow_2\in \Omega$, and $\DD_1$ and $\DD_2$ are two dyadic lattices which are reciprocally \textit{good} to one another. Moreover, such cubes $P,S$ satisfy $P\subset F\setminus{(H_{\DD_1} \cup T_{\DD_1})}$ and $S\subset F\setminus{(H_{\DD_2} \cup T_{\DD_2})}$. The latter inclusions imply that $P$ is a \textit{transit cube} with respect to $\DD_1$ and we write $P\in \DD_1^{\text{tr}}$, and analogously for $S$ with respect to $\DD_2$ (we will specify the precise meaning of the terms \textit{good} and \textit{transit} in Remark \ref{rem5.3}). $P,S$ and all their dyadic children have $M$-\textit{thin boundaries} (with respect to $\mu$), for some $M>0$ dimensional constant. Recall that we say that a set $X\subset \mathbb{R}^{n+1}$ has $M$-thin boundary if
	\begin{equation*}
		\mu\big(\big\{\ox\in \mathbb{R}^{n+1} : \text{dist}(\ox, \partial X) \leq \tau \big\}\big) \leq M\tau^n, \qquad \text{for all }\, \tau>0.
	\end{equation*}
	As one may expect, expressions of the form $\langle \PP_{\Lambda,\mu}(\chi_\Delta b), (\chi_\Delta b) \rangle$ are null if the kernel associated to the operator $\PP$ is anti-symmetric, as it occurs in \cite{To3} with the Cauchy kernel, or in \cite{Vo} with the vector Riesz kernel. If this is not the case, following the proofs of general non-homogeneous $Tb$-theorems and, more precisely, the arguments of \cite[\textsection 10.2]{NTVo2} or \cite[\textsection 9]{HyMar}, the latter expressions can be dealt with if we have $\| \PP_{\Lambda,\varepsilon}\nu \|_{L^\infty(\mu)}\lesssim 1$, $\| \PP^\ast_{\Lambda,\varepsilon}\nu \|_{L^\infty(\mu)}\lesssim 1$ uniformly on $\varepsilon>0$, as well as the following weak boundedness property
	\begin{equation*}
		|\langle \PP_{\Lambda,\mu}(\chi_Q b), (\chi_Q b) \rangle| \lesssim \mu(\lambda Q), \quad \lambda >1,
	\end{equation*}
	that suffices to hold for cubes $Q$ with $M$-thin boundary and contained in parallelepipeds $\Delta$ of the form $P\cap S$, where $P$ and $S$ are as above. Such kind of restricted weak boundedness property for cubes with thin boundary is already checked in the proofs of \cite[Theorem 5.5]{MPrTo} and \cite[Theorem 4.3]{MPr}.\medskip\\
	Our goal in this subsection will be to verify the above conditions by choosing a proper 1-Lipschitz function $\Lambda$ which will depend on an additional exceptional set $\pSS$. If such additional properties hold, the proofs of the $Tb$-theorems found in \cite[Ch.5]{To3} and \cite{Vo} can be adapted to our non anti-symmetric setting. To define the exceptional set $\pSS$, we will follow an analogous construction to that of \cite[\textsection 5.2]{To3}. We begin by writing
	\begin{align*}
		\pSS_1'&:= \big\{ \ox\in F\,:\, \PP_{\ast}\nu(\ox)>\alpha \big\},\\
		\pSS_2'&:= \big\{ \ox\in F\,:\, \PP^\ast_{\nu,\ast}1(\ox)>\alpha \big\},
	\end{align*} 
	for a large constant $\alpha>0$ to be chosen below. For the moment, let us say that $\alpha \gg c_1L$. We also write,
	\begin{align*}
		\text{for }\, \ox\in \pSS_1':\quad e_1(\ox)&:=\text{sup}\big\{ \varepsilon>0\,:\, |\PP_{\varepsilon}\nu(\ox)|>\alpha \big\},\\
		\text{for }\, \ox\in \pSS_2':\quad e_2(\ox)&:=\text{sup}\big\{ \varepsilon>0\,:\, |\PP_{\nu,\varepsilon}^\ast 1(\ox)|>\alpha \big\}.
	\end{align*}
	If $\ox\in F\setminus{(\pSS_1'\cup \pSS_2')}$ we convey $e_1(\ox)=e_2(\ox):=0$. We define also
	\begin{equation*}
		\pSS_1 := \bigcup_{\ox\in \pSS_1'} B_{e_1(\ox)}(\ox) \qquad \text{and} \qquad \pSS_2 := \bigcup_{\ox\in \pSS_2'} B_{e_2(\ox)}(\ox),
	\end{equation*}
	as well as the exceptional set
	\begin{equation}
		\label{eq5.13}
		\pSS := \pSS_1\cup \pSS_2.
	\end{equation}
	Let us first show that for any $\ox\in \mathbb{R}^{n+1}$, $\mu(\pSS\setminus{H_{\DD(\ox)}})$ is small if $\alpha$ is taken big enough. The proof of the following lemma is analogous to that of \cite[Lemma 5.2]{To3}.
	\begin{lem}
		\label{lem5.10}
		Let $\ox_0 \in \mathbb{R}^{n+1}$ and $\DD:=\DD(\ox_0)$. Then, for $\alpha>0$ big enough,
		\begin{equation*}
			\mu(\pSS\setminus{H_{\DD}}) \leq \frac{4c_5}{\alpha}\mu(F).
		\end{equation*}
	\end{lem}
	\begin{proof}
		Let $\oy\in \pSS\setminus{H_{\DD}} $ and assume, for example, $\oy\in \pSS_1$. Then $\oy\in B_{e_1(\ox)}(\ox)$ for some $\ox\in \pSS_1'$. Let $\varepsilon_0(\ox)$ be such that $|\PP_{\varepsilon_0(\ox)}\nu|>\alpha$ and $\oy \in B_{\varepsilon_0(\ox)}(\ox)$. Observe that
		\begin{align*}
			\big\rvert \PP_{\varepsilon_0(\ox)}\nu(\ox) - \PP_{\varepsilon_0(\ox)}\nu (\oy) \big\rvert &\leq \big\rvert \PP_{\varepsilon_0(\ox)}\big(\chi_{B_{2\varepsilon_0(\ox)}(\oy)}\nu\big)\ox)  \big\rvert + \big\rvert \PP_{\varepsilon_0(\ox)}\big(\chi_{B_{2\varepsilon_0(\ox)}(\oy)}\nu\big) (\oy)  \big\rvert\\
			&\hspace{2cm}+c_1\int_{\mathbb{R}^{n+1}\setminus{B_{2\varepsilon_0(\ox)}(\oy)}} |P(\ox-\oz)-P(\oy-\oz)|\dd\mu(\oz).
		\end{align*}
		Since $\oy\not\in H_{\DD}$, the first two terms are bounded above by
		\begin{equation*}
			c_1 \frac{\mu\big( B_{2\varepsilon_0(\ox)}(\oy) \big)}{\varepsilon_0(\ox)^n} \leq 2^nLc_1.
		\end{equation*}
		Regarding the third term, if $A$ is the CZ constant of $P$ from property $(c)$ of \cite[Lemma 2.1]{MPr}, integration over annuli and using again that $\oy\not\in H_{\DD}$ yield
		\begin{equation*}
			c_1\int_{\mathbb{R}^{n+1}\setminus{B_{2\varepsilon_0(\ox)}(\oy)}} |P(\ox-\oz)-P(\oy-\oz)|\dd\mu(\oz) \lesssim 2^nALc_1.
		\end{equation*}
		Therefore, naming $\kappa := 2^nLc_1(2+A)$, we have $\big\rvert \PP_{\varepsilon_0(\ox)}\nu(\ox) - \PP_{\varepsilon_0(\ox)}\nu (\oy) \big\rvert \leq \kappa$, and thus
		\begin{equation*}
			\big\rvert \PP_{\varepsilon_0(\ox)}\nu(\oy) \big\rvert > \alpha -\kappa,
		\end{equation*}
		which implies, in particular, $\PP_{\ast}\nu(\oy)>\alpha-\kappa$. If we have had $\oy\in \pSS_2$ we would have obtained the same bound for $\PP_{\ast}^\ast \nu(\oy)$, since $P$ and $P^\ast$ share the same CZ constants. In any case, pick $\alpha \geq 2\kappa$ and observe that
		\begin{align*}
			\mu(\pSS\setminus{H_{\DD}}) &\leq \int_{\pSS_1\setminus{H_{\DD}}} \dd \mu + \int_{\pSS_2\setminus{H_{\DD}}} \dd \mu\\
			& \leq \frac{2}{\alpha} \int_{F\setminus{H_{\DD}}} \PP_{\ast}\nu \dd \mu + \frac{2}{\alpha}\int_{F\setminus{H_{\DD}}} \PP^\ast_{\ast}\nu \dd \mu \leq \frac{4c_5}{\alpha}\mu(F).
		\end{align*} 
	\end{proof}
	One of the implications of the above lemma is the following: by setting $\delta_1:=(1+\delta_0)/2$ (where $\delta_0$ is the parameter appearing in Lemma \ref{lem5.2}), we have $\delta_0<\delta_1<1$, and choosing $\alpha := \max\{2\kappa, 8c_5/(1-\delta_0)\}$, we get
	\begin{equation*}
		\mu\big( H_{\DD(\ox)}\cup T_{\DD(\ox)} \big) + \mu\big( \pSS\setminus{H_{\DD(\ox)}} \big) \leq \delta_1\mu(F), \qquad \forall \ox\in \mathbb{R}^{n+1}.
	\end{equation*}
	With this in mind, and defining the \textit{total exceptional set}
	\begin{equation*}
		W_{\DD(\ox)} := H_{\DD(\ox)}\cup T_{\DD(\ox)}\cup \pSS,
	\end{equation*}
	one obtains
	\begin{equation}
		\label{eq5.14}
		\mu\big( F\setminus{W_{\DD(\ox)}} \big) \geq (1-\delta_1)\mu(F),
	\end{equation}
	which is a necessary inequality in order to carry out the final probabilistic argument in the proof of the $Tb$-theorem (see \cite[\textsection 5.11.2]{To3}).
	
	The exceptional set $\pSS$ exhibits additional important properties regarding suppressed kernels (see subsection \ref{subsec3.1}), provided their associated 1-Lipschitz function $\Lambda$ satisfies certain conditions. In order to prove them, we present two preliminary results. Fix $\sigma$ any finite Borel measure in $\mathbb{R}^{n+1}$ and $\Lambda: \mathbb{R}^{n+1}\to [0,\infty)$ a 1-Lipschitz function. The next lemma admits an analogous proof to that of \cite[Lemma 8.3]{Vo}.
	
	\begin{lem}
		\label{lem5.11}
		For any $\ox\in \mathbb{R}^{n+1}$ and any $\varepsilon\geq \Lambda (\ox)$,
		\begin{equation*}
			\big\rvert \PP_{\varepsilon}\sigma(\ox) - \PP_{\Lambda,\varepsilon}\sigma(\ox) \big\rvert \lesssim \sup_{r\geq \Lambda(\ox)} \frac{|\sigma|(B_r(\ox))}{r^n}.
		\end{equation*}
		The same result holds changing $\PP$ by $\PP^\ast$.
	\end{lem}
	
	
	The following lemma with an analogous proof to that of \cite[Lemma 5.5]{To3}.
	
	\begin{lem}
		\label{lem5.12}
		Let $\ox\in \mathbb{R}^{n+1}$ and $r_0>0$ such that $\mu(B_r(\ox))\leq L r^n$ for $r\geq r_0$, as well as $|\PP_{\varepsilon}\nu(\ox)|\leq \alpha$ and $|\PP^\ast_{\varepsilon}\nu(\ox)|\leq \alpha$ for $\varepsilon\geq r_0$. If $\Lambda(\ox)\geq r_0$, then
		\begin{equation*}
			\big\rvert \PP_{\Lambda,\varepsilon}\nu \big (\ox)\rvert \lesssim \alpha + c_1L \qquad \text{and} \qquad \big\rvert \PP^\ast_{\Lambda,\varepsilon}\nu \big (\ox)\rvert \lesssim \alpha + c_1L,
		\end{equation*}
		uniformly on $\varepsilon>0$.
	\end{lem}
	
	Bearing in mind the above result, we choose our 1-Lipschitz function $\Lambda:\mathbb{R}^{n+1}\to [0,\infty)$ to satisfy, for some $\ox_0\in \mathbb{R}^{n+1}$,
	\begin{equation*}
		\Lambda(\ox) \geq \text{dist}\big(\ox, \mathbb{R}^{n+1}\setminus{\big( H_{\DD(\ox_0)}\cup \pSS \big)} \big).
	\end{equation*}
	This choice is consistent with the Lipschitz functions $\Lambda$ that appear in the proofs of the $Tb$-theorems presented in \cite[Ch.5]{To3} and \cite{Vo}, in the sense that they are constructed to ensure that the previous inequality is satisfied. This way, since $H_{\DD(\ox_0)}\cup \pSS$ contains all non-Ahlfors balls and all the balls $B_{e_1(\ox)}(\ox)$, $B_{e_2(\ox)}(\ox)$ for $\ox\in F$, we have
	\begin{equation*}
		\Lambda(\ox)\geq \max\{\pazocal{R}(\ox), e_1(\ox),e_2(\ox)\},
	\end{equation*}
	where $\pazocal{R}(\ox)$ is the Ahlfors radius defined in \eqref{eq5.5}. So by the definition of $\pSS$ and choosing $r_0:= \max\{\pazocal{R}(\ox), e_1(\ox),e_2(\ox)\}$, Lemma \ref{lem5.2} yields
	\begin{lem}
		\label{lem5.13}
		Let $\ox_0\in \mathbb{R}^{n+1}$ and $\Lambda:\mathbb{R}^{n+1}\to [0,\infty)$ a 1-Lipschitz function such that $\Lambda(\ox) \geq \text{\normalfont{dist}}\big(\ox, \mathbb{R}^{n+1}\setminus{\big( H_{\DD(\ox_0)}\cup \pSS \big)} \big)$ for all $\ox\in \mathbb{R}^{n+1}$. Then,
		\begin{equation*}
			\PP_{\Lambda,\ast}\nu(\ox) \leq c_{\Lambda} \quad \text{and} \quad \PP_{\Lambda,\ast}^\ast \nu(\ox) \leq c_{\Lambda}, \quad \forall \ox \in F,
		\end{equation*}
		with $c_\Lambda$ depending only on $c_1, c_5, L$ and $\delta_0$.
	\end{lem}
	
	Hence, if we choose such a $\Lambda$, we have $\| \PP_{\Lambda,\varepsilon}\nu \|_{L^\infty(\mu)}\lesssim 1$ and $\| \PP^\ast_{\Lambda,\varepsilon}\nu \|_{L^\infty(\mu)}\lesssim 1$ uniformly on $\varepsilon>0$. So we are left to verify the weak boundedness property for cubes with $M$-thin boundary contained in the parallelepipeds presented at the beginning of this subsection. In order to do it we will need two auxiliary results, similar to those found in \cite[\textsection 3]{MPr}.
	
	\begin{lem}
		\label{lem5.14}
		Let $\varphi$ be a function supported on a cube $Q\subset \mathbb{R}^{n+1}$ with $\|\varphi\|_\infty \leq 1$. Then, $\PP_{\Lambda}(\varphi\nu)$ is a locally integrable function and moreover, if $\mu(Q)\lesssim \ell(Q)^n$, there exists $\ox_0\in \frac{1}{4}Q$ and a constant $c_0$ such that
		\begin{equation*}
			\big\rvert \PP_{\Lambda}(\varphi\nu) (\ox_0) \big\rvert \leq c_0.
		\end{equation*}
	\end{lem}
	\begin{proof}
		In \cite[Lemma 3.5]{MPr}, the authors deal with general distributions with $n$-growth. In our statement, however, $\nu$ is a specified signed measure and the cubes $Q$ satisfy the additional growth condition $\mu(Q)\lesssim \ell(Q)^n$. To prove the local integrability of $\PP_{\Lambda}(\varphi\nu)$, let us fix $\ox\in \mathbb{R}^{n+1}$ and name $I_Q:=\{j=1,\ldots,N\,:\, Q\cap B_j \neq \varnothing\}$. We compute, bearing in mind Lemma \ref{lem3.1}:
		\begin{align*}
			\big\rvert \PP_{\Lambda}(\varphi\nu) (\ox) \big\rvert &= \Bigg\rvert \sum_{j\in I_Q} \frac{\langle T_0, \varphi_j \rangle}{\pazocal{L}^{n+1}(B_j)} \int_{Q\cap B_j} P_{\Lambda}(\ox,\oy)\varphi(\oy) \dd\oy \Bigg\rvert \\
			&\lesssim \sum_{j\in I_Q} \frac{1}{r_j} \int_{Q\cap B_j} \frac{\dd\oy}{|\ox-\oy|^n}=:\sum_{j\in I_Q} \frac{1}{r_j} I_j(\ox).
		\end{align*}
		To study $I_j(\ox)$ we split the integral into the domain
		\begin{equation*}
			D_{1,j} := Q\cap B_j\cap \big\{\oy\,:\,2|\ox-\oy|\geq \text{diam}(B_j\cap Q)\big\},
		\end{equation*}
		and its complementary $D_{2,j}:=(Q\cap B_j)\setminus{D_1}$. For the first domain of integration we directly have
		\begin{equation*}
			\int_{D_{1,j}}\frac{\dd\oy}{|\ox-\oy|^n} \lesssim \frac{\pazocal{L}^{n+1}(D_{1,j})}{\text{diam}(B_j\cap Q)^n}\leq \text{diam}(B_j\cap Q)\leq r_j.
		\end{equation*}
		For the second one notice that
		\begin{equation*}
			D_{2,j}=Q\cap B_j\cap \{\oy\,:\,2|\ox-\oy|< \text{diam}(B_j\cap Q)\} \subset B_{3\text{diam}(B_j\cap Q)}(\ox).
		\end{equation*}
		Therefore, considering the annuli $A_k:= B_{2^{-k}3\text{diam}(B_j\cap Q)}(\ox)\setminus{B_{2^{-k-1}3\text{diam}(B_j\cap Q)}(\ox)}$ for $k\geq 0$ we obtain
		\begin{align*}
			\int_{D_{2,j}}\frac{\dd\oy}{|\ox-\oy|^n} \leq \sum_{k\geq 0} \int_{A_k} \frac{\dd\oy}{|\ox-\oy|^n}\lesssim \text{diam}(B_j\cap Q) \sum_{k\geq 0}\frac{1}{2^k}\leq r_j.
		\end{align*}
		Hence
		\begin{equation*}
			\big\rvert \PP_{\Lambda}(\varphi\nu) (\ox) \big\rvert \lesssim |I_Q|<N<\infty,
		\end{equation*}
		and thus the local integrability of $\PP_{\Lambda}(\varphi\nu)$ follows. To prove the second assertion, we use $\mu(Q)\lesssim \ell(Q)^n$ together with Tonelli's theorem and Lemma \ref{lem3.1} to obtain
		\begin{align*}
			\int_Q \big\rvert \PP_{\Lambda}(\varphi\nu) (\ox) \big\rvert \dd\ox &= \int_Q \bigg\rvert \int_{F\cap Q} P_{\Lambda}(\ox,\oy) \varphi(\oy) \dd\nu(\oy) \bigg\rvert \dd\ox\\
			&\leq \int_{F\cap Q} \bigg( \int_Q \frac{\dd\ox}{|\ox-\oy|^n} \bigg) |\varphi (\oy)|\dd\mu(\oy)\leq\ell(Q)\int_{F\cap Q}|\varphi(\oy)|\dd\mu(\oy)\\
			&\leq \ell(Q)\mu(Q)\lesssim \ell(Q)^{n+1} = \pazocal{L}^{n+1}(Q),
		\end{align*}
		Therefore,
		\begin{equation*}
			\frac{1}{\pazocal{L}^{n+1}(\frac{1}{4}Q)}\int_{\frac{1}{4}Q}\big\rvert \PP_{\Lambda}(\varphi\nu) (\ox) \big\rvert \dd\ox \lesssim 1,
		\end{equation*}
		and the desired result follows.
	\end{proof}
	With the above lemma, we are able to prove a weaker \textit{localization-type} result for $\PP_{\Lambda}$, analogous to \cite[Theorem 3.1]{MPr}. It will be valid for our particular signed measure $\nu$ and cubes contained in $F\setminus{H}$. Recall that the latter inclusion implies, by property \textit{6} in Lemma \ref{lem5.3}, that for any $\ox\in Q$, if $R(\ox)$ is the cube centered at $\ox$ with side length $\ell(R)$, then $\mu(\lambda R(\ox))\leq \sqrt{n+1} L\, \ell(\lambda R)^n\simeq \ell(\lambda R)^n$, for any $\lambda >0$.
	\begin{lem}
		\label{lem5.15}
		Let $Q\subset F\setminus{H}$ be a cube and $\varphi$ a test function with $0\leq \varphi \leq 1$, $\|\nabla \varphi\|_\infty \leq \ell(Q)^{-1}$ and such that $\varphi \equiv 1$ on $Q$ and $\varphi \equiv 0$ on $(2Q)^c$. Then,
		\begin{equation*}
			|\PP_{\Lambda}(\varphi\nu)(\ox)|\lesssim 1, \qquad \forall \ox \in Q.
		\end{equation*}
	\end{lem}
	\begin{proof}
		Let us fix any $\ox\in Q$ and consider $\ox_0\in \frac{1}{2}Q$ the point obtained in Lemma \ref{lem5.14}. Observe that $\varphi(\ox)=\varphi(\ox_0)=1$. We rewrite $\PP_\Lambda(\varphi\nu)(\cdot)$ as follows
		\begin{align*}
			\PP_\Lambda(\varphi\nu)(\cdot) &= \int_{\mathbb{R}^{n+1}\setminus{4Q}} P_{\Lambda}(\cdot,\oy)(\varphi(\oy)-\varphi(\ox))\dd\nu(\oy)\\
			&\hspace{1cm}+\int_{4Q} P_{\Lambda}(\cdot,\oy)(\varphi(\oy)-\varphi(\ox))\dd\nu(\oy)+\varphi(\ox)\int P_\Lambda(\cdot,\oy)\dd\nu(\oy),
		\end{align*}
		and we apply this decomposition to  $\PP_\Lambda(\varphi\nu)(\ox)$ and $\PP_\Lambda(\varphi\nu)(\ox_0)$. Therefore\medskip\\
		\begin{align*}
			|\PP_{\Lambda}&(\varphi\nu)(\ox)-\PP_{\Lambda}(\varphi\nu)(\ox_0)|\\
			&\leq \bigg\rvert \int_{\mathbb{R}^{n+1}\setminus{4Q}} (P_{\Lambda}(\ox,\oy)-P_{\Lambda}(\ox_0,y))(\varphi(\oy)-\varphi(\ox))\dd\nu(\oy) \bigg\rvert\\
			&\hspace{0.5cm}+\bigg\rvert \int_{4Q} P_{\Lambda}(\ox,\oy)(\varphi(\oy)-\varphi(\ox))\dd\nu(\oy) \bigg\rvert + \bigg\rvert \int_{4Q} P_{\Lambda}(\ox_0,\oy)(\varphi(\oy)-\varphi(\ox))\dd\nu(\oy) \bigg\rvert\\
			&\hspace{0.5cm}+\bigg\rvert \varphi(\ox)\int P_\Lambda(\ox,\oy)\dd\nu(\oy) \bigg\rvert + \bigg\rvert \varphi(\ox)\int P_\Lambda(\ox_0,\oy)\dd\nu(\oy) \bigg\rvert =: I_1+I_2+I_3+I_4+I_5.
		\end{align*}
		Regarding $I_4$, observe that Lemma \ref{lem5.13} implies
		\begin{equation*}
			|\varphi(\ox)|\bigg\rvert \int P_{\Lambda}(\ox,\oy)\dd\nu(\oy) \bigg\rvert = |\pazocal{P}_\Lambda\nu(\ox)|\lesssim 1.
		\end{equation*}
		The same estimate holds for $I_5$. To study $I_2$ observe that
		\begin{equation*}
			\bigg\rvert \int_{4Q} P_{\Lambda}(\ox,\oy)(\varphi(\oy)-\varphi(x))\dd\nu(\oy) \bigg\rvert\leq c_1 \|\nabla\varphi\|_\infty\int_{4Q} \frac{\dd\mu(\oy)}{|\ox-\oy|^{n-1}} \lesssim 1,
		\end{equation*}
		where we have used that $\ox\in Q\subset F\setminus{H}$, so that we are able to integrate over decreasing annuli using the $n$-growth of $\mu$ for cubes centered at $\ox$. The same holds for $I_3$ because $\ox_0\in \frac{1}{2}Q$ and $\varphi(\ox)=\varphi(\ox_0)$, so that $ I_3 \lesssim 1$. Finally, for $I_1$ we apply property \textit{2} of Lemma \ref{lem3.1} and integrate over increasing annuli to deduce
		\begin{align*}
			I_1 \lesssim 2\ell(Q)\|\varphi\|_\infty  \int_{\mathbb{R}^{n+1}\setminus{4Q}}\frac{\dd\mu(\oy)}{|\ox-\oy|^{n+1}}  \lesssim 1.
		\end{align*}
		Therefore, $|\PP_{\Lambda}(\varphi\nu)(\ox)|\leq |\PP_{\Lambda}(\varphi\nu)(\ox)-\PP_{\Lambda}(\varphi\nu)(\ox_0)|+|\PP_{\Lambda}(\varphi\nu)(\ox_0)|\lesssim 1$, and the desired result follows.
	\end{proof}
	The previous lemma implies the main result of this subsection, which concerns a particular weak boundedness property for $\PP_{\Lambda}$. We call it \textit{property 8}, in the sense that it is the additional property needed in Lemma \ref{lem5.3} in order to apply a $Tb$-theorem.
	\begin{cor}[\textit{Property 8}]
		\label{cor5.16}
		Let $Q\subset \mathbb{R}^{n+1}$ be a cube contained in $F\setminus H$ with $M$-thin boundary. Then,
		\begin{equation*}
			|\langle \PP_{\Lambda,\mu}(\chi_Q b), (\chi_Q b) \rangle| \leq c_8 \mu(2Q),
		\end{equation*}
		for some $c_8>0$ constant.
	\end{cor}
	\begin{proof}
		Fix a cube $Q\subset F\setminus{H}$. We shall assume that $Q$ is open, since the involved measures $\mu$ and $\nu$ are null on sets of zero Lebesgue measure. Observe that since the center of $Q$ does not belong to $H$, in particular we have $\mu(2Q)\lesssim L\ell(2Q)^n$. Take a test function $\varphi$ with $0\leq \varphi \leq 1$, $\|\nabla \varphi\|_\infty \leq \ell(Q)^{-1}$ and such that $\varphi \equiv 1$ on $Q$ and $\varphi \equiv 0$ on $(2Q)^c$. Then,
		\begin{align*}
			|\langle \PP_{\Lambda,\mu}(\chi_Q b), (\chi_Q b) \rangle| &\leq |\langle \PP_{\Lambda,\mu}(\varphi b), (\chi_Q b) \rangle| + |\langle \PP_{\Lambda,\mu}((\varphi-\chi_Q )b), (\chi_Q b) \rangle| =: A+B.
		\end{align*}
		To estimate $A$ we observe that $\PP_{\Lambda,\mu}(\varphi b) = \PP_{\Lambda}(\varphi\nu)$ and apply directly Lemma \ref{lem5.15} to deduce $A\leq c_1\int_Q|\PP_{\Lambda}(\varphi\nu)(\ox)|\dd\mu(\ox) \lesssim \mu(Q).$\medskip\\
		To estimate $B$, set $\varphi_1:=(\varphi-\chi_Q)b$ and $\varphi_2:=\chi_Qb$. Observe that $\text{supp}(\varphi_1)\subset \overline{2Q}\setminus{Q}\subset \mathbb{R}^{n+1}\setminus{Q}$ and $\text{supp}(\varphi_2)\subset \overline{Q}$. We now apply \cite[Lemma 5.23]{To3} with $\Omega_1:=Q$ and $\Omega_2:=\mathbb{R}^{n+1}\setminus{\overline{\Omega_1}}$. The proof of the previous result is almost identical in our context: one just needs to change the function $d(\ox)^{-1/2}$ appearing in the previous reference by $d(\ox)^{-n/2}$. In any case, we deduce
		\begin{equation*}
			B\lesssim \|\varphi_1\|_{L^2(\mu)}\,\|\varphi_2\|_{L^2(\mu)} \leq c_1 \,\mu\big(\overline{2Q}\setminus{Q}\big)^{1/2}\mu\big(\overline{Q}\big)^{1/2}\leq c_1\,\mu(2Q),
		\end{equation*}
		and we are done.
	\end{proof}
	
	The previous corollary suffices to prove an analogous $Tb$-theorem to that of \cite[Ch.5]{To3}, since the weak boundedness property needs only to be applied to cubes contained in parallelepipeds $\Delta:=P\cap S$ with $P,S$ cubes having $M$-thin boundary (and all their dyadic children too) that belong to $\DD_1:=\DD(\ow_1)$ and $\DD_2:=\DD(\ow_2)$ respectively for some $\ow_1,\ow_2\in \Omega$. Moreover, $P\subset F\setminus{(H_{\DD_1} \cup T_{\DD_1})}$ and $S\subset F\setminus{(H_{\DD_2} \cup T_{\DD_2})}$. Therefore, since in particular $H\subset H_{\DD_1}\cap H_{\DD_2}$, the cubes contained in $\Delta$ do not intersect $H$. Hence, Corollary \ref{cor5.16} can be applied to such cubes, yielding the following result:
	
	\begin{thm}
		\label{thm5.17}
		Let $\mu$ be a positive finite Borel measure on $\mathbb{R}^{n+1}$ supported on a compact set $F\subset \mathbb{R}^{n+1}$. Assume there is a finite measure $\nu$ and, for each $\ow \in \Omega$, two subsets $H_{\DD(\ow)}, T_{\DD(\ow)} \subset \mathbb{R}^{n+1}$ consisting of dyadic cubes in $\DD(\ow)$ so that:
		\begin{enumerate}[leftmargin=*,itemsep=0.1cm]
			\item[\textit{1}.] $ \nu=b\mu$ for some $b$ with $\|b\|_\infty \leq c_b$.
			\item[\textit{2}.] Every ball satisfying $\mu(B_r)>Lr^n$ is contained in $\bigcap_{\ow\in \Omega} H_{\DD(\ow)}$.
			\item[\textit{3}.] If $Q\in \DD(\ow)$ is such that $Q\not\subset T_{\DD(\ow)}$, then $\mu(Q)\leq c_T|\nu(Q)|$.
			\item[\textit{4}.] $\mu\big( H_{\DD(\ow)}\cup T_{\DD(\ow)} \big) \leq \delta_0\mu(F)$, for some $\delta_0\in(0,1)$ .
			\item[\textit{5}.] If $\PP_{\nu,\ast}$ and $\PP^\ast_{\nu,\ast}$ denote the maximal convolution operators associated and $P$ and $P^\ast$ respectively with respect to $\nu$, we have for each $\ow\in \Omega$,
			\begin{equation*}
				\int_{F\setminus{H_{\DD(\ow)}}}\PP_{\ast}\nu(\oy)\dd\mu(\oy)\leq c_\ast \mu(F), \qquad \int_{F\setminus{H_{\DD(\ow)}}}\PP^\ast_{\ast}\nu(\oy)\dd\mu(\oy)\leq c_\ast\mu(F).
			\end{equation*}
			\item[6.] Let $Q\subset \mathbb{R}^{n+1}$ be a cube contained in $F\setminus{\bigcap_{\ow\in \Omega} H_{\DD(\ow)}}$ with $M$-thin boundary, and $\Lambda:\mathbb{R}^{n+1}\to[0,\infty)$ a 1-Lipschitz function satisfying
			\begin{equation*}
				\Lambda(\ox) \geq \text{\normalfont{dist}}\big(\ox, \mathbb{R}^{n+1}\setminus{\big( H_{\DD(\ow)}\cup \pSS \big)} \big), \qquad\forall \ox\in \mathbb{R}^{n+1},
			\end{equation*}
			where $\pSS$ is the exceptional set defined in \eqref{eq5.13}. Then,
			\begin{equation*}
				|\langle \PP_{\Lambda,\mu}(\chi_Q b), (\chi_Q b) \rangle| \leq c_W \mu(2Q),
			\end{equation*}
			where $\PP_\Lambda$ is the operator associated to the suppressed kernel $P_{\Lambda}$ defined in \eqref{eq3.1}.
		\end{enumerate}
		Then, there is $G\subset F\setminus{ \bigcap_{\ow\in \mathbb{R}^{n+1}} \big( H_{\DD(\ow)} \cup T_{\DD(\ow)} \big) }$ compact, and positive constants $A_1, A_2$ and $A_3$ so that
		\begin{enumerate}[leftmargin=*,itemsep=0.1cm]
			\item[\textit{i}.] $\mu(F)\leq A_1 \mu(G)$,
			\item[\textit{ii}.] $\mu|_G(B_r(\ox))\leq A_2r^n$, for every ball $B_r(\ox)$,
			\item[\textit{iii}.] $\|\PP_{\mu|_G}\|_{L^2(\mu|_G)\to L^2(\mu|_G)}\leq A_3$. 
		\end{enumerate}
		The letters $c_b, L, c_T, \delta_0, c_\ast, c_W$ and $M$ denote constants.
	\end{thm}
	
	\begin{rem}
		\label{rem5.3}
		Let us give some details on how to prove Theorem \ref{thm5.17}, although the arguments to follow are just those given in the proofs of \cite[Theorem 5.1]{To3} or \cite[Theorem 7.1]{Vo}, using, essentially, the weak boundedness property instead of the anti-symmetry of the Cauchy and Riesz kernels. Let us recall that we wrote, for $N\geq 4$ integer,
		\begin{equation*}
			\Omega:=\big[-2^{N-1},2^{N-1}\big)^{n+1}.
		\end{equation*}
		Fix $\ow\in \Omega$. Recall $F\subset Q^0(\ow):=\ow+[-2^N,2^N)^{n+1}$. A dyadic cube $Q\in Q^0(\ow)$ with $\mu(Q)\neq 0$ is called \textit{terminal} if $Q\subset H_{\DD(\ow)}\cup T_{\DD(\ow)}$ and we write $Q\in \DD^{\text{term}}(\ow)$. Otherwise is called \textit{transit} and we write $Q\in \DD^{\text{tr}}(\ow)$. With this, one considers a \textit{martingale decomposition} of a function $f\in L^1_{\text{loc}}(\mu)$ in terms of $Q^0(\ow)$. For any cube $Q\subset \mathbb{R}^{n+1}$ with $\mu(Q)\neq 0$ one sets
		\begin{equation*}
			\langle f \rangle_Q:=\frac{1}{\mu(Q)}\int_Q f\dd\mu
		\end{equation*}
		and defines the operator
		\begin{equation*}
			\Xi f := \frac{\langle f \rangle_{Q^0(\ow)}}{\langle b \rangle_{Q^0(\ow)}}b.
		\end{equation*}
		It is clear that $\Xi f\in L^2(\mu)$ if $f\in L^2(\mu)$ and $\Xi^2 = \Xi$. Moreover, the definition of $\Xi$ does not depend on the choice of $\ow\in \Omega$. 
        
		If $Q\in \DD(\ow)$, the set of at most $2^{n+1}$ dyadic children of $Q$ whose $\mu$-measure is not null is denoted by $\pazocal{CH}(Q)$. For any $Q\in \DD^{\text{tr}}(\ow)$ and $f\in  L^1_{\text{loc}}(\mu)$ we define the function $\Delta_Q f$ as
		\begin{align*}
			\Delta_Q f := 
			\begin{cases} 
				\hspace{1.07cm} 0 & \text{in } \, \mathbb{R}^{n+1}\setminus{\bigcup_{R\in \pazocal{CH}(Q)}R}, \\
				\Big( \frac{\langle f \rangle_R}{\langle b \rangle_R}-\frac{\langle f \rangle_Q}{\langle b \rangle_Q} \Big)b & \text{in }\, R \,\text{ if }\, R\in \pazocal{CH}(Q)\cap \DD^{\text{tr}}(\ow), \\
				\hspace{0.5cm} f-\frac{\langle f \rangle_Q}{\langle b \rangle_Q}b & \text{in }\, R \,\text{ if }\, R\in \pazocal{CH}(Q)\cap \DD^{\text{term}}(\ow). 
			\end{cases}
		\end{align*}
		The fundamental properties of the operators $\Xi$ and $\Delta_Q$ are proved in \cite[Lemmas 5.10, 5.11]{To3} and, essentially, allow to decompose $f\in L^2(\mu)$ as
		\begin{equation*}
			f = \Xi f + \sum_{Q\in \DD^{\text{tr}}} \Delta_Q f,
		\end{equation*}
		where the sum is unconditionally convergent in $L^2(\mu)$ and, in addition,
		\begin{equation*}
			\|f\|_{L^2(\mu)}^2 \approx \|\Xi f\|^2_{L^2(\mu)}+\sum_{Q\in \DD^{\text{tr}}} \|\Delta_Q f\|_{L^2(\mu)}^2.
		\end{equation*}
		At this point, one of the fundamental steps of the proof of the $Tb$-theorem consists in using the above decomposition to estimate the $L^2(\mu)$ norm of the suppressed operator $\pazocal{P}_\Lambda$ when applied to the so called \textit{good functions}. To define them, we need to introduce first \textit{good} and \textit{bad} cubes. Let $\ow_1,\ow_2\in \Omega$ and consider $\DD_1:=\DD(\ow_1)$, $\DD_2:=\DD(\ow_2)$ two dyadic lattices. We consider as in \cite[Definition 6.2]{NTVo2} the parameter
		\begin{equation*}
			\alpha := \frac{1}{2(n+1)},
		\end{equation*}
		and we will say that $Q\in \DD_1^{\text{tr}}$ is \textit{bad with respect to} $\DD_2$ if either
		\begin{enumerate}[leftmargin=*,itemsep=0.1cm]
			\item[\textit{1}.] there is $R\in \DD^{\text{tr}}_2$ such that $\text{dist}(Q,\partial R)\leq \ell(Q)^{\alpha}\ell(R)^{1-\alpha}$ and $\ell(R)\geq 2^m\ell(Q)$, for some positive integer $m$ to be fixed later, or
			\item[\textit{2}.] there is $R\in \DD^{\text{tr}}_2$ such that $2^{-m}\ell(Q)\leq \ell(R) \leq 2^m \ell(Q)$, $\text{dist}(Q,R)\leq 2^m\ell(Q)$ and, at least, one of the children of $R$ does not have $M$-thin boundary.
		\end{enumerate}
		If $Q$ is not bad, then we say that it is \textit{good with respect to} $\DD_2$. An important property regarding bad cubes is that they do not appear very often in dyadic lattices. More precisely, given $\varepsilon_b>0$ arbitrarily small, if $m$ and $M$ are chosen big enough, then for each fixed $Q\in \DD_1$, the probability that it is bad with respect to $\DD_2$ is not larger than $\varepsilon_b$. That is,
		\begin{equation*}
			\mathbb{P}\big(\big\{ \ow_2\in \Omega\,:\, Q\in \DD_1 \text{ is bad with respect to } \DD(\ow_2) \big\}\big)\leq \varepsilon_b.
		\end{equation*}
		In light of the above notions, we say that a function $f\in L^2(\mu)$ is $\DD_1$\textit{-good with respect to} $\DD_2$ if $\Delta_Q f = 0$ for all bad cubes $Q\in \DD^{\text{tr}}_1$ (with respect to $\DD_2$). Now one proceeds as follows: define the $1$-Lipschitz function
		\begin{equation*}
			\Lambda_{\DD(\ow)}(\ox):=\text{dist}\big(\ox, \mathbb{R}^{n+1}\setminus{W_{\DD(\ow)}}\big), \qquad \ow\in \Omega.
		\end{equation*}
		It satisfies $\Lambda_{\DD(\ow)}(\ox)\geq \text{dist}\big(\ox, \mathbb{R}^{n+1}\setminus\big( H_{\DD(\ow)}\cup \pazocal{S} \big)  \big)$, so that Lemma \ref{lem5.13} can be applied to $\Lambda_{\DD(\ow)}$, and then we have:
		\begin{lem} $($\cite[Lemma 5.13]{To3}$)$
			\label{lem5.18}
			Let $\DD_1=\DD(\ow_1)$ and $\DD_2=\DD(\ow_2)$ with $\ow_1,\ow_2\in \Omega$. Given $\varepsilon>0$, let $\Lambda:\mathbb{R}^{n+1}\to [\varepsilon, \infty)$ be a 1-Lipschitz function such that
			\begin{equation*}
				\Lambda(\ox) \geq \max\big\{ \Lambda_{\DD_1}(\ox), \Lambda_{\DD_2}(\ox) \big\}, \qquad \forall \ox\in \mathbb{R}^{n+1}.
			\end{equation*}
			Then, if $f\in L^2(\mu)$ is $\DD_1$-good with respect to $\DD_2$ and $g\in L^2(\mu)$ is $\DD_2$-good with respect to $\DD_1$,
			\begin{equation*}
				|\langle \pazocal{P}_{\Lambda}(f\mu), g \rangle| \lesssim \|f\|_{L^2(\mu)}\|g\|_{L^2(\mu)},
			\end{equation*}
			where the implicit constant depends on $c_b,L,c_T,\delta_0,c_\ast,c_W$ and $\varepsilon_b$, but not on $\varepsilon$.
		\end{lem}
		The proof of this result uses the above martingale decomposition of the functions $f,g$, so that one is left to estimate:
		\begin{align*}
			\langle \pazocal{P}_{\Lambda}(f\mu), g \rangle = \langle \pazocal{P}_{\Lambda}(\Xi f\mu), \Xi g \rangle + \langle \pazocal{P}_{\Lambda}(\Xi f\mu), g \rangle &+ \langle \pazocal{P}_{\Lambda}(f\mu), \Xi g \rangle\\
			&+\sum_{Q\in \DD^{\text{tr}}_1, R\in \DD^{\text{tr}}_2} \langle \pazocal{P}_{\Lambda}((\Delta_Q f)\mu), \Delta_R g \rangle.
		\end{align*}
		If our kernel were anti-symmetric, the first term of the right-hand side would be null. Although this is not our case, it can still be estimated as follows (notice that the weak boundedness property will not be used in the arguments below):
		\begin{equation*}
			\langle \pazocal{P}_{\Lambda}(\Xi f\mu), \Xi g \rangle \leq \| \pazocal{P}_{\Lambda}(\Xi f\mu)\|_{L^2(\mu)}\|\Xi g\|_{L^2(\mu)}.
		\end{equation*}
		Observe that $\text{supp}(\mu)\subset F \subset Q^0(\ow_1)\cap Q^0(\ow_2)$ and by definition
		\begin{equation*}
			\Xi g := \frac{\langle g \rangle_{Q^0(\ow_2)}}{\langle b \rangle_{Q^0(\ow_2)}}b.
		\end{equation*}
		This implies, since $Q^0(\ow_2)$ is always a transit cube (this is easy to see just arguing by contradiction and using assumption \textit{4} in Theorem \ref{thm5.17}),
		\begin{align*}
			\|\Xi g\|_{L^2(\mu)} &\leq \frac{\mu(Q^0(\ow_2))^{1/2}}{|\nu(Q^0(\ow_2))|}\bigg( \int_F |b|^2\dd\mu \bigg)^{1/2}\|g\|_{L^2(\mu)}\\
			&\leq c_T^{1/2}\bigg(  \frac{1}{|\nu(Q^0(\ow_2))|}\int_F|b|^2\dd\mu \bigg)^{1/2}\|g\|_{L^2(\mu)}\leq c_1c_T\|g\|_{L^2(\mu)}.
		\end{align*}
		Moreover, by Lemma \ref{lem5.13} and using that $Q^0(\ow_1)$ is a transit cube, we deduce
		\begin{align*}
			\| \PP_{\Lambda}(\Xi f\mu)\|_{L^2(\mu)} &= \frac{|\langle f \rangle_{Q^0(\ow_1)}|}{|\langle b \rangle_{Q^0(\ow_1)}|} \|\PP_{\Lambda}(b\mu)\|_{L^2(\mu)} \leq c_{\Lambda}\frac{|\langle f \rangle_{Q^0(\ow_1)}|}{|\langle b \rangle_{Q^0(\ow_1)}|} \mu(Q^0(\ow_1))^{1/2}\\
			&\leq c_{\Lambda}c_T |\langle f \rangle_{Q^0(\ow_1)}|\mu(Q^0(\ow_1))^{1/2}\leq c_{\Lambda}c_T\|f\|_{L^2(\mu)}.
		\end{align*}
		Hence
		\begin{equation*}
			|\langle \pazocal{P}_{\Lambda}(\Xi f\mu), \Xi g \rangle|\lesssim \|f\|_{L^2(\mu)}\|g\|_{L^2(\mu)}.
		\end{equation*}
		The terms $\langle \pazocal{P}_{\Lambda}(\Xi f\mu), g \rangle$ and $\langle \pazocal{P}_{\Lambda}(f\mu), \Xi g \rangle$ can be estimated similarly (see \cite[p. 155]{To3}), and it is important to notice that to do so one uses Lemma \ref{lem5.13} for the adjoint suppressed operator $\PP^\ast_{\Lambda}$. Hence, the only term left is
		\begin{equation*}
			\sum_{Q\in \DD^{\text{tr}}_1, R\in \DD^{\text{tr}}_2} \langle \pazocal{P}_{\Lambda}((\Delta_Q f)\mu), \Delta_R g \rangle.
		\end{equation*}
		The above sum is studied in \cite[\textsection 5.6, \textsection 5.7, \textsection 5.8 \& \textsection 5.9]{To3}, and the arguments can be followed analogously up to \cite[\textsection 5.9]{To3}. Obviously, there are changes that need to be done regarding the dimensionality. 
		Such modifications were already done in a general multidimensional setting in the study carried out in \cite{Vo} for Riesz kernels. In any case, it is in \cite[\textsection 5.9]{To3} where the weak boundedness property needs to be invoked in order to deal with expressions of the form
		\begin{equation*}
			\langle \PP_{\Lambda,\mu}(\chi_\Delta b), (\chi_\Delta b) \rangle,
		\end{equation*}
		where $\Delta$ is a certain parallelepiped introduced at the beginning of Subsection \ref{subsec5.4}. Such expressions were already tackled in \cite{NTVo2} and \cite{HyMar}, and they can be deduced from the estimate
		\begin{equation*}
			|\langle \PP_{\Lambda,\mu}(\chi_Q b), (\chi_Q b) \rangle| \lesssim \mu(2 Q),
		\end{equation*}
		where in our setting we may assume $Q$ to be contained in $F\setminus{H}$ and with $M$-thin boundary. This precise bound is covered by assumption \textit{6} in the statement of Theorem \ref{thm5.17}. From this point on, the rest of the argument can be followed as in the remaining sections of \cite[Ch.5]{To3}. Let us also remark that the proof can be followed almost identically (apart from some dimensional changes) taking into account that we have constructed \textit{unique} measures $\mu$ and $\nu$ so that assumption \textit{5} holds for the operators $\pazocal{P}$ and $\pazocal{P}^\ast$ \textit{simultaneously}. This enables to obtain relation \eqref{eq5.14}, which is essential to carry out the final probabilistic arguments in the proof of the $Tb$-theorem found in \cite[Ch.5]{To3} or \cite{Vo}.
	\end{rem}
	
	In any case, in light of Theorem \ref{thm5.17} and also bearing in mind the argument of \eqref{eq5.4} and that assumptions $\Aone$ and $\Atwo$ are superfluous, we have proved the following:
	\begin{thm}
		\label{thm5.19}
		Let $E\subset \mathbb{R}^{n+1}$ be a compact set satisfying $\Athree$. Then,
		\begin{equation*}
			\widetilde{\gamma}(E)\approx \widetilde{\gamma}_+(E).
		\end{equation*}
	\end{thm}
	
	\section{The general comparability in \mathinhead{\mathbb{R}^{n+1}}{}}
	\label{sec6}
	
	The main goal of this section is to obtain similar results to Theorem \ref{thm5.19} but removing assumption $\Athree$. We will be able to do this for any compact set in $\mathbb{R}^{n+1}$.

	\subsection{The capacity of some parallelepipeds in \mathinhead{\mathbb{R}^{n+1}}{}}
	\label{subsec6.1}
	First, let us present some preliminary results that extend that of \cite[Proposition 6.1]{MPr}. Let us remark that throughout the forthcoming discussion, any parallelepiped will have sides parallel to the coordinate axes.
	
	\begin{lem}
		\label{lem6.1}
		Let $a\in \mathbb{R}$ and $R\subset \mathbb{R}^n\times\{a\}$ be a parallelepiped contained in the affine hyperplane $\{t=a\}$. Then
		\begin{equation*}
			\widetilde{\gamma}_{+}(R)\gtrsim \pazocal{H}^{n}(R).
		\end{equation*}
	\end{lem}
	\begin{proof}
		Let $\mu:=\pazocal{H}^n|_{R}$ and pick any $\overline{x}=(x,t)$ with $|t-a|>0$. Observe that
		\begin{align*}
			P_{\text{sy}}\ast \mu(x,t)&=\int_R\frac{|t-a|}{\big[ (t-a)^2+|x-y|^2 \big]^{\frac{n+1}{2}}}\text{d}\pazocal{H}^n(y) =\int_{R-x}\frac{|t-a|}{\big[ (t-a)^2+|u|^2 \big]^{\frac{n+1}{2}}}\text{d}\pazocal{H}^n(u),
		\end{align*}
		where $u:=y-x$ and $R-x$ denotes a translation of $R$ with respect to the vector $(-x,0)\in \mathbb{R}^{n+1}$. Let us pick $D_{r_0}$ an $n$-dimensional ball embedded in $\mathbb{R}^n\times\{a\}$, centered at $(0,a)\in\mathbb{R}^{n}\times\{a\}$ and with radius $r_0=r_0(\overline{x})$ big enough so that $R-x\subset D_{r_0}$. Then, there exists a dimensional constant $C>0$ so that
		\begin{align*}
			P_{\text{sy}}\ast \mu(x,t)&\leq \int_{D_{r_0}}\frac{|t-a|}{\big[ (t-a)^2+|u|^2 \big]^{\frac{n+1}{2}}}\text{d}\pazocal{H}^n(u) = C\int_0^{r_0} \frac{|t-a|}{\big[ (t-a)^2+r^2 \big]^{\frac{n+1}{2}}}r^{n-1}\text{d}r\\
			&=C\int_0^{1} \frac{|t-a|r_0^{-1}}{\big[ (t-a)^2r_0^{-2}+\rho^2 \big]^{\frac{n+1}{2}}}\rho^{n-1}\text{d}\rho.
		\end{align*}
		where in the last step we have introduced the change of variables $r_0\rho=r$. Naming $\tau:=|t-a|r_0^{-1}$ the previous integral can be finally written as
		\begin{equation*}
			\int_0^1 \frac{\tau}{\big[ \tau^2+\rho^2 \big]^{\frac{n+1}{2}}}\rho^{n-1}\text{d}\rho,
		\end{equation*}
		which admits an explicit representation in terms of Gauss's (or Kummer's) hypergeometric function (use \cite[\textsection 1.4, eq.(13)]{Bu}, for example), obtaining the estimate
		\begin{align*}
			P_{\text{sy}}\ast \mu(x,t)& \leq \frac{C}{n}\bigg( \frac{\rho}{\tau} \bigg)^n {}_2F_1\Bigg( \frac{n}{2},\frac{n+1}{2};\frac{n+2}{2};-\bigg( \frac{\rho}{\tau} \bigg)^2 \Bigg)\Bigg\rvert_{\rho=0}^{\rho=1}\\
			&=\frac{C}{n} \cdot \frac{1}{\tau^n}\, {}_2F_1\bigg( \frac{n}{2},\frac{n+1}{2};\frac{n+2}{2};-\frac{1}{\tau^2} \bigg)\\
			&=\frac{C}{n} \cdot \frac{1}{\tau(1+\tau^2)^{(n-1)/2}}\, {}_2F_1\bigg( 1,\frac{1}{2};\frac{n+2}{2};-\frac{1}{\tau^2} \bigg).
		\end{align*}
		In the last step we have applied \cite[Eq.15.3.3]{AbS}. This last expression, thought as a function of $\tau>0$, is bounded uniformly on $(0,+\infty)$ with respect to a dimensional constant (and in fact attains its maximum when $\tau\to 0$). In other words, we deduce that $P_{\text{sy}}\ast \mu(x,t)\lesssim 1$ whenever $|t-a|>0$. On the other hand, it is clear that if $t=a$, then $P_{\text{sy}}\ast \mu(x,t)=0$, meaning that in general
		\begin{equation*}
			P_{\text{sy}}\ast \mu(x,t)\lesssim 1, \hspace{0.5cm} \forall (x,t)\in \mathbb{R}^{n+1}.
		\end{equation*}
		Therefore we conclude
		\begin{equation*}
			\gamma_{\text{sy},+}(R)\gtrsim \mu(R)= \pazocal{H}^n(R),
		\end{equation*}
		and using Theorem \ref{thm3.5} we obtain the desired result.
	\end{proof}
	
	The above lemma characterizes the $\widetilde{\gamma}_{+}$ capacity of parallelepipeds contained in affine hyperplanes of the form $\{t=a\}$. We will refer to such hyperplanes as \textit{horizontal} hyperplanes. Observe that the above bound combined with \cite[Lemma 4.1]{MPr} implies
	\begin{equation*}
		\pazocal{H}^n(R)\lesssim \widetilde{\gamma}_{+}(E)\leq \gamma_{+}(R) \leq  \gamma(R) \lesssim \pazocal{H}_\infty^n(R)\leq \pazocal{H}^n(R),
	\end{equation*}
	for any parallelepiped contained in a horizontal hyperplane. Hence, for such objects,
	\begin{equation*}
		\pazocal{H}^n(R)\approx \gamma_{+}(R) \approx \gamma(R).
	\end{equation*}
	
	\medskip
	Our next goal will be to study the capacity associated to parallelepipeds contained in \textit{vertical} hyperplanes, that is, sets of the form $\{x_i=a\}$ for some $a\in \mathbb{R}$ and some $i=1,\ldots,n$. Previous to that, we need to make an auxiliary construction: let us assume that we have a compact set $E\subset \mathbb{R}^{n+1}$ contained, for example, in $\{x_1=0\}$. Consider $T$ a distribution in $\mathbb{R}^{n+1}$ with $\text{supp}(T)\subseteq E$ as well as the maps
	\begin{align*}
		\pi:&\quad\mathbb{R}^{n+1}\longrightarrow \mathbb{R}^{n},  \hspace{2cm}\iota:\mathbb{R}^{n}\longrightarrow \hspace{0.27cm}\mathbb{R}^{n+1},\\
		&(x_1,x)\;\hspace{0.05cm}\longmapsto x  \hspace{2.95cm}x\hspace{0.22cm}\longmapsto (0,x)
	\end{align*}
	the canonical projection onto the last $n$ coordinates and the canonical inclusion into the hyperplane $\{x_1=0\}$. Let us take $\varphi\in \pazocal{C}^\infty(\mathbb{R}^n)$ a smooth function and define $\beta_M$, for some $M>0$ (a positive parameter that we may take as large as we need), a smooth bump function in $\mathbb{R}^{n+1}$ that equals 1 in an open $M$-neighborhood of $E$, i.e. $\beta_M\equiv 1$ in $\pazocal{U}_M(E)$. Also, we require that $\beta_M\equiv 0$ in $\mathbb{R}^{n+1}\setminus{\pazocal{U}_{2M}(E)}$ and $0\leq \beta_M \leq 1$. We set 
	\begin{equation*}
		\widetilde{\varphi}:=(\varphi \circ \pi)\cdot \beta_M,
	\end{equation*}
	that is a smooth extension of $\varphi|_E$ to $\mathbb{R}^{n+1}$, with compact support contained in $\pazocal{U}_{2M}(E)$. With this, we define the following distribution in $\mathbb{R}^{n}$ associated to $T$:
	\begin{equation*}
		\langle \widetilde{T}, \varphi \rangle := \langle T, \widetilde{\varphi} \rangle,\hspace{0.5cm} \varphi \in \pazocal{C}^\infty(\mathbb{R}^n),
	\end{equation*}
	which is a definition independent of the choice of $\beta_M$, since $\text{supp}(T)\subseteq E$. Observe that $\text{supp}(\widetilde{T})\subseteq \pi(E)$ (a compact set of $\mathbb{R}^n$) and also that for any $\psi\in \pazocal{C}^{\infty}(\mathbb{R}^{n+1})$,
	\begin{equation*}
		\langle T, \psi \rangle = \langle \widetilde{T}, \psi \circ \iota \rangle.
	\end{equation*}
	Notice that the previous identity is accurate because, even though $\psi$ and $\psi\circ \iota \circ \pi$ may not coincide in $\mathbb{R}^{n+1}\setminus{\{x_1=0\}}$, we have $\text{supp}(T)\subseteq E \subset \{x_1=0\}$, that ensures the validity of the above equality. Let us take $(\Phi_\varepsilon)_\varepsilon$ an approximation of the identity in $\mathbb{R}^n$ and set
	\begin{equation*}
		\widetilde{T}_\varepsilon:=\widetilde{T}\ast \Phi_\varepsilon,
	\end{equation*}
	for $0<\varepsilon\ll M$, so that the following inclusions hold
	\begin{equation*}
		\text{supp}(\widetilde{T}_\varepsilon)\subseteq \pazocal{V}_{2\varepsilon}(\pi(E))\subset \pi(\pazocal{U}_M(E)),
	\end{equation*}
	where the notation $\pazocal{V}$ is used to emphasize that we are considering an open neighborhood in $\mathbb{R}^n$ and not $\mathbb{R}^{n+1}$. So $(\widetilde{T}_\varepsilon)_\varepsilon$ defines a collection of signed measures in $\mathbb{R}^n$ that approximates $\widetilde{T}$ in a distributional sense:
	\begin{equation*}
		\lim_{\varepsilon\to 0} \;\langle \widetilde{T}_\varepsilon, \varphi \rangle = \langle \widetilde{T}, \varphi \rangle =  \langle T, \widetilde{\varphi} \rangle, \hspace{0.5cm} \varphi\in \pazocal{C}^\infty(\mathbb{R}^n).
	\end{equation*}
	Let us stress (although the reader may have already noticed), that the spaces of functions we have been considering are of the form $\pazocal{C}^\infty$ and not $\pazocal{C}^\infty_c$. This can be done since our distributions are compactly supported. Let us proceed by naming $\Sigma_\varepsilon:=\overline{\pazocal{V}_{2\varepsilon}(\pi(E))}$, and observe that for any $\psi\in \pazocal{C}^\infty(\mathbb{R}^{n+1})$
	\begin{align*}
		\langle T,\psi \rangle = \langle \widetilde{T}, \psi \circ \iota \rangle &= \lim_{\varepsilon\to 0} C\int_{\Sigma_E} \widetilde{T}_\varepsilon(x)\,\psi\circ \iota (x)\,\text{d}\pazocal{H}^n(x)\\
		&= \lim_{\varepsilon\to 0} C\int_{\Sigma_E} \widetilde{T}_\varepsilon\circ \pi (0,x)\,\psi(0,x)\,\text{d}\pazocal{H}^n(x)\\
		&= \lim_{\varepsilon\to 0} C\int_{\mathbb{R}^{n+1}} \widetilde{T}_\varepsilon\circ \pi (\overline{x})\,\psi(\overline{x})\,\text{d}\pazocal{H}^n|_{\iota(\Sigma_E)}(\overline{x}) \\
		&= \lim_{\varepsilon\to 0} \big \langle C(\widetilde{T}_\varepsilon\circ \pi)\pazocal{H}^n|_{\iota(\Sigma_\varepsilon)}, \psi \big\rangle,
	\end{align*}
	where $\widetilde{T}_\varepsilon\circ \pi$ is smooth and supported on $\mathbb{R}\times \Sigma_\varepsilon\subset \mathbb{R}\times \mathbb{R}^n$, and $C>0$ is some dimensional constant. Therefore, the previous construction implies the following remark:
	\begin{rem}
		\label{rem6.1}
		Given a distribution $T$ in $\mathbb{R}^{n+1}$ supported on $E\subset \{x_1=0\}$, there is a family of signed measures $(\nu_\varepsilon)_\varepsilon$ with $\text{supp}(\nu_\varepsilon)\subset \iota(\Sigma_\varepsilon)=:\iota\big(\overline{\pazocal{V}_{2\varepsilon}(\pi(E))}\big)$ of the form
		\begin{equation*}
			\nu_\varepsilon = \psi_\varepsilon\cdot \pazocal{H}^n|_{\iota(\Sigma_\varepsilon)},
		\end{equation*}
		where $\psi_\varepsilon$ is a smooth function that satisfies $\psi_\varepsilon\circ \iota = \widetilde{T}_\varepsilon$ (a function of $\pazocal{C}^\infty_c(\mathbb{R}^n)$ supported on $\Sigma_\varepsilon$), that is such that
		\begin{equation*}
			\langle T, \psi \rangle = \lim_{\varepsilon\to 0} \langle \nu_\varepsilon, \psi \rangle, \hspace{0.5cm} \psi \in \pazocal{C}^\infty(\mathbb{R}^n).
		\end{equation*}
		The reader may think of the above result as a construction that exploits the fact that the support of $T$ is contained in a hyperplane where $x_1$ is constant, so that we may approximate $T$ via an \textit{approximation of the identity with respect to the remaining variables} $x_2,\ldots,x_{n+1}$ (here $x_{n+1}=t$). Moreover, notice that if one starts assuming a condition of the form $\|P\ast T\|_\infty \leq 1$, by choosing $\varepsilon$ small enough we can assume, for example, $\|P\ast \nu_\varepsilon\|_\infty \leq 2$. Indeed, just fix any $\psi \in \pazocal{C}^\infty_c(\mathbb{R}^{n+1})$ and observe that
		\begin{equation*}
			\langle P \ast T, \psi \rangle = \langle T, P^\ast \ast \psi \rangle = \lim_{\varepsilon\to 0} \langle \nu_\varepsilon, P^\ast \ast \psi \rangle = \lim_{\varepsilon\to 0} \langle P\ast \nu_\varepsilon, \psi \rangle,
		\end{equation*}
		meaning that $\lim_{\varepsilon\to 0} |\langle P \ast \nu_\varepsilon, \psi \rangle|\leq \|\psi\|_{L^1(\mathbb{R}^{n+1})}$, which implies the desired estimate. In the previous argument we have used that $P^\ast \ast \psi$ is a smooth function, which can be argued thinking of $P^\ast$ as an $L^1_{\text{loc}}(\mathbb{R}^{n+1})$ function and thus as a regular distribution. This allows us to prove the following lemma:
	\end{rem}
	\begin{lem}
		\label{lem6.2}
		Let $E\subset \mathbb{R}^{n+1}$ be a compact set contained in an affine hyperplane of the form $\{x_i=a\}$, for some $a\in\mathbb{R}$ and $i=1,\ldots,n$. Then
		\begin{equation*}
			\gamma(E)=0.
		\end{equation*}
	\end{lem}
	\begin{proof}
		We shall assume that $E$ is contained in $\{x_1=0\}$ for the sake of simplicity. Consider $D_{r_0}$ an $n$-dimensional ball in $\{x_1=0\}$ centered at the origin and with radius $r_0$. Let us assume that $\gamma(D_{r_0})>0$ and reach a contradiction. If $\gamma(D_{r_0})>0$, there exists a distribution $T$ admissible for $\gamma_{\Theta^{1/2}}(D_{r_0})$ with $|\langle T, 1 \rangle|>0$. By Remark \ref{rem6.1} we obtain a family of signed measures that approximate $T$ of the form
		\begin{equation*}
			\nu_\varepsilon=\psi_\varepsilon\cdot \pazocal{H}^n|_{D_{r_0+2\varepsilon}},
		\end{equation*}
		for some $\psi_\varepsilon$ smooth, such that $\psi_\varepsilon\circ \iota \in \pazocal{C}^\infty_c(\mathbb{R}^n)$ with $\text{supp}(\psi_\varepsilon\circ \iota)\subset \pi(D_{r_0+2\varepsilon})$. We may also choose $\varepsilon$ small enough so that $\|P\ast \nu_\varepsilon\|_\infty \leq 2$. Since $|\langle T, 1 \rangle|>0$, we are able to pick some $\overline{x}_0=(x_0,t_0)\in D_{r_0+2\varepsilon}$ such that $|\psi_{\varepsilon}(\overline{x}_0)|>0$, as well as an $n$-dimensional ball centered at $\overline{x}_0$ with radius $\eta$ satisfying $D_\eta(\overline{x}_0)\subset D_{r_0+2\varepsilon}$ and $|\psi_\varepsilon|\geq A>0$ there, for some constant $A>0$ (this can be done by the continuity of $\psi_\varepsilon\circ \iota$).\medskip\\
		Proceed by fixing $Q=Q(\overline{x}_0)$ a cube in $\mathbb{R}^{n+1}$ centered at $\overline{x}_0$ with $4\ell(Q)\leq \eta$ and $\varphi\in \pazocal{C}^\infty_c(\mathbb{R}^{n+1})$ such that $0\leq \varphi \leq 1, \, \varphi|_Q\equiv 1, \varphi|_{\mathbb{R}^{n+1}\setminus{2Q}}\equiv 0$ and $\|\nabla \varphi\|_\infty\leq \ell(Q)^{-1}$. Also take $B_\xi(\overline{x}_0)$ a ball in $\mathbb{R}^{n+1}$ centered at $\overline{x}_0$ with radius $\xi$ so that $4\xi \leq \ell(Q)$, and name $D_\xi(\overline{x}_0):=B_{\xi}(\overline{x}_0)\cap \{x_1=0\}$, that is such that
		\begin{equation*}
			D_\xi(\overline{x}_0)\subset Q(\overline{x}_0)\cap \{x_1=0\} \subset D_{\eta}(\overline{x}_0).
		\end{equation*}
		We finally define the positive measure
		\begin{equation*}
			\mu':=A\cdot \pazocal{H}^n|_{D_\xi(\overline{x}_0)}
		\end{equation*}
		and observe that by the choice of $\varphi$ and the non-negativity of $P$, we have
		\begin{align*}
			\|P\ast \mu' \|_\infty &= \,\big \|P\ast A\,\pazocal{H}^n|_{D_\xi(\overline{x}_0)} \big \|_\infty= \big \|P\ast \varphi A\,\pazocal{H}^n|_{D_\xi(\overline{x}_0)} \big \|_\infty\\
			&\leq \big \|P\ast \varphi\,\psi_\varepsilon \pazocal{H}^n|_{D_{r_0+2\varepsilon}} \big \|_\infty = \|P\ast \varphi\nu_\varepsilon \|_\infty \lesssim 1,
		\end{align*}
		where the last inequality is due to the localization estimate \cite[Theorem 3.1]{MPr}. Therefore, we have constructed $D_\xi(\overline{x}_0)$ an $n$-dimensional ball that admits a positive measure $\mu$ supported on it, proportional to $\pazocal{H}^n|_{D_\xi(\overline{x}_0)}$ and with $\|P\ast \mu\|_\infty \leq 1$. Let us prove that this last condition is not possible. Indeed, for each $\overline{x}=(0,x_2,\ldots,x_n,t)\in D_{\xi}(\overline{x}_0)$ pick an $n$-dimensional ball $D_\rho(\overline{x})$ centered at $\overline{x}$ and with radius $\rho$ small enough so that $D_{\rho}(\overline{x})\subset D_{\xi}(\overline{x}_0)$. We write $\Delta_{\downarrow,\rho}(\overline{x})$ its \textit{lower temporal half}, which is obtained from the intersection $D_\rho(\overline{x})\cap \{t-s>0\}$; as well as $\Delta_{\uparrow,\rho}(\overline{x})$ the \textit{upper temporal half}, obtained from  $D_\rho(\overline{x})\cap \{t-s<0\}$. Now, integration in ($n$-dimensional) spherical coordinates yields
		\begin{align*}
			P\ast \mu(\overline{x})&\geq A\int_{\Delta_{\downarrow,\rho}(\overline{x})} \frac{t-s}{\big[ (t-s)^2+(x_2-y_2)^2+\cdots+(x_n-y_n)^2 \big]^{\frac{n+1}{2}}}\text{d}\pazocal{H}^n(y_2,\ldots,y_n,s)\\
			&\simeq A\int_{\Delta_{\uparrow,\rho}(0)} \frac{\tau}{\big[ \tau^2+u_2^2+\cdots+u_n^2 \big]^{\frac{n+1}{2}}}\text{d}u_2\cdots \text{d}u_n\text{d}\tau =A\cdot L \int_0^\rho\frac{\text{d}r}{r} = +\infty,
		\end{align*}
		where $L$ is a positive dimensional constant obtained from integration in the angular domain (notice that $\tau>0$ in $\Delta_{\uparrow,\rho}(0)$). So by the arbitrariness of $\overline{x}$ we get that $\|P\ast \mu\|_{L^\infty(\mu)}=+\infty$, since $\mu(D_\eta(\overline{x}_0))\simeq A\,\pazocal{H}^n(D_\eta(\overline{x}_0))>0$. But this is contradictory with $\|P\ast \mu\|_\infty\leq 1$, because such condition implies, in particular, $\|P\ast \mu\|_{L^\infty(\mu)}\lesssim 1$ (use a Cotlar type inequality analogous to that of \cite[Lemma 5.4]{MaP}). Therefore we conclude that $\gamma(D_{r_0})=0$ for any radius $r_0$, implying the desired result, by the monotonicity of $\gamma$.
	\end{proof}
	
	\begin{cor}
		\label{cor6.3}
		Let $R\subset \mathbb{R}^{n+1}$ be a parallelepiped and $\ell_1,\ldots,\ell_n,\ell_t$ its side lengths. Let $R_\uparrow$ denote the upper face of $R$ $($contained in a horizontal affine hyperplane$\,)$,
		\begin{enumerate}[leftmargin=*,itemsep=0.1cm]
			\item[\textit{1}$\,)$] if $\gamma(R_\uparrow)=0$, then $\gamma(R)=0$,
			\item[\textit{2}$\,)$] if $\ell_t\leq \min\{\ell_1,\ldots,\ell_n\}$, then $\gamma(R)\approx \gamma_{+}(R_\uparrow)$.
		\end{enumerate}
	\end{cor}
	\begin{proof}
		To prove \textit{1}, begin by noticing that $\gamma(R_\uparrow)=0$ implies $\pazocal{H}^n(R_\uparrow)=0$, by the comments made after Lemma \ref{lem6.1}. Embedding $R_\uparrow$ into $\mathbb{R}^n$ via $\pi_x:\mathbb{R}^{n+1}\to\mathbb{R}^n$, $\pi_x(\overline{x})=x$, the canonical projection onto the spatial coordinates; we deduce that $Q=\pi_x(R)$ is such that $\pazocal{L}^n(Q)=0$. But $Q$ is itself a parallelepiped of $\mathbb{R}^n$ (that we assume to be closed without loss of generality), i.e.
		\begin{equation*}
			Q=[a_1,b_1]\times \cdots \times [a_n,b_n], \hspace{0.5cm} \text{for some }\; a_i,b_i\in \mathbb{R},\, a_i\leq b_i.
		\end{equation*}
		So there must exist some $[a_j,b_j]$ such that $a_j=b_j$. Hence,
		\begin{equation*}
			R=[a_1,b_1]\times \cdots \times [a_{j-1},b_{j-1}]\times \{a_j\}\times [a_{j+1},b_{j+1}]\times \cdots \times [a_{n+1},b_{n+1}],
		\end{equation*}
		and thus $R\subset \{x_j=a_j\}$, a vertical hyperplane. Applying Lemma \ref{lem6.2} the result follows.\medskip\\
		Moving on to \textit{2}, let us assume $\ell_t>0$ (if not the result is trivial) and notice that by \cite[Lemma 4.1]{MPr} and Lemma \ref{lem6.1} we have
		\begin{align*}
			\gamma(R)\lesssim \pazocal{H}^{n}_\infty(R)\lesssim \bigg[ \frac{\ell_1}{\ell_t}\cdots \frac{\ell_n}{\ell_t}\cdot 1 \bigg]\cdot\ell_t^n \simeq \pazocal{H}^n(R_\uparrow)\lesssim \gamma_{+}(R_\uparrow),
		\end{align*}
		and we are done.
	\end{proof}
	\subsection{General comparability in \mathinhead{\mathbb{R}^{n+1}}{}}
	\label{subsec6.2}
	In light of  the previous results, we first obtain an estimate analogous to Theorem \ref{thm5.19} for parallelepipeds, without requiring assumption $\Athree$. Then, we establish a similar result for finite unions of closed cubes of the same size with disjoint interiors. Since $\Atwo$ can be assumed without loss of generality, the proof of the main theorem follows.
	\begin{lem}
		\label{lem6.4}
		Let $R\subset \mathbb{R}^{n+1}$ be a parallelepiped with sides parallel to the coordinate axes. Then,
		\begin{align*}
			\widetilde{\gamma}(R)\approx \widetilde{\gamma}_+(R).
		\end{align*}
	\end{lem}
	\begin{proof}
		Let us denote by $\ell_1,\ldots,\ell_n,\ell_t$ the side lengths of $R$. Observe that for our purposes, we may assume $R$ to be closed without loss of generality. Notice that by Lemmas \ref{lem6.1} and \ref{lem6.2} we already know the above result if any $\ell_i$ is null. So let us assume that
		\begin{equation}
			\label{eq6.1}
			\ell_i>0, \qquad \text{for each }\; i=1,\ldots,n,t.
		\end{equation}
		Moreover, we will also assume without loss of generality that $R$ is contained in $B_1(0)$, the unit ball of $\mathbb{R}^{n+1}$. The argument that follows is inspired by that presented in \cite[Ch.6]{Vo}. Let us begin by applying Theorem \ref{thm4.1} to the compact set $E:=R$ to obtain a first family of cubes $\{\QQ_{1},\ldots, \QQ_{N_1}\}$ satisfying properties $\Pone$ to $\Pfive$. We observe that:
		\begin{itemize}[leftmargin=*]
			\item Regarding property $\Pfour$, if $\ell_{\text{max}}:=\max \{\ell_1,\ldots,\ell_n,\ell_t\}$, we get for each $i_1=1,\ldots,N_1$,
			\begin{equation*}
				\text{diam}(\QQ_{i_1})\leq \frac{1}{10}\text{diam}(R)\leq \frac{\sqrt{n+1}}{10}\ell_{\max}\quad \text{and then} \quad \ell(\QQ_{i_1})\leq \frac{1}{10}\ell_{\text{max}}.
			\end{equation*}
			\item Regarding $\Pone$, we get $\frac{5}{8}\QQ_{i_1}\cap R \neq \varnothing$ for each $i_1=1,\ldots,N_1$. We wish to study the intersections $2\QQ_{i_1}\cap R$ using the previous property. Notice that $\frac{5}{8}\QQ_{i_1}\cap R \neq \varnothing$ implies that the sets $2\QQ_{i_1}\cap R$ are parallelepipeds such that their side with maximal length presents length between $\frac{11}{16}\ell(\QQ_{i_1})$ and $2\ell(\QQ_{i_1})$. On the other hand, we distinguish two cases for the sides with minimal length, depending on $\ell_{\text{min}}:=\min \{\ell_1,\ldots,\ell_n,\ell_t\}$:
		\end{itemize}
		\begin{itemize}
			\item[\textbf{A)}] If it happens
			\begin{equation*}
				\ell_{\text{min}}\geq \frac{11}{16}\max_{i_1=1,\ldots,N_1} \ell(\QQ_{i_1}),   
			\end{equation*}
			then the parallelepipeds $2\QQ_{i_1}\cap R$ have minimal side length between $\frac{11}{16}\ell(\QQ_{i_1})$ and $2\ell(\QQ_{i_1})$.
			\item[\textbf{B)}] If on the other hand
			\begin{equation*}
				\ell_{\text{min}} < \frac{11}{16}\max_{i_1=1,\ldots,N_1} \ell(\QQ_{i_1}),   
			\end{equation*}
			the parallelepipeds $2\QQ_{i_1}\cap R$ have minimal side length bigger or equal than the quantity $\min\big\{ \ell_{\text{min}}, \frac{11}{16}\ell(\QQ_{i_1}) \big\}$ and smaller or equal than $\min\big\{ \ell_{\text{min}}, 2\ell(\QQ_{i_1}) \big\}$, for each $i_1=1,\ldots,N_1$.
		\end{itemize}
		For each $i_1=1,\ldots, N_1$ we call $R_{i_1}:=2\QQ_{i_1}\cap R$ and consider two cases
		\begin{itemize}[itemsep=0.1cm]
			\item[\textbf{1)}] Assumption $\Athree$ is satisfied for $E:=R$.
			\item[\textbf{2)}] Assumption $\Athree$ is not satisfied for $E:=R$.
		\end{itemize}
		If \textbf{1)} occurs we are done, by Theorem \ref{thm5.19}. If \textbf{2)} occurs, notice that if \textbf{A)} also occurs we get that $R_1,\ldots R_{N_1}$ are parallelepipeds such that all of their respective sides have comparable lengths. In other words, there exists a cube $P_{i_1}\subset R_{i_1}$ so that $\widetilde{\gamma}(R_{i_1})\approx \widetilde{\gamma}(P_{i_1})$, for each $i_1=1,\ldots, N_1$. Then, using that $\Athree$ is not satisfied together with point \textit{2} in Corollary \ref{cor6.3}, we get
		\begin{align}
			\label{eq6.2}
			\widetilde{\gamma}(R)<C_1^{-1}\sum_{i=1}^{N_1}\widetilde{\gamma}(R_{i_1}) &\lesssim C_1^{-1}\sum_{i=1}^{N_1}\widetilde{\gamma}_+(R_{i_1}) \\
			\nonumber
			&\leq C_1^{-1} C_1\widetilde{\gamma}_+(R) = \widetilde{\gamma}_+(R),
		\end{align}
		where we have also used property $\Pthree$ in Theorem \ref{thm4.1}. So the only case left to study is when \textbf{2)} and \textbf{B)} happen simultaneously.
        
		In this case we apply again, for each $i_1=1,\ldots, N_1$, the splitting given by Theorem \ref{thm4.1} to all the parallelepipeds $R_{i_1}$. This way, we obtain a second family $\{\QQ_{i_11},\ldots,\QQ_{i_1N_{i_1}}\}$ associated to each $R_{i_1}$, satisfying properties $\Pone$ to $\Pfive$. Let us fix $i_1=1,\ldots, N_{1}$ and observe that:
		\begin{itemize}[leftmargin=*]
			\item Regarding property $\Pfour$, now we have for each $i_2=1,\ldots, N_{i_1}$,
			\begin{equation*}
				\text{diam}(\QQ_{i_1i_2})\leq \frac{1}{10}\text{diam}(R_{i_1})\leq \frac{\sqrt{n+1}}{10}\cdot 2 \ell (\QQ_{i_1})\quad \text{and then} \quad \ell(\QQ_{i_1i_2})\leq \frac{1}{5}\ell (\QQ_{i_1}).
			\end{equation*}
			This implies, in particular,
			\begin{equation}
				\label{eq6.3}
				2\ell(\QQ_{i_1i_2})\leq \frac{2}{5}\ell (\QQ_{i_1})<\frac{11}{16}\ell(\QQ_{i_1}).
			\end{equation}
			\item Property $\Pone$ now reads $\frac{5}{8}\QQ_{i_1i_2}\cap R_{i_1} \neq \varnothing$. Therefore, $2\QQ_{i_1i_2}\cap R_{i_1}$ are now parallelepipeds such that their maximal side length is between $\frac{11}{16}\ell(\QQ_{i_1i_2})$ and $2\ell(\QQ_{i_1i_2})$ (where we have applied \eqref{eq6.3}). For their side with minimal length, we distinguish:
		\end{itemize}
		\begin{itemize}
			\item[\textbf{A)}] If it happens
			\begin{equation*}
				\ell_{\text{min}}\geq \frac{11}{16}\max_{i_2=1,\ldots,N_{i_1}} \ell(\QQ_{i_1i_2}),   
			\end{equation*}
			then the parallelepipeds $2\QQ_{i_1i_2}\cap R_{i_1}$ have minimal side length between $\frac{11}{16}\ell(\QQ_{i_1i_2})$ and $2\ell(\QQ_{i_1i_2})$.
			\item[\textbf{B)}] If on the other hand
			\begin{equation*}
				\ell_{\text{min}} < \frac{11}{16}\max_{i_2=1,\ldots,N_{i_1}} \ell(\QQ_{i_1i_2}),   
			\end{equation*}
			now the parallelepipeds $2\QQ_{i_1i_2}\cap R_{i_1}$ have minimal side length bigger or equal than $\min\big\{ \ell_{\text{min}}, \frac{11}{16}\ell(\QQ_{i_1i_2}) \big\}$ and smaller or equal than $\min\big\{ \ell_{\text{min}}, 2\ell(\QQ_{i_1i_2}) \big\}$.
		\end{itemize}
		For each $i_1=1,\ldots, N_1$ and $i_2=1,\ldots, N_{i_1}$ we call $R_{i_1i_2}:=2\QQ_{i_1i_2}\cap R_{i_1}$. Two cases may occur for each $i_1$:
		\begin{itemize}[itemsep=0.1cm]
			\item[\textbf{1)}] Assumption $\Athree$ is satisfied for $E:=R_{i_1}$, that is
			\begin{align*}
				\widetilde{\gamma}(R_{i_1})\geq C_1^{-1}\sum_{i_2=1}^{N_{i_1}}\widetilde{\gamma}(R_{i_1i_2}). 
			\end{align*}
			\item[\textbf{2)}] Assumption $\Athree$ is not satisfied for $E:=R_{i_1}$, that is
			\begin{align*}
				\widetilde{\gamma}(R_{i_1}) < C_1^{-1}\sum_{i_2=1}^{N_{i_1}}\widetilde{\gamma}(R_{i_1i_2}). 
			\end{align*}
		\end{itemize}
		We are interested in proving $\widetilde{\gamma}(R_{i_1})\approx \widetilde{\gamma}_+(R_{i_1})$ for every $i_1$, since if this is the case, arguing as in \eqref{eq6.2} we are done.
		The only indices $i_1$ where we would not be able to deduce $\widetilde{\gamma}(R_{i_1})\approx \widetilde{\gamma}_+(R_{i_1})$ are, again, those for which \textbf{2)} and \textbf{B)} occur simultaneously. For such indices $i_1$ we would again apply the splitting provided by Theorem \ref{thm4.1} to every $R_{i_11},\ldots, R_{i_1N_{i_1}}$ and construct for each $R_{i_1i_2}$ a third family of cubes $\{\QQ_{i_1i_21},\ldots, \QQ_{i_1i_2N_{i_1i_2}}\}$. Now, for each $i_2=1,\ldots N_{i_1}$ one similarly obtains
		\begin{equation*}
			\ell(\QQ_{i_1i_2i_3})\leq \frac{1}{5}\ell(\QQ_{i_1i_2}), \qquad \forall i_3 = 1,\ldots, N_{i_1i_2},
		\end{equation*}
		and that $2\QQ_{i_1i_2i_3}\cap R_{i_1i_2}$ are now parallelepipeds such that their maximal side length is between $\frac{11}{16}\ell(\QQ_{i_1i_2i_3})$ and $2\ell(\QQ_{i_1i_2i_3})$ for each $i_3=1,\ldots, N_{i_1i_2}$. Regarding their minimal side lengths we again distinguish cases \textbf{A)} and \textbf{B)} in the current setting. Finally, for each $i_2$ one calls $R_{i_1i_2i_3}:=2\QQ_{i_1i_2i_3}\cap R_{i_1i_2}$ and studies two cases: 
		\begin{itemize}[itemsep=0.1cm]
			\item[\textbf{1)}] If assumption $\Athree$ is satisfied for $E:=R_{i_1i_2}$,
			\item[\textbf{2)}] or if assumption $\Athree$ is not satisfied for $E:=R_{i_1i_2}$.
		\end{itemize}
		Combinations \textbf{1)A)}, \textbf{1)B)} and \textbf{2)A)} lead to the estimate $\widetilde{\gamma}(R_{i_1i_2})\approx \widetilde{\gamma}_+(R_{i_1i_2})$. If one obtained such result for every $i_2$, proceeding as in \eqref{eq6.2} it would yield $\widetilde{\gamma}(R_{i_1})\approx \widetilde{\gamma}_+(R_{i_1})$, and we would be done. However, \textbf{2)} and \textbf{B)} may occur simultaneously for some indices $i_2$. In this setting, we would repeat the above splitting argument for the families $R_{i_1i_21},\ldots, R_{i_1i_2N_{i_1i_2}}$ associated to those $R_{i_1i_2}$ where \textbf{2)} and \textbf{B)} occur.
        
		We repeat this processes iteratively and we notice that after a number of steps large enough, say $S\geq 1$ steps, \textbf{B)} will no longer be satisfied. This is due to the fact that relation \eqref{eq6.1} ensures $\ell_{\text{min}}\geq \kappa>0$, for some positive constant $\kappa$ depending only on $R$; and also because the size of the cubes at each step strictly decreases. Indeed, at step $S$ of the iteration, property \textit{4} in Theorem \ref{thm4.1} and the fact that $R_{i_1i_2\cdots i_{S-1}}$ has maximal side length bounded by $2\ell(\QQ_{i_1i_2\cdots i_{S-1}})$ imply
		\begin{align*}
			\ell(\QQ_{i_1i_2\cdots i_S})\leq \frac{1}{5}\ell(\QQ_{i_1i_2\cdots i_{S-1}})\leq \cdots \leq \frac{1}{5^{S-1}}\frac{1}{10}\ell_{\text{max}}, \qquad \forall i_S = 1,\ldots, N_{i_1i_2\cdots i_{S-1}}.
		\end{align*}
		So choosing $S$ large enough, depending on $\kappa$, it is clear that \textbf{A)} will be satisfied instead of \textbf{B)}. Therefore, whether if \textbf{1)} occurs, or if \textbf{2)} and \textbf{A)} occur, one equally deduces (arguing as in \eqref{eq6.1} in the latter setting),
		\begin{equation*}
			\widetilde{\gamma}(R_{i_1i_2\cdots i_{S-1}})\approx \widetilde{\gamma}_+(R_{i_1i_2\cdots i_{S-1}}).
		\end{equation*}
		So tracing back all the steps of the iteration one gets, in general, $\widetilde{\gamma}(R)\approx \widetilde{\gamma}_+(R)$. 
	\end{proof}

    \begin{lem}
		\label{lem6.5}
		Let $E:= Q_1\cup Q_2 \cup \cdots \cup Q_N$ be a finite union of closed cubes of equal size with disjoint interiors and sides parallel to the coordinate axes. Then,
		\begin{align*}
			\widetilde{\gamma}(E)\approx \widetilde{\gamma}_+(E).
		\end{align*}
	\end{lem}
	\begin{proof}
        First, we observe that it is enough to prove the result if $Q_1,\ldots, Q_N$ are disjoint. Indeed, if they were not, write for each $j=1, \ldots, N$,
        \begin{equation*}
            Q_j:=\prod_{i=1,\ldots,n,t}[a_{ij},b_{ij}]  \quad \text{and} \quad m_{ij}:=\frac{b_{ij}-a_{ij}}{2}, \quad i=1,\ldots,n,t,
        \end{equation*}
        and for each $0<\lambda<1$ define the contracted concentric cube
        \begin{equation*}
            Q_j^{\lambda}:=\prod_{i=1,\ldots,n,t} [a_{ij}+\lambda m_{ij}, b_{ij}-\lambda m_{ij}],
        \end{equation*}
        as well as $E^{\lambda} := Q_1^{\lambda}\cup Q_2^{\lambda} \cup \cdots \cup Q_N^{\lambda}$. Observe that the family $\{Q_1^{\lambda}, \dots, Q_N^{\lambda}\}$ consists of closed disjoint cubes of the same size, and therefore $\widetilde{\gamma}(E^{\lambda}) \lesssim \widetilde{\gamma}_+(E^{\lambda})$. Moreover, by construction we get that for any $0<\lambda_1\leq \lambda_2<1$, $E^{\lambda_2}\subset E^{\lambda_1}$. Then, by monotonicity
        \begin{equation*}
           \sup_{0<\lambda<1}\widetilde{\gamma}(E^{\lambda})= \lim_{\lambda \to 0} \widetilde{\gamma}(E^{\lambda}) =  \widetilde{\gamma}(\mathring{E}),
        \end{equation*}
        where $\mathring{E}$ denotes the interior of $E$. We have a analogous identity for $\widetilde{\gamma}_+$. Since for our particular choice of $E$ we have $\widetilde{\gamma}_+(\mathring{E})\approx \widetilde{\gamma}_+(E)$ (see Remark \ref{rem6.3}), we are done:
        \begin{equation*}
             \widetilde{\gamma}_+(E)\lesssim \widetilde{\gamma}_+(\mathring{E})\leq\widetilde{\gamma}(\mathring{E}) = \sup_{0<\lambda<1}\widetilde{\gamma}(E^{\lambda}) \lesssim \sup_{0<\lambda<1}\widetilde{\gamma}_+(E^{\lambda}) \leq \widetilde{\gamma}_+(E).
        \end{equation*}
        
        So we shall assume that $Q_1,\ldots, Q_N$ are disjoint. In this context, the iterative scheme of the proof given for Lemma \ref{lem6.4} can be also applied, but now taking into account the parameter
		\begin{equation*}
			\delta:= \min_{i\neq j}\big\{ \text{dist}(Q_i,Q_j) \big\}>0.
		\end{equation*}
		More precisely, for example, at the first step of the iteration it may happen:
		\begin{enumerate}
			\item[\textbf{A)}] $E_{i_1}:=2\QQ_{i_1}\cap E$ has one connected component (and thus it is a parallelepiped) for every $i_1=1,\ldots,N_1$.
			\item[\textbf{B)}] Or there exists an index $i_1$ such that $E_{i_1}$ presents more than one connected component.
		\end{enumerate}
		We would also distinguish whether if:
		\begin{itemize}
			\item[\textbf{1)}] Assumption $\Athree$ is satisfied for $E$.
			\item[\textbf{2)}] Assumption $\Athree$ is not satisfied for $E$.
		\end{itemize}
		If \textbf{1)} happens, we are done. If \textbf{2)} happens, observe that if in turn \textbf{A)} occurred for every $i_1$, applying Lemma \ref{lem6.4} and the same estimates of \eqref{eq6.2} we would also be done. So we are left to study the case where \textbf{2)} occurs and there exist indices $i_1$ so that $E_{i_1}$ presents more than one connected component. In this setting, we would repeat the above argument for the every compact set $E_{1},\ldots, E_{N_1}$, based on the splitting given by Theorem \ref{thm4.1}. We repeat this process iteratively, that is: at step $S\geq 1$ of the iteration we would obtain (for a set of the form $E_{i_1i_2\cdots i_{S-1}}$ (where we convey $E_{i_0}:=E$), with multiple connected components consisting of parallelepipeds) a family of cubes $\big\{\QQ_{i_1i_2\cdots i_{S-1}1},\ldots, \QQ_{i_1i_2\cdots i_{S-1}N_{i_1i_2\cdots i_{S-1}}}\big\}$ with diameters comparable to $5^{-S}\text{diam}(E)$. Now we would distinguish the cases
		\begin{enumerate}
			\item[\textbf{A)}] If $E_{i_1i_2\cdots i_{S}}:=2\QQ_{i_1i_2\cdots i_S}\cap E_{i_1i_2\cdots i_{S-1}}$ has one connected component for every $i_S$.
			\item[\textbf{B)}] Or if there exists an index $i_S$ so that $E_{i_1i_2\cdots i_S}$ presents more than one connected component.
		\end{enumerate}
		As well as
		\begin{itemize}
			\item[\textbf{1)}] Assumption $\Athree$ is satisfied for $E_{i_1i_2\cdots i_{S-1}}$.
			\item[\textbf{2)}] Assumption $\Athree$ is not satisfied for $E_{i_1i_2\cdots i_{S-1}}$.
		\end{itemize}
		The possibilities \textbf{1)A)}, \textbf{1)B)} and \textbf{2)A)} are \textit{good} in the sense that lead to the estimate $\widetilde{\gamma}(E_{i_1i_2\cdots i_{S-1}})\lesssim \widetilde{\gamma}_+(E_{i_1i_2\cdots i_{S-1}})$. If on the other hand \textbf{2)} and \textbf{B)} occur simultaneously, we would move on to the next step of the iteration, by applying again Theorem \ref{thm4.1} to each $E_{i_1i_2\cdots i_{S-1}1},\ldots, E_{i_1i_2\cdots i_{S-1}N_{i_1i_2\cdots i_{S-1}}}$. The key point, however, is that for a finite number of steps $S$ large enough (depending on $\delta$) we would have $5^{-S}\text{diam}(E)\ll \delta$, so that \textbf{B)} can no longer happen, and thus we obtain the desired estimate tracing back all the steps of the iteration as in Lemma \ref{lem6.4}.
	\end{proof}
	
	\begin{rem}
		\label{rem6.2}
		Observe that constants in Lemmas \ref{lem6.4} and \ref{lem6.5} do not depend on $\ell_{\text{min}}$ in the first case, nor on $\delta>0$ in the second. Such positive quantities, however, do determine the number of steps needed to carry out the iterative argument of the proofs. But this is not an issue, since the constant in $\Athree$ is taken to be $C_1$ and it cancels with its own inverse from property $\Pthree$ in Theorem \ref{thm4.1}, eliminating possible dependence on $\ell_{\text{min}}$ or $\delta$ when tracing back each step of the iteration.
	\end{rem}
    \begin{rem}
    \label{rem6.3}
        In Lemma \ref{lem6.5} we have claimed that $\widetilde{\gamma}_+(E)\lesssim \widetilde{\gamma}_+(\mathring{E})$, for $E=Q_1\cup \cdots \cup Q_N$ a union of $N$ cubes with disjoint interiors. We will give the scheme to prove this by induction. The case $N=1$ is clear: first notice that $\widetilde{\gamma}(Q_1)\lesssim \pazocal{H}^n_{\infty}\leq \ell(Q_1)^n$. Now, pick any $0<\lambda<1$ and set $Q_1^\lambda\subset \mathring{Q_1}$ the closed cube concentric with $Q_1$ and side length $\lambda \ell(Q_1)$. Consider $\mu^\lambda:=\ell(Q_1^{\lambda})^{-1}\pazocal{L}^{n+1}|_{Q_1^{\lambda}}$ and observe that $P\ast\mu^{\lambda}\lesssim 1$ and $P^\ast\ast\mu^{\lambda} \lesssim 1$. Then, $\widetilde{\gamma}_+(Q^\lambda_1)\geq \mu^\lambda(Q_1^\lambda)=\ell(Q_1^\lambda)^n$, and making $\lambda \to 1$ we get $\widetilde{\gamma}_+(\mathring{Q}_1)\geq\ell(Q_1)^n$, so the case $N=1$ is proved. Now, we proceed by induction assuming the case $N$ true and proving the case $N+1$. To do so, we simply use the semi-additivity of $\widetilde{\gamma}_+$ (see \cite[Proposition 3.1, Theorem 3.11]{He}) followed by the induction hypothesis:
        \begin{align*}
            \widetilde{\gamma}_+(E)=\widetilde{\gamma}_+(Q_1\cup \cdots \cup Q_{N+1})&\lesssim \widetilde{\gamma}_+(Q_1\cup \cdots \cup Q_{N})+\widetilde{\gamma}_+(Q_{N+1})\\
            &\lesssim \widetilde{\gamma}_+(\mathring{Q_1}\cup \cdots \cup \mathring{Q}_{N})+\widetilde{\gamma}_+(\mathring{Q}_{N+1})\\
            &\leq 2\widetilde{\gamma}_+(\mathring{Q_1}\cup \cdots \cup \mathring{Q}_{N+1}) = 2 \widetilde{\gamma}_+(\mathring{E}).
        \end{align*}
    \end{rem}

    From the previous theorem we can finally deduce the main result of this section:

    \begin{thm}
    \label{thm6.6}
        For any compact set $E\subset \mathbb{R}^{n+1}$,
        \begin{equation*}
            \widetilde{\gamma}(E)\approx \widetilde{\gamma}_+(E).
        \end{equation*}
    \end{thm}
    \begin{proof}
        We may assume $\Atwo$ to hold by the outer regularity of $\gamma_{+}$, that by Theorem \ref{thm3.5} is a comparable capacity to $\widetilde{\gamma}_+$. Then, by a direct application of Lemma \ref{lem6.5} we are done.
    \end{proof}

    The previous result combined with the fourth statement in \cite[Proposition 3.1]{He} yields:
	\begin{thm}
		\label{thm6.7}
		The capacity $\widetilde{\gamma}$ is semi-additive. That is, there is an absolute constant $C>0$ so that for any $E_1,E_2,\ldots$ disjoint compact sets of $\mathbb{R}^{n+1}$,
		\begin{equation*}
			\widetilde{\gamma}\bigg( \bigcup_{j=1}^\infty E_j \bigg)\leq C\, \sum_{j=1}^{\infty} \widetilde{\gamma}(E_j).
		\end{equation*}
	\end{thm}
	
	\section{Further results in \mathinhead{\mathbb{R}^{2}}{}. The \mathinhead{\widetilde{\gamma}}{} capacity of rectangles}
	\label{sec7}
	
	In this section we compute the $\widetilde{\gamma}$ capacity of a closed rectangle $R\subset \mathbb{R}^2$ with sides parallel to the coordinate axes and respective side lengths $\ell_x>0,\,\ell_t>0$. More precisely:
	\begin{thm}
		\label{thm7.1}
		\begin{equation*}
			\widetilde{\gamma}(R)\approx \ell_t\, \Bigg[ \frac{1}{2}\ln\bigg( 1+\frac{\ell_t^2}{\ell_x^2}\bigg) +\frac{\ell_t}{\ell_x}\arctan\bigg( \frac{\ell_x}{\ell_t} \bigg) \Bigg]^{-1}.
		\end{equation*}
	\end{thm}
	\begin{proof}
		Let $R\subset \mathbb{R}^2$ be such a rectangle and assume, without loss of generality, that its lower left corner coincides with the origin. To simplify the computations, we also normalize its temporal side length $\ell_t$ to be 1 by dilating $R$ by the factor $\lambda:=\ell_t^{-1}$. We name the resulting rectangle $R_0$, that is such that
		\begin{equation*}
			\widetilde{\gamma}(R)=\lambda^{-1}\widetilde{\gamma}(R_0)=\ell_t\,\widetilde{\gamma}(R_0).
		\end{equation*}
		We introduce the parameter $r:=\frac{\ell_x}{\ell_t}$, that is nothing but the spatial side length of $R_0$, as well as the measure $\mu:=\pazocal{L}^{2}|_{R_0}$.
        
		By a direct computation, one obtains that the potential $P\ast \mu$ at a point $\overline{x}=(x,t)$ is given by the following explicit expression:
		\bigskip
		\begin{equation*}
			P\ast \mu(\overline{x})=
			\begin{dcases}
				\text{if } t\leq 0, \hspace{1.1cm} 0,\\
				\text{if } t\in(0,1], \hspace{0.5cm} \frac{x}{2}\ln\bigg( 1+\frac{t^2}{x^2} \bigg)+\frac{r-x}{2}\ln\bigg( 1+ \frac{t^2}{(r-x)^2} \bigg)\\
				\hspace{4.4cm}+\,t\bigg[ \frac{\pi}{2}\text{sgn}(x)-\arctan\bigg( \frac{t}{x} \bigg)+\arctan\bigg( \frac{r-x}{t} \bigg) \bigg],\\
				\text{if } t>1, \hspace{1.1cm} \frac{x}{2}\ln\bigg[ \frac{x^2+t^2}{x^2+(t-1)^2} \bigg]+\frac{r-x}{2}\ln\bigg[ \frac{(x-r)^2+t^2}{(x-r)^2+(t-1)^2} \bigg]\\
				\hspace{3cm}+\,1\bigg[ \arctan\bigg( \frac{r-x}{t-1} \bigg) + \arctan\bigg( \frac{x}{t-1} \bigg) \bigg]\\
				\hspace{3cm}+\,t\bigg[\arctan\bigg( \frac{t-1}{x} \bigg)-\arctan\bigg( \frac{t}{x} \bigg)\\
				\hspace{6cm}+\arctan\bigg( \frac{r-x}{t} \bigg)-\arctan\bigg( \frac{r-x}{t-1} \bigg) \bigg].
			\end{dcases}
		\end{equation*}
		
		In the previous formula we have written the factor $1$ in the fifth line just to emphasize that such factor would be $\ell_t$ if $R$ was not normalized to be $R_0$. It is not difficult to prove that the above expression defines a continuous function in the whole $\mathbb{R}^2$ (once extended to $x=r$ and $x=0$ when $t>0$ by taking limits). Also, using the following identity
		\begin{equation*}
			\text{sgn}(x)=\frac{2}{\pi}\bigg[ \arctan(x)+\arctan\bigg( \frac{1}{x} \bigg) \bigg],
		\end{equation*}
		it follows that for any fixed $t$, $P\ast \mu(x,t)=P\ast \mu(r-x,t),$ or in other words, the one variable function $x\mapsto P\ast \mu(x,t)$ is symmetric with respect to the point $x=r/2$, for any $t\in\mathbb{R}$. Moreover, it is clear that it is nonnegative, tends to 0 as $x\to \pm\infty$ and in fact, for $t>0$, it attains its maximum precisely at $x=r/2$. This last property can be argued as follows: begin by fixing $t\in(0,1]$ and noticing that
		\begin{equation*}
			\lim_{x\to 0} P\ast \mu(x,t) = \lim_{x\to r} P\ast \mu(x,t) = \frac{r}{2}\ln\bigg( 1+\frac{t^2}{r^2} \bigg)+t\,\arctan\bigg( \frac{r}{t} \bigg)=:C_t.
		\end{equation*}
		Compute the derivative with respect to $x$ of $P\ast \mu$ at points $x\neq 0, x\neq r$, that is given by
		\begin{equation*}
			\partial_xP\ast \mu(x,t)'=\frac{1}{2}\Bigg[\ln\bigg( 1+\frac{t^2}{x^2} \bigg) - \ln\bigg( 1+\frac{t^2}{(r-x)^2} \bigg)  \Bigg],
		\end{equation*}
		and that satisfies $\partial_xP\ast \mu(\cdot,t)<0$ in $(r/2,r)\cup(r,\infty)$, $\partial_xP\ast \mu(\cdot,t)>0$ in $(-\infty,0)\cup(0,r/2)$ and $\partial_xP\ast \mu(r/2,t)=0$. Therefore, $P\ast \mu(\cdot,t)$ may attain its global maximum at $x=0$, $x=r/2$ or $x=r$. The value of $P\ast \mu(\cdot,t)$ at $x=r/2$ is
		\begin{align*}
			P\ast \mu\bigg( \frac{r}{2},t\bigg) =\frac{r}{2}\ln\bigg( 1+\frac{4t^2}{r^2} \bigg)+2t\,\arctan\bigg( \frac{r}{2t} \bigg)>C_t.
		\end{align*}
		Hence, for $t\in (0,1)$, $P\ast \mu(\cdot,t)$ attains its global maximum at $x=r/2$. For $t>1$ a similar study can be carried out, yielding the same conclusion.
        
		Now, by restricting $P\ast \mu$ to the vertical line $x=r/2$ we obtain a nonnegative piece-wise continuous function of $t$ given by the expression
		\bigskip
		\begin{equation*}
			P\ast \mu\bigg( \frac{r}{2},t\bigg)=
			\begin{dcases}
				\text{if } t\leq 0, \hspace{1.1cm} 0,\\
				\text{if } t\in(0,1], \hspace{0.5cm} \frac{r}{2}\ln\bigg( 1+\frac{4t^2}{r^2} \bigg)+2t\,\arctan\bigg( \frac{r}{2t} \bigg),\\
				\text{if } t>1, \hspace{1.1cm} \frac{r}{2}\ln\bigg[ \frac{r^2+4t^2}{r^2+4(t-1)^2} \bigg]+\,2\arctan\bigg( \frac{r}{2(t-1)} \bigg)\\
				\hspace{4cm}-\,2t\bigg[\frac{\pi}{2}-\arctan\bigg( \frac{2(t-1)}{r} \bigg)-\arctan\bigg( \frac{r}{2t} \bigg)\Bigg].
			\end{dcases}
		\end{equation*}
		which also tends to 0 as $t\to \pm \infty$ and it can be proved that it attains its maximum for $t=1$, for any value of $r>0$. Combining the above computations, we have obtained
		\begin{equation*}
			P\ast \mu(\overline{x})\leq P\ast \mu(r/2,1), \hspace{0.5cm} \forall \overline{x}\in \mathbb{R}^2.
		\end{equation*}
		For the sake of clarity, Figure \ref{fig2} depicts the graph of the potential $P\ast \mu$ for a particular value of $r$ (that already represents its qualitative behavior).
        \begin{figure}[b]
			\centering
			\includegraphics[width=0.85\textwidth]{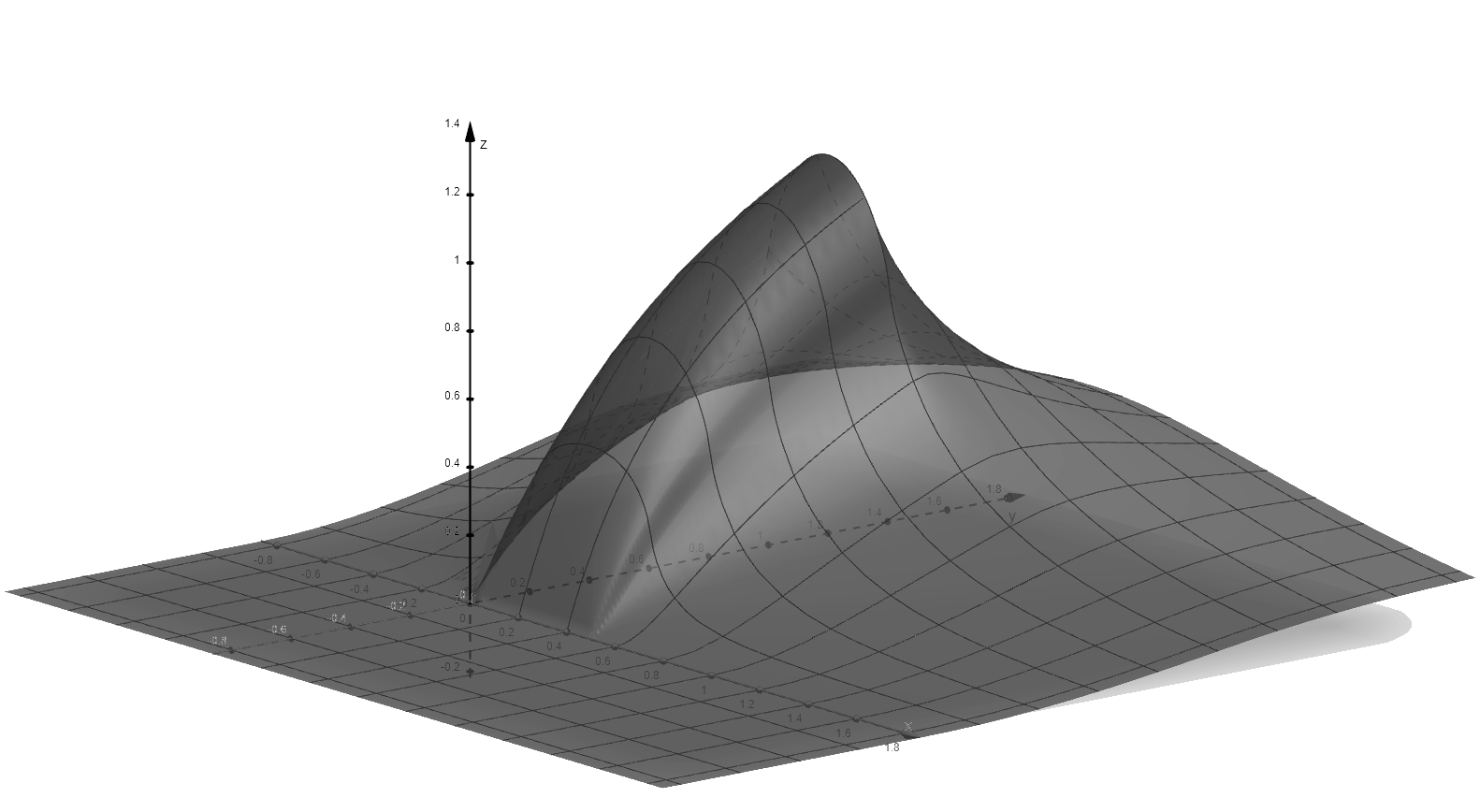}
			\caption{Graph of $P\ast \mu$ for $r=1/2$.}
			\label{fig2}
		\end{figure}
        
		The particular value of $M(r):=P\ast \mu(r/2,1)$ provides a lower bound for $\gamma_{+}$:
		\begin{equation*}
			\gamma_{+}(R_0)\geq M(r)^{-1}\mu(R_0)=M(r)^{-1}r = \Bigg[ \frac{1}{2}\ln\bigg( 1+\frac{4}{r^2}\bigg) +\frac{2}{r}\arctan\bigg( \frac{r}{2}\bigg) \Bigg]^{-1}.
		\end{equation*}
		To obtain an upper bound we shall work with the capacity $\gamma_{\text{sy},+}$ (comparable to $\gamma_{+}$ by Theorem \ref{thm3.5}). By definition of $P_{\text{sy}}$ we get
		\begin{align*}
			P_{\text{sy}}\ast \mu(x,t) := \frac{1}{2}\big( P\ast \mu(x,t) + P^\ast \ast \mu(x,t) \big) = \frac{1}{2}\big( P\ast \mu(x,t) + P \ast \mu(x,1-t) \big),
		\end{align*}
		where the last equality can be easily deduced from the definition of $P^\ast$ and the symmetry of $R$ with respect to the horizontal line $t=1/2$. By a similar study to the one done for $P\ast \mu$, it can be proved that $(P_{\text{sy}}\ast \mu)|_{R_0}$ attains its minimum at the vertices of $R_0$. Just proceed by fixing $t\in[0,1]$ and studying the one variable function $x\mapsto P_{\text{sy}}\ast \mu(x,t)$ restricted to the domain $[0,r]$; and by fixing $x\in[0,1]$ and studying $t\mapsto P_{\text{sy}}\ast \mu(x,t)$ once restricted to $[0,1]$. The former is again symmetric with respect to the point $x=r/2$ and attains its minimum for $x=0$ and $x=r$, while the latter is symmetric with respect $t=1/2$ and attains its minimum for $t=0$ and $t=1$. In Figure \ref{fig3} the reader may visualize the graph of $P_{\text{sy}}\ast \mu$ for a particular value of $r$.\medskip\\
		Therefore, for any $(x,t)\in R_0$ we have
		\begin{align*}
			P_{\text{sy}}\ast \mu (x,t) \geq \lim_{(x,t)\to (0,0)} P_{\text{sy}}\ast \mu (x,t) = \frac{r}{4}\ln\bigg( 1+\frac{1}{r^2} \bigg) + \frac{1}{2}\arctan(r)=:\frac{1}{2}m(r).
		\end{align*}
		Now take any admissible measure $\nu$ for $\gamma_{\text{sy},+}(R_0)$ and observe that
		\begin{equation*}
			\langle \nu, 1 \rangle \leq 2\,m(r)^{-1} \langle \nu, P_{\text{sy}}\ast \mu \rangle =  2\,m(r)^{-1} \langle P_{\text{sy}}\ast\nu, \mu \rangle \leq 2\,m(r)^{-1}\mu(R_0)=2\,m(r)^{-1}r,
		\end{equation*}
		\begin{figure}[b]
			\centering
			\includegraphics[width=0.85\textwidth]{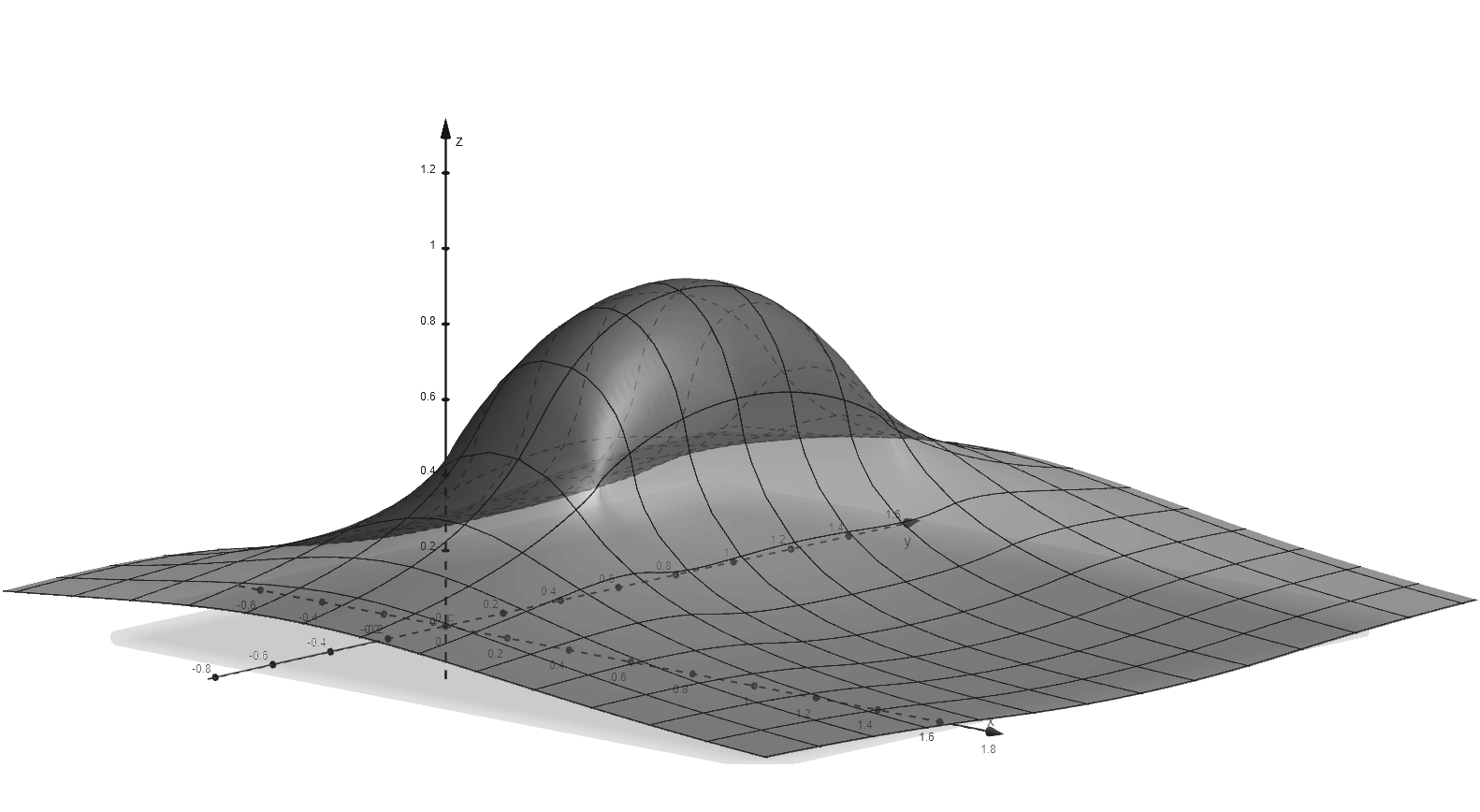}
			\caption{Graph of $P_{\text{sy}}\ast \mu$ for $r=1/2$.}
			\label{fig3}
		\end{figure}
		where we have applied Tonelli's theorem, the symmetry of $P_{\text{sy}}$ and the fact that $\mu \ll \pazocal{L}^2$. So by the arbitrariness of $\nu$ and the comparability of $\gamma_{+}$ with $\gamma_{\text{sy},+}$ we deduce that there exists an absolute constant $C>0$ so that
		\begin{equation*}
			\gamma_{+}(R_0)\leq C\,m(r)^{-1}r=C\Bigg[ \frac{1}{2}\ln\bigg( 1+\frac{1}{r^2}\bigg) +\frac{\arctan(r)}{r} \Bigg]^{-1}.
		\end{equation*}
		Finally, using that $M(r)-4m(r)\leq 0$ for any $r>0$, we get the following estimate
		\begin{equation*}
			\gamma_{+}(R_0)\approx m(r)^{-1}r= \Bigg[ \frac{1}{2}\ln\bigg( 1+\frac{1}{r^2}\bigg) +\frac{\arctan(r)}{r}\Bigg]^{-1}.
		\end{equation*}
		Therefore, regarding the original rectangle $R$ and applying Theorem \ref{thm3.5} and Lemma \ref{lem6.4}, Theorem \ref{thm7.1} follows.
	\end{proof}
	
	\begin{rem}
		\label{rem7.1}
		Let us check how the above relation extends to cases $\ell_x=0$ (i.e. $r=0$) and $\ell_t=0$ (i.e. $r=+\infty$). For the first case, notice that if $r\leq 1/2$ we have
		\begin{align*}
			\frac{1}{3|\ln{(r)}|} \leq \Bigg[ \frac{1}{2}\ln\bigg( 1+\frac{1}{r^2}\bigg) +\frac{\arctan(r)}{r}\Bigg]^{-1}\leq \frac{1}{|\ln{(r)}|} ,
		\end{align*}
		where the hypothesis of $r\leq 1/2$ is used in the first bound. Therefore, we deduce
		\begin{equation*}
			\widetilde{\gamma}(R) \approx \ell_t\,\bigg\rvert \ln\bigg(\frac{\ell_x}{\ell_t}\bigg) \bigg\rvert^{-1}, \hspace{0.5cm} \text{if} \quad \ell_x\leq \ell_t/2,
		\end{equation*}
		which is a result consistent with the outer regularity of $\gamma_+$ and the fact that a vertical line segment has null $\gamma$ capacity (see \cite[Proposition 6.1]{MPr}). On the other hand, for the regime $r\to +\infty$, the following holds if $r\geq 1$
		\begin{align*}
			\frac{r}{2} \leq \Bigg[ \frac{1}{2}\ln\bigg( 1+\frac{1}{r^2}\bigg) +\frac{\arctan(r)}{r}\Bigg]^{-1}\leq r,
		\end{align*}
		(in fact, the lower bound holds for any $r>0$) so in this case we deduce
		\begin{equation*}
			\widetilde{\gamma}(R) \approx \ell_x, \hspace{0.5cm} \text{if} \quad \ell_x\geq \ell_t,
		\end{equation*}
		that is what we expected by point \textit{2} in Corollary \ref{cor6.3}.
	\end{rem}

	\vspace{1.5cm}
	{\small
		\begin{tabular}{@{}l}
			\textsc{Joan\ Hernández,} \\ \textsc{Departament de Matem\`{a}tiques, Universitat Aut\`{o}noma de Barcelona,}\\
			\textsc{08193, Bellaterra (Barcelona), Catalonia.}\\
			{\it E-mail address}\,: \href{mailto:joan.hernandez@uab.cat}{\tt{joan.hernandez@uab.cat}}
		\end{tabular}
	}
    
	\bigskip
	\noindent
	{\small
		\begin{tabular}{@{}l}
			\textsc{Joan\ Mateu,} \\ \textsc{Departament de Matem\`{a}tiques, Universitat Aut\`{o}noma de Barcelona,}\\
			\textsc{08193, Bellaterra (Barcelona), Catalonia.}\\
			{\it E-mail address}\,: \href{mailto:joan.mateu@uab.cat}{\tt{joan.mateu@uab.cat}}
		\end{tabular}
	}
    
	\bigskip
	\noindent
	{\small
		\begin{tabular}{@{}l}
			\textsc{Laura\ Prat,} \\ \textsc{Departament de Matem\`{a}tiques, Universitat Aut\`{o}noma de Barcelona,}\\
			\textsc{08193, Bellaterra (Barcelona), Catalonia.}\\
			{\it E-mail address}\,: \href{mailto:laura.prat@uab.cat}{\tt{laura.prat@uab.cat}}
		\end{tabular}
	}  

\bigskip
	\bigskip

	\begin{thebibliography}{CMM+2}
		
		\bibitem[AbS]{AbS} Abramowitz, M. \& Stegun, I. A. (1983). \textit{Handbook of Mathematical Functions with Formulas, Graphs, and Mathematical Tables}. Applied Mathematics Series. Vol. 55. New York: United States Department of Commerce, National Bureau of Standards; Dover Publications.
		
		
		
		\bibitem[Ba]{Ba} Baire, R. (1905). \textit{Leçons sur les fonctions discontinues professées au Collège de France}. Paris: Gauthier-Villars.
		
		\bibitem[BlGe]{BlGe} Blumenthal, R. M. \& Getoor, R. K. (1960). Some theorems on stable processes. \textit{Transactions of the American Mathematical Society}, \textit{95} (\textit{2}), pp. 263-273.
		
		\bibitem[Bu]{Bu} Buchholz, H. (1969). \textit{The Confluent Hypergeometric Function}. Berlin, Heidelberg: Springer-Verlag.
		
		\bibitem[Ca]{Ca} Carlsson, H (2011). \textit{Lecture notes on distributions}. Chalmers University of Technology, Department of Mathematics.
		
		\bibitem[Ch]{Ch} Christ, M. (1990). \textit{Lectures on Singular Integral Operators}. CBMS Regional Conference Series in Mathematics, Number 77. Providence, Rhode Island: American Mathematical Society.
		
		\bibitem[DPV]{DPV} Di Nezza, E., Palatucci, G. \& Valdinoci, E. (2012). Hitchhiker's guide to the fractional Sobolev spaces. \textit{Bulletin of Mathematical Sciences}, \textit{5}(\textit{136}), pp. 521-573.
		
		\bibitem[Du]{Du} Duoandikoetxea, J. (2001). \textit{Fourier analysis}. (Vol. 29). American Mathematical Society.
		
		\bibitem[He]{He} Hernández, J. (2024). On the $(1/2,+)$-caloric capacity of Cantor sets. \textit{Annales Fennici Mathematici}, \textit{49}(\textit{1}), pp. 211–239. \url{https://doi.org/10.54330/afm.144428}.
		
		
		
		\bibitem[HPo]{HPo} Harvey, R. \& Polking, J. (1970). Removable singularities of solutions of linear partial differential equations. \textit{Acta Mathematica}, \textit{125}, pp. 39-56.
		
		
		
		
		\bibitem[HyMar]{HyMar} Hytönen, T. \& Martikainen, H. (2012). Non-homogeneous $Tb$ theorem and random dyadic cubes on metric measure spaces. \textit{Journal of Geometric Analysis}, \textit{22} (\textit{4}), pp. 1071-1107.
		
		\bibitem[Ki]{Ki} Kishi, M. (1963). Maximum Principles in the Potential Theory. \textit{Nagoya Mathematical Journal}, \textit{23}, pp. 165-187. \url{https://doi.org/10.1017/S0027763000011247}.
		
		\bibitem[La]{La} Landkof, N. S. (1972). \textit{Foundations of Modern Potential Theory}. Berlin: Springer-Verlag.
		
		\bibitem[MPrTo]{MPrTo} Mateu, J., Prat, L. \& Tolsa, X. (2022). Removable singularities for Lipschitz caloric functions in time varying domains. \textit{Revista Matemática Iberoamericana}, \textit{38}(\textit{2}), pp. 547-588.
		
		\bibitem[MPr]{MPr} Mateu, J. \& Prat, L. (2024). Removable singularities for solutions of the fractional heat equation in time varying domains. \textit{Potential Analysis}, \textit{60}, pp. 833-873.
		
		
		\bibitem[Ma]{Ma} Mattila, P. (1995). \textit{Geometry of Sets and Measures in Euclidean Spaces: Fractals and Rectifiability.} Cambridge: Cambridge University Press.
		
		\bibitem[MaP]{MaP} Mattila, P. \& Paramonov P.V. (1995). On geometric properties of harmonic $\text{Lip}_1$-capacity. \textit{Pacific Journal of Mathematics}, \textit{2}, pp. 469-491.
		
		
		\bibitem[NTVo1]{NTVo1} Nazarov, F., Treil, S. \& Volberg A. (2002). The $Tb$-theorem on non-homogeneous spaces that proves a conjecture of Vitushkin. \textit{CRM preprint}, \textit{519}(\textit{2}), pp. 1-84.
		
		\bibitem[NTVo2]{NTVo2} Nazarov, F., Treil, S. \& Volberg A. (2003). The $Tb$-theorem on non-homogeneous spaces. \textit{Acta Mathematica}, \textit{190}(\textit{2}), pp. 151-239.
		
		\bibitem[Ra]{Ra} Ransford, T. (1995). \textit{Potential Theory in the Complex Plane} (London Mathematical Society, Student Texts 28). Cambridge: Cambridge University Press.
		
		\bibitem[St]{St} Stein, E. M. (1970). \textit{Singular integrals and differentiability properties of functions}. Princeton: Princeton University Press.
		
		
		\bibitem[To1]{To1} Tolsa, X. (2001). Littlewood–Paley Theory and the $T(1)$ Theorem with Non-doubling Measures. \textit{Advances in Mathematics}, \textit{164}(\textit{1}), pp. 57-116.

        \bibitem[To2]{To2} Tolsa, X. (2004). The semiadditivity of continuous analytic capacity and the inner boundary conjecture. \textit{American Journal of Mathematics}, \textit{126}, (\textit{3}), pp. 523-567.
		
		
		\bibitem[To3]{To3} Tolsa, X. (2014). \textit{Analytic Capacity, the Cauchy Transform, and Non-homogeneous Calderón-Zygmund Theory.} Birkhäuser, Cham.
		
		\bibitem[Va]{Va} Vázquez, J. L. (2018). Asymptotic behaviour for the fractional heat equation in the Euclidean space. \textit{Complex Variables and Elliptic Equations}, \textit{63}(\textit{7-8}), pp. 1216-1231.
		
		\bibitem[Ve]{Ve} Verdera, J. (2007). Classical potential theory and analytic capacity. \textit{Advanced courses of Mathematical Analysis II, World Scientific A}, pp. 174-192.
		
		\bibitem[Vo]{Vo} Volberg, A. (2003). \textit{Calderón-Zygmund Capacities and Operators on Nonhomogeneous Spaces}. Volume 100, CBMS Regional Conference Series in Mathematics. Providence, Rhode Island: American Mathematical Society.
		
	\end{thebibliography}
\end{document}